\documentclass[11pt]{article}
\usepackage[letterpaper, margin=1in]{geometry}

\usepackage{amscd,amsfonts,amsmath,amssymb,amsthm,array,enumitem,fancyhdr,graphicx,multicol,scalerel,spectralsequences,subcaption,svg,tabularx,thmtools,thm-restate,tikz,tikz-cd}
\usepackage{xcolor}
\usepackage[colorlinks,
    linkcolor={violet},
    citecolor={olive!80!orange},
    urlcolor={blue!80!black}]{hyperref}
\usepackage[alphabetic]{amsrefs}

\theoremstyle{plain}
\newtheorem{theorem}{Theorem}[section]
\newtheorem{corollary}[theorem]{Corollary}
\newtheorem{lemma}[theorem]{Lemma}
\newtheorem{proposition}[theorem]{Proposition}

\newtheorem*{theorem*}{Theorem}

\newtheorem*{cconj}{Concordance Conjecture}
\newtheorem*{acconj}{Almost-Concordance Conjecture}
\newtheorem*{lightbulb}{Concordance Lightbulb Theorem}

\theoremstyle{definition}
\newtheorem{definition}[theorem]{Definition}
\newtheorem*{definition*}{Definition}

\newtheorem{example}[theorem]{Example}

\theoremstyle{remark}

\newcommand{\N}{\mathbb{N}}
\newcommand{\Z}{\mathbb{Z}}

\newcommand{\upst}{^{\text{st}}}
\newcommand{\upth}{^{\text{th}}}

\newcommand{\Aut}{\operatorname{Aut}}
\newcommand{\id}{\operatorname{id}}
\newcommand{\im}{\operatorname{im}}
\newcommand{\coker}{\operatorname{coker}}
\newcommand{\Hom}{\operatorname{Hom}}

\newcommand{\ol}[1]{\overline{#1}}
\newcommand{\wh}[1]{\widehat{#1}}
\newcommand{\wt}[1]{\widetilde{#1}}

\DeclareMathOperator*{\plus}{\scalerel*{+}{\sum}}

\makeatletter
\newcommand{\colim@}[2]{%
  \vtop{\m@th\ialign{##\cr
    \hfil$#1\operator@font colim$\hfil\cr
    \noalign{\nointerlineskip\kern1.5\ex@}#2\cr
    \noalign{\nointerlineskip\kern-\ex@}\cr}}%
}
\newcommand{\colim}{%
  \mathop{\mathpalette\colim@{\rightarrowfill@\scriptscriptstyle}}\nmlimits@
}
\makeatother

\title{Almost-concordance of knots in aspherical 3-manifolds}
\author{Ryan Stees}
\date{}

\begin{document}

\maketitle


\begin{abstract}
In this paper, we study topological concordance modulo local knotting, or \emph{almost-concordance}, of knots in 3-manifolds $M\neq S^3$. 
A. Levine, Celoria \cite{Celoria}, and Friedl-Nagel-Orson-Powell \cite{FNOP} conjecture that, absent the presence of an embedded dual 2-sphere, any free homotopy class $x$ of knots in $M$ contains infinitely many concordance classes modulo the action of the concordance group of knots in $S^3$ by local knotting. 
We develop a method for confirming this conjecture for any nontrivial class $x$ in any aspherical $M$ and provide computations that prove the conjecture in a large family of open cases.
Our technique employs an extension of Milnor's link invariants to knots and links in non-simply-connected 3-manifolds \cite{Stees24}.
We exhibit a large family of examples where, in a precise sense, we maximize the number of almost-concordance classes distinguished by these invariants.
\end{abstract}


\section{Introduction} \label{sec:intro}
The most familiar setting for studying (topological) concordance is the concordance group $\mathcal{C}$ of knots in $S^3$.
While the overall structure of this group remains mysterious, much is known; see \cites{LevineJ69a, LevineJ69b, CassonGordon,COT}. 
In particular, Cochran-Orr-Teichner \cite{COT} developed a filtration on $\mathcal{C}$, half of whose successive quotients are infinitely generated \cite{CHL}. 
In contrast, relatively little is known about concordance of knots in closed orientable 3-manifolds $M\neq S^3$.

The study of knot concordance in such 3-manifolds $M$ differs significantly with the classical $S^3$ setting. 
The set $\mathcal{C}^M$ of concordance classes of knots in $M$ does not in general admit a group structure.
Indeed, the connected sum operation does not preserve the 3-manifold $M$, as $M\#M\cong M$ if and only if $M\cong S^3$. 
Additionally, a non-simply-connected $M$ allows for embedding phenomena which cannot occur in $S^3$. 
In such $M$ we do not have an analogue of the unknot or unlink in the sense that there is no natural choice of knot or link to which all others can be reasonably compared. 
For instance, knots which are concordant must belong to the same free homotopy class in $[S^1,M]$. 
A free homotopy between concordant knots is given by the composition of the concordance with projection to $M$:\[S^1\times[0,1]\hookrightarrow M\times[0,1]\xrightarrow{\text{proj}} M.\]
Letting $\mathcal{C}^M_x$ denote the set of concordance classes of knots in $M$ representing $x\in[S^1,M]$, we obtain the decomposition \[\mathcal{C}^M=\bigsqcup_{x\in[S^1,M]}\mathcal{C}^M_x.\]

In contrast with the aforementioned results about $\mathcal{C}$, whether or not $\mathcal{C}^M_x$ is an \emph{infinite set} is open in general. 
One well-understood case is when the class $x$ admits an embedded dual 2-sphere.

\begin{lightbulb}[\cites{DNPR,Yildiz,FNOP}] \label{thm:lightbulb}
If $x\in[S^1,M]$ admits an embedded dual 2-sphere, then $\mathcal{C}^M_x$ contains a single element; that is, all knots representing $x$ are concordant. Moreover, this statement holds in the smooth category.
\end{lightbulb}

\noindent Celoria, based on a question of A. Levine, and Friedl-Nagel-Orson-Powell have conjectured that the setting of the \hyperref[thm:lightbulb]{Concordance Lightbulb Theorem} is the exceptional case.

\begin{cconj}[\cites{Celoria,FNOP}] \label{conj:c}
The concordance set $\mathcal{C}^M_x$ contains infinitely many elements if and only if the class $x$ does not admit an embedded dual 2-sphere.
\end{cconj}

While $\mathcal{C}^M$ does not admit a group structure, there is an action of the concordance group $\mathcal{C}$ on $\mathcal{C}^M$ by \emph{local knotting}; see Figure~\ref{fig:local}.
\begin{figure}
    \centering
    \begin{tikzpicture}
        \node[anchor=south west,inner sep=0] (image) at (0,0) {\includesvg[width=0.5\textwidth]{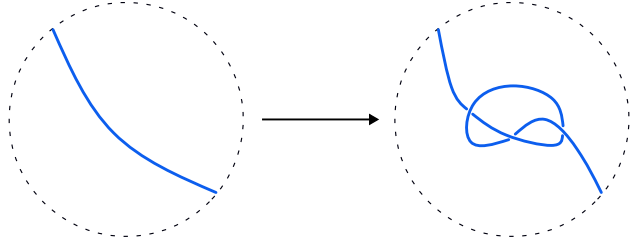}};
    \end{tikzpicture}
    \caption{Locally knotting $K\subset M$ by the right-handed trefoil.}
    \label{fig:local}
\end{figure}
Given knots $J\subset S^3$ and $K\subset M$, we define the action of $J$ on $K$ using a connected sum of pairs: \[(M,J\cdot K)=(S^3,J)\cdot(M,K)=(S^3,J)\,\#\,(M,K).\] 
This action descends to concordance classes, yielding an action of $\mathcal{C}$ on $\mathcal{C}^M$. 
One might hope that a positive answer to the \hyperref[conj:c]{Concordance Conjecture} could, in some cases, be achieved by locally knotting some $K\subset M$ representing $x$ by each of a family of nonconcordant knots in $S^3$. 
Instead, we will consider when it is possible to show $\mathcal{C}^M_x$ is infinite \emph{without} leveraging the size of the concordance group $\mathcal{C}$. 
To this end, we consider the orbit set of $\mathcal{C}^M$ under the action of local knotting.

Define $\wt{\mathcal{C}}^M=\mathcal{C}^M/\mathcal{C}$, the set of concordance classes of knots in $M$ modulo local knotting. 
While the study of concordance modulo local knotting traces its roots to work of Cappell-Shaneson \cite{CappellShaneson}, Celoria recently coined the term \emph{almost-concordance} to mean concordance modulo local knotting. 
We will thus refer to $\wt{\mathcal{C}}^M$ as the set of almost-concordance classes of knots in $M$. 
Note that we have an analogous decomposition \[\wt{\mathcal{C}}^M=\bigsqcup_{x\in[S^1,M]}\wt{\mathcal{C}}^M_x,\] where $\wt{\mathcal{C}}^M_x$ consists of almost-concordance classes representing the free homomtopy class $x$. 
We consider the following strengthening of the \hyperref[conj:c]{Concordance Conjecture}.

\begin{acconj}[\cites{Celoria,FNOP}] \label{conj:ac}
The almost-concordance set $\wt{\mathcal{C}}^M_x$ contains infinitely many elements if and only if the class $x$ does not admit an embedded dual 2-sphere.
\end{acconj}

Partial progress toward this conjecture has been made in \cites{MillerD,Schneiderman,Yildiz,FNOP,NOPP,Kuzbary}; we briefly summarize what is known. 
Observe that the \hyperref[thm:lightbulb]{Concordance Lightbulb Theorem} \cites{DNPR,Yildiz,FNOP} proves one direction of the \hyperref[conj:ac]{Almost-Concordance Conjecture}. 
In the opposite direction, the \hyperref[conj:ac]{Almost-Concordance Conjecture} holds when:
\begin{itemize}
\item $M$ is Seifert-fibered and $x$ is the class of the Seifert fiber \cite{MillerD}.
\item $x$ is nullhomotopic \cites{Schneiderman,Yildiz,FNOP}.
\item $M$ is a lens space and $x$ is any element of $[S^1,M]$ \cite{NOPP}.
\item $x$ is torsion with $\pi_1(M)/\langle\langle x\rangle\rangle\neq 1$ \cite{FNOP}.
\item $[x]=2u\in H_1(M)$ for some primitive class $u$ of infinite order \cite{FNOP}.
\item $[x]$ is represented by a separating curve on a non-separating embedded oriented surface in $M$ \cite{FNOP}.
\item $M=\#^k(S^1\times S^2)$ and $[x]=0\in H_1(M)$ \cite{Kuzbary}.
\end{itemize}

Some of the above results combine to show that the reverse direction of the \hyperref[conj:ac]{Almost-Concordance Conjecture} is known for most $x$ in spherical $M$ and for the trivial class in any $M$. 
Thus, the vast majority of remaining open cases in prime, irreducible 3-manifolds are nontrivial classes $x$ in aspherical $M$. 
Any such class has infinite order when represented as an element of $\pi_1(M)$. 
In this paper, we develop a method to prove the \hyperref[conj:ac]{Almost-Concordance Conjecture} for \emph{any} nontrivial $x$ in \emph{any} aspherical $M$. 
We then provide computations that prove the conjecture in a large class of open cases.


\subsection{Overview of technique and results} \label{subsec:intro:results}

We will distinguish knots up to almost-concordance using the \emph{lower central homology invariants} $\theta_n$ from \cite{Stees24}, an extension of Milnor's $\ol{\mu}$-invariants for links in $S^3$ \cite{Milnor}.
We employ the $\theta$-invariants to obstruct almost-concordance in the following main theorem, which we prove in Section~\ref{subsec:main:family}.

\begin{restatable*}{theorem}{family} \label{thm:family}
Given any knot $K$ in any aspherical $M$ representing any nontrivial $x\in[S^1,M]$, there exists a family of knots $\{K_\alpha^n\}=\cup_{n\geq 2}\{K_\alpha^n\}_{\alpha\in \ol{C}_n}$ such that
\begin{itemize}
\item For $m\neq n$, we have $\{K_\alpha^m\}\cap\{K^n_\beta\}=\{K\}$.
\item Each $K_\alpha^n$ represents $x\in[S^1,M]$.
\item For each fixed $n\geq 2$, elements of $\{K_\alpha^n\}$ are in bijective correspondence with elements of \[\ol{C}_n=\frac{\im\left(H_3(\Gamma/\Gamma_n)\to H_3(\pi/\Gamma_n)\right)}{\im\left(H_3(\Gamma/\Gamma_{n+1})\to H_3(\pi/\Gamma_n)\right)},\] where $\pi$ is the knot group of $K$, $\Gamma=\ker\left(\pi\twoheadrightarrow\pi_1(M)\right)$, and the homomorphisms on $H_3(-)$ are induced by the canonical group homomorphisms $\Gamma/\Gamma_{n+1}\twoheadrightarrow\Gamma/\Gamma_n\hookrightarrow\pi/\Gamma_n$. 
In particular, the fixed knot $K$ is $K_0^n$ for all $n$.
\item For distinct $K_\alpha^m$ and $K_\beta^n$, we have $[K_\alpha^m]\neq[K_\beta^n]\in\wt{\mathcal{C}}^M_x$, that is, $K_\alpha^m$ and $K_\beta^n$ are not almost-concordant.
\end{itemize}
\end{restatable*}

The $\theta$-invariants are concerned with the $\pi_1(M)$-lower central series of the \emph{knot group} $\pi=\pi_1(E_K)$, where $E_K$ is the exterior of $K\subset M$.
Let $\Gamma$ be the kernel of the inclusion-induced homomorphism $\pi\twoheadrightarrow\pi_1(M)$.
Then $\Gamma$ is normally generated by a meridian of $K$.
The \emph{$\pi_1(M)$-lower central series} of $\pi$ is \[\cdots\trianglelefteq\Gamma_{n+1}\trianglelefteq\Gamma_n\trianglelefteq\cdots\trianglelefteq\Gamma_2\trianglelefteq\Gamma\trianglelefteq\pi,\] where $\Gamma_2=[\Gamma,\Gamma]$ is the commutator subgroup of $\Gamma$ and $\Gamma_{n+1}=[\Gamma,\Gamma_n]$ in general.
In other words, the $\pi_1(M)$-lower central series of $\pi$ consists of the group $\pi$ along with the (classical) lower central series of the subgroup $\Gamma$.
Note that $\Gamma_n$ is normal in $\pi$ for all $n$.
We will call the quotients $\pi/\Gamma_n$ the \emph{$\pi_1(M)$-lower central quotients} of $\pi$.

Cappell-Shaneson first observed that the exterior of a concordance between two knots $K$ and $K'$ in $M\times[0,1]$ is a homology cobordism rel boundary with coefficients in $\Z[\pi_1(M)]$ \cite{CappellShaneson}.
This observation, along with an application of the Stallings-Dwyer Theorem \cites{Stallings,Dwyer}, shows that the $\pi_1(M)$-lower central quotients $\pi/\Gamma_n$, along with boundary data, are concordance invariants.
In fact, because local knotting does not affect the $\Z[\pi_1(M)]$-homology type of the knot exterior, these quotients are almost-concordance invariants.

The $\theta$-invariants obstruct the existence of peripheral-data-preserving isomorphisms between the $\pi_1(M)$-lower central quotients of knot groups.
For a fixed knot $K\subset M$ and some other knot $K'\subset M$, the $n\upth$ lower central homology invariant $\theta_n(K')$ of $K'$ relative to $K$ takes values in $H_3(X_n(K))/\Aut(\pi/\Gamma_n,\partial)$, where $X_n(K)$ is a space which depends on $K$ and $\Aut(\pi/\Gamma_n,\partial)$ is some group acting on $X_n(K)$ by homotopy self-equivalences.
We provide details in Section~\ref{sec:Milnor}.

These invariants are well-adapted to proving broad statements about the \hyperref[conj:ac]{Almost-Concordance Conjecture} because they are defined for knots (in fact, links) in any closed orientable 3-manifold, they are almost-concordance invariants, as we prove in Section~\ref{subsec:Milnor:localknot}, and they can be related to well-understood invariants of links in $S^3$, as we see in Section~\ref{subsec:construct:S3}.
More specifically, although we do not have a group structure on the almost-concordance set $\mathcal{C}^M_x$, the difference of the invariants $\theta_n(K')-\theta_n(K)$ is well-defined and sometimes has a useful interpretation in terms of Milnor's invariants of links in $S^3$.

Given any knot $K$ representing some nontrivial $x\in[S^1,M]$, we produce the family $\{K^n_\alpha\}$ using a family of associated Brunnian links $L^n_\alpha\subset S^3$.
The Brunnian link $L^n_\alpha$, together with a collection of (based) embedded curves in the exterior of $K$, gives instructions for modifying $K$ to produce $K^n_\alpha$ via an ambient connected sum.
We call this procedure an \emph{ambient connected sum with a Brunnian link}.
An example is illustrated in Figure~\ref{fig:whitehead}.
Roughly speaking, the difference between the exteriors of $K^n_\alpha$ and $K$ in $M$ is the exterior of the Brunnian link $L^n_\alpha$ in $S^3$.
This statement is made precise in Proposition~\ref{prop:difference}, where we show that, for each $n\geq 2$ and $\alpha\in\ol{C}_n$, $\theta_n(K^n_\alpha)-\theta_n(K)$ is the image of Orr's invariant $\theta_n^O(L^n_\alpha)$ \cite{Orr89} under some suitable homomorphism which relates the link group of $L^n_\alpha$ to the knot group of $K$.

Theorem~\ref{thm:family} answers the \hyperref[conj:ac]{Almost-Concordance Conjecture} for the class $x$ in the 3-manifold $M$ if the set $\cup_{n\geq 2}(\ol{C}_n-\{0\})$ is infinite.
In Section~\ref{subsec:main:large}, we analyze the homomorphism $H_3(\Gamma/\Gamma_n)\to H_3(\pi/\Gamma_n)$ and provide conditions which guarantee that the family of knots $\{K_\alpha^n\}$ is, in some sense, as large as possible.
These conditions arise from our analysis of the $\Z[\pi_1(M)]$-module structure on $H_*(\Gamma/\Gamma_n)$ in Section~\ref{sec:zpi}.

\begin{restatable*}{theorem}{infinite} \label{thm:infinite}
Suppose that the following hold for the aspherical 3-manifold $M$ and the homotopically essential knot $K\subset M$ (for some basing of $K$):
\begin{enumerate}[label=(\arabic*)]
\item $[K]$ is primitive and infinite order in $H_1(M)$.
\item The centralizer of any nontrivial power of $[K]\in\pi_1(M)$ is cyclic.
\item The left cosets of $\langle[K]\rangle\leq\pi_1(M)$ admit a total ordering which is invariant under the $\pi_1(M)$-action by left multiplication.
\end{enumerate}
Let $\mathcal{M}(k)$ denote the number of linearly independent Milnor invariants of $k$-component links of length $k$ up to relabeling components. For each $n\geq 2$, the group $\ol{C}_n$ is free abelian of rank at least ${\mathcal{M}(n+1)}$.
\end{restatable*}

\noindent In Section~\ref{subsec:main:large}, we will see that $\mathcal{M}(n)$ grows on the order $n^n/n!$.
In particular, $\lim_{n\to\infty}\mathcal{M}(n)=\infty$.
Note that condition (3) holds if $\pi_1(M)$ is left-orderable and $\langle[K]\rangle$ is a relatively convex subgroup (see \cite{ClayRolfsen}). Further note that condition (2) holds for every nontrivial $g\in\pi_1(M)$ when $M$ is hyperbolic (see \cite{AFW}).

We then provide an explicit family of examples in Section~\ref{subsec:main:ex} realizing the conditions in Theorem~\ref{thm:infinite}. Specifically, let $M$ be a genus $g\geq 1$ surface bundle over $S^1$ with monodromy $\psi$, that is, \[M=\frac{\Sigma_g\times[0,1]}{(p,1)\sim(\psi(p),0)}.\] Assume without loss of generality that $\psi$ has a fixed point $p_0\in\Sigma_g$, and let $K=(\{p_0\}\times[0,1])/\sim$. As a consequence of Theorem~\ref{thm:infinite}, we prove the following.

\begin{restatable*}{corollary}{examplecor} \label{cor:examplecor}
If the monodromy $\psi$ is Anosov (when $g=1$) or pseudo-Anosov (when $g\geq 2$) and the induced automorphism $\psi_*$ of $H_1(\Sigma_g)$ has eigenvalues which are real and positive, then the \hyperref[conj:ac]{Almost-Concordance Conjecture} holds for $x=[K]\in[S^1,M]$.
\end{restatable*}

\noindent Note that Corollary~\ref{cor:examplecor} resolves the \hyperref[conj:ac]{Almost-Concordance Conjecture} in some cases where $x=[K]\in\pi_1(M)$ is primitive, infinite order, and normally generates $\pi_1(M)$; see Example \ref{ex:torusbundleanosov}.
Aside from the \hyperref[thm:lightbulb]{Concordance Lightbulb Theorem}, the conjecture has not been previously resolved in any such cases.


\subsection{Notation} \label{subsec:intro:notate}

Throughout, $K\subset M$ will denote a knot in the closed orientable 3-manifold $M$. 
We will usually assume $M$ and $K$ are oriented, so that the fundamental class $[M]$ and meridian of $K$ are well-defined. 
We will typically assume $M$ is aspherical and $[K]\in[S^1,M]$ is nontrivial unless otherwise stated. 
We will denote the regular neighborhood of a submanifold $X$ in a manifold $Y$ by $\nu X$ and the exterior of $X$ in $Y$ by $E_X$.
For convenience, we will often assume the basepoint of $M$ is on $\partial\nu K$, so that there is a canonical basing of $K$ consisting of a straight line segment from the basepoint to a point on $K$ in a $D^2$ fiber of $\nu K$.
We will also denote the based homotopy class of $K$ given by this basing by $[K]\in\pi_1(M)$.
The knot group of $K\subset M$, the fundamental group $\pi_1(E_K)$, will be denoted by $\pi$. 
We will let $\Gamma\trianglelefteq\pi$ be the normal subgroup $\Gamma=\ker(\pi\twoheadrightarrow\pi_1(M))$, so that $\pi/\Gamma_n$ is the $n\upth$ $\pi_1(M)$-lower central quotient of $\pi$.
For some other knot $K'\subset M$, we will use analogous notation.


\subsection*{Acknowledgements}

The author thanks Mark Powell for suggesting the invariants from \cite{Stees24} could be applied to study the \hyperref[conj:ac]{Almost-Concordance Conjecture}.
The author is also grateful to Kent Orr and Mark Powell for many valuable conversations.
A talk Jonathan Johnson gave at ICERM provided insight into the orderability aspect of Corollary~\ref{cor:examplecor}, and the author thanks him for helpful conversations.
Additionally, the author thanks Alexandra Kjuchukova, Slava Krushkal, and Maggie Miller for discussions concerning this work.


\tableofcontents


\section{Milnor's invariants in closed orientable 3-manifolds} \label{sec:Milnor}

In this section, we suspend our assumption that $M$ is aspherical in order to discuss the invariants of \cite{Stees24}, which are defined for links in closed orientable 3-manifolds, in full generality. 
Here we summarize some of the main contents of \cite{Stees24}. 
See also \cite{Stees23} for further details.

To any link $L\subset M$, we associate a tower of spaces $\{X_n(L)\}_{n\geq 1}$ over $\pi_1(M)$, well-defined up to homotopy equivalence and equipped with canonical homotopy classes of maps $\iota_n(L):M\to X_n(L)$. 
The space $X_n(L)$ is the homotopy pushout of the diagram $\mathcal{D}_n(L)$ below, where $j:E_L\hookrightarrow M$ is the inclusion and $p_n$ is induced by the canonical projection $\pi\twoheadrightarrow\pi/\Gamma_n$.
\begin{center}
\begin{tikzcd}
 & & E_L \arrow[rr, "p_n"] \arrow[dd, hookrightarrow, "j"'] & & K(\pi/\Gamma_n,1) \\
\mathcal{D}_n(L) & = & & & \\
 & & M & &
\end{tikzcd}
\end{center}
A useful model for $X_n(L)$ is the standard homotopy pushout $X_n(L)=M_{p_n}^\times\cup_{E_L} M=M_{p_n}^\times\cup_{\partial E_L}\nu L$, where $M_{p_n}^\times$ is the reduced mapping cylinder of the map $p_n$. 
The canonical projections $p_{m,n}:\pi/\Gamma_m\twoheadrightarrow\pi/\Gamma_n$ for $m\geq n$ induce morphisms of diagrams $\mathcal{D}_m(L)\to\mathcal{D}_n(L)$ over $\pi_1(M)$ which induce based maps $\psi_{m,n}=\psi_{m,n}(L):X_m(L)\to X_n(L)$, yielding the following tower of spaces over $\pi_1(M)$.
\begin{center}
\begin{tikzcd}
& & M \arrow[dl,"\iota_{n+1}(L)"'] \arrow[d,"\iota_n(L)"] \arrow[drr,"\iota_1(L)"] & & &\\
\cdots \arrow[r,"\psi_{n+2{,}n+1}"'] & X_{n+1}(L) \arrow[r,"\psi_{n+1{,}n}"'] & X_n(L) \arrow[r,"\psi_{n{,}n-1}"'] & \cdots \arrow[r,"\psi_{2{,}1}"'] & X_1(L)
\end{tikzcd}
\end{center}
Using the Stallings-Dwyer Theorem and obstruction theory, one sees that a concordance between two links $L, L'\subset M$ determines a canonical based homotopy equivalence of pairs \[\ol{h}:(X_n(L'),\nu L')\xrightarrow{\simeq}(X_n(L),\nu L)\] for each $n$, and the maps $\iota_n(L'):M\to X_n(L')$ and $\iota_n(L):M\to X_n(L)$ correspond via this homotopy equivalence, that is, we have $\ol{h}\circ\iota_n(L')\simeq \iota_n(L)$.

The invariants of \cite{Stees24} are defined \emph{relative to a fixed link $L$}. 
In 3-manifolds $M\neq S^3$, there is no sensible choice of link in $M$ to which all others may be reasonably compared. 
Recall from Section~\ref{sec:intro} that concordant links have components which are freely homotopic in $M$, so it does not make sense to compare, for instance, links with homotopically essential components to the local unlink. 
The fixed link $L$ therefore plays the role of the unlink, and other links $L',L''\subset M$ are compared ``over $L$".

Given a link $L'\subset M$ and an \emph{$n$-basing} $\phi$ for $L'$ relative to the fixed link $L$, data which determines a canonical homotopy equivalence of pairs $\ol{h}_\phi:(X_n(L'),\nu L')\xrightarrow{\simeq}(X_n(L),\nu L)$, invariants are obtained by considering the homotopy class $h_n(L',\phi)=\ol{h}_\phi\circ\iota_n(L')\in[M,X_n(L)]_0$ and comparing it to the fixed class $h_n(L,\id)=\iota_n(L)$. 

\begin{definition} \label{def:nbasing}
Fix an $m$-component link $L\subset M$, and let $L'\subset M$ be another $m$-component link. An {\it $n$-basing} for $L'$ relative to $L$ is an ordered pair $(\phi,\phi_\partial)$ such that
\begin{enumerate}[label=(\arabic*)]
\item $\phi_\partial:H_1(\partial E_{L'})\xrightarrow{\cong}H_1(\partial E_L)$ is an \emph{admissible} isomorphism, that is, it preserves meridians, orientations, and orderings of components, and
\item $\phi:\pi'/\Gamma'_n\xrightarrow{\cong}\pi/\Gamma_n$ is an isomorphism over $\pi_1(M)$ such that $\phi_\partial$ is {\it compatible} with $\phi$, that is, given any basing $\tau'$ for $L'$, there exists a basing $\tau$ for $L$ such that the following diagram commutes.
\begin{center}
\begin{tikzcd}
\pi_1(\partial E_{L'}\cup\tau') \arrow[r, "\phi_{\partial*}"', "\cong"] \arrow[d] & \pi_1(\partial E_L\cup\tau) \arrow[d] \\
\pi' \arrow[d, twoheadrightarrow] & \pi \arrow[d, twoheadrightarrow] \\
\pi'/\Gamma'_n \arrow[r, "\cong", "\phi"'] & \pi/\Gamma_n
\end{tikzcd}
\end{center}
\end{enumerate}
\end{definition}

\noindent We will often suppress notation and write $\phi$ instead of the pair $(\phi,\phi_\partial)$.

\begin{proposition}[\cite{Stees24} Proposition 4.2]\label{prop:nbasing}
The link $L'$ admits an $n$-basing relative to $L$ if and only if there exists a based homotopy equivalence of pairs \[h:(M_{p_n'}^\times,\partial E_{L'})\xrightarrow{\simeq}(M_{p_n}^\times,\partial E_L)\] over $\pi_1(M)$ restricting to an admissible homeomorphism $\partial E_{L'}\xrightarrow{\cong}\partial E_L$. 
Additionally, any two homotopy equivalences realizing the same $n$-basing are based homotopic as maps of pairs. 
Furthermore,  such a homotopy equivalence $h$ extends canonically (up to based homotopy) to a based homotopy equivalence of pairs \[\ol{h}:(X_n(L'),\nu L')\xrightarrow{\simeq} (X_n(L),\nu L).\]
\end{proposition}

\begin{definition} \label{def:hpair}
Fix an $m$-component link $L\subset M$, and let $L'\subset M$ be another $m$-component link. Suppose $L'$ admits an $n$-basing $(\phi,\phi_\partial)$ relative to $L$. The {\it $n^{\text{th}}$ lower central homotopy invariant} of the pair $(L',\phi)$ relative to $L$ is the homotopy class \[h_n(L',\phi)=\ol{h}_\phi\circ\iota_n(L')\in[M,X_n(L)]_0,\] where $\ol{h}_\phi$ is the homotopy equivalence of pairs $\ol{h}_\phi:(X_n(L'),\nu L')\xrightarrow{\simeq} (X_n(L),\nu L)$ obtained from Proposition~\ref{prop:nbasing}.
\end{definition}

The class $h_n(L',\phi)$ is well-defined, but it may depend on the choice of $n$-basing $\phi$. This dependence is measured by the group of {\it self-$n$-basings} of $L$, denoted $\Aut(\pi/\Gamma_n,\partial)$, which consists of all $n$-basings of $L$ relative to itself. The group $\Aut(\pi/\Gamma_n,\partial)$ acts on the set $[M,X_n(L)]_0$ by post-composition with homotopy self-equivalences of the pair $(X_n(L),\nu L)$. To remove the indeterminacy introduced by the choice of $n$-basing, we take the value of $h_n(L',\phi)$ in the orbit space $[M,X_n(L)]_0/\Aut(\pi/\Gamma_n,\partial)$.

\begin{definition} \label{def:h}
Let $L$ and $L'$ be as in the previous definition. The {\it $n^{\text{th}}$ lower central homotopy invariant} of $L'$ relative to $L$ is the image $h_n(L')$ of the homotopy class $h_n(L',\phi)$ in the orbit space $[M,X_n(L)]_0/\Aut(\pi/\Gamma_n,\partial)$. We say this invariant \emph{vanishes} if $h_n(L')=h_n(L)$.
\end{definition}

An approach K. Orr pioneered in \cite{Orr89} inspired the approach to defining the lower central homotopy invariants; they are defined in a similar manner to invariants of P. Heck \cite{Heck}. The invariants $h_n$ also determine two additional invariants:
\begin{itemize}
\item The \emph{lower central homology invariants} \[\theta_n(L')=[\theta_n(L',\phi)]\in H_3(X_n(L))/\Aut(\pi/\Gamma_n,\partial),\] where $\theta_n(L',\phi)=h_n(L',\phi)_*[M]\in H_3(X_n(L))$.
\item The invariants $\ol{\mu}_n(L')$, a further reduction of the $h_n$ which we call {\it Milnor's invariants}. Like Milnor's $\ol{\mu}$-invariants for links in $S^3$, they determine the lower central quotients of link groups one step at a time.
\end{itemize}

We will often refer to any of the invariants $h_n$, $\theta_n$, and $\ol{\mu}_n$ colloquially as ``Milnor's invariants", although the invariants $\ol{\mu}_n$ are the true extensions of Milnor's invariants for links in $S^3$. The lower central homology invariants $\theta_n$ will feature heavily in this work. We now paraphrase the main results involving the invariants $h_n$, $\theta_n$, and $\ol{\mu}_n$ (see \cite{Stees24} for details):
\begin{itemize}
\item \emph{Invariance:} The $h_n$ (hence $\theta_n$ and $\ol{\mu}_n$) are topological concordance invariants. Given concordant links $L',L''\subset M$, we have $h_n(L')=h_n(L'')$ when defined. In particular, if $L'$ is concordant to the fixed link $L$, then $h_n(L')$ is defined and vanishes for all $n$.
\item \emph{Characterization for $h_n$:} We have $h_n(L')=h_n(L'')$ if and only if $L'$ is $n$-cobordant to $L''$. Roughly speaking, $L'$ and $L''$ are {\it $n$-cobordant} if they cobound a surface in $M\times[0,1]$ which looks like a concordance to the $n^{\text{th}}$ lower central quotients. In particular, $h_n(L')$ vanishes if and only if $L'$ is $n$-cobordant to the fixed link $L$.
\item \emph{Characterization for $\ol{\mu}_n$:} We have $\ol{\mu}_n(L')=\ol{\mu}_n(L'')$ if and only if there exists an $(n+1)$-basing for $L''$ relative to $L'$. In particular, $\ol{\mu}_n(L')$ vanishes if and only if $L'$ admits an $(n+1)$-basing relative to the fixed link $L$.
\item {\it Realization:} Every homotopy class in $[M,X_n(L)]_0$ satisfying some mild conditions is realized as $h_n(L',\phi)$ for some link $L'\subset M$.
\item \emph{Lifting property:} The invariant $\ol{\mu}_n(L')$ vanishes if and only if the invariant $h_n(L')$ lifts to a realizable class in $[M,X_{n+1}(L)]_0$.
\item {\it Specialization to previous invariants:} For links in $S^3$ and $n\geq 2$, the invariants $h_n$ relative to the unlink are equivalent to invariants of Orr \cite{Orr89}, and the invariant $\ol{\mu}_n$ relative to the unlink is equivalent to Milnor's $\ol{\mu}$-invariant of length $n+1$ \cite{Milnor}. For empty links, the invariants $\theta_n$ and $\ol{\mu}_n$ are the Cha-Orr homology cobordism invariants, ``Milnor's invariants of 3-manifolds", from \cite{ChaOrr}.
\end{itemize}


\subsection{Invariance under local knotting} \label{subsec:Milnor:localknot}

We now establish the invariance of the lower central homotopy invariants $h_n$ (hence the lower central homology invariants $\theta_n$ and Milnor's invariants $\ol{\mu}_n$) under local knotting.

\begin{proposition} \label{prop:localknot}
Let $K\subset M$ be a knot, and suppose $K'=J\cdot K$ is obtained from $K$ by local knotting, where $J\subset S^3$. 
Then $h_n(K')=h_n(K)$ for all $n$.
\end{proposition}

\begin{proof}
We first define a map of pairs $(E_{K'},\partial\nu K')\to(E_K,\partial\nu K)$ which restricts on the boundary to an admissible homeomorphism and which induces an isomorphism on $H_*(-;\Z[\pi_1(M)])$. 
First, define a degree one map $g:S^3\to S^3$ which sends $J$ homeomorphically to the unknot $U$ by extending a map $E_J\to E_U$ which maps $\partial\nu J\xrightarrow{\cong}\partial\nu U$ by an admissible homeomorphism.
An obstruction theory argument implies such a map exists since $E_U$ is a $K(\Z,1)$. 
Define $f:M\to M$ by $f=g\,\#\id_M$, where we view $(M,K')$ as a connected sum of pairs $(M,K')=(S^3,J)\,\#\,(M,K)$. 
This map sends $K'$ homeomorphically to $K$. 
By restricting to $E_{K'}$, we obtain the desired map of pairs. 
A Mayer-Vietoris argument shows that $f|_{E_{K'}}$ induces an isomorphism on $H_*(-;\Z[\pi_1(M)])$.
By the Stallings-Dwyer theorem, $f|_{E_{K'}}$ induces isomorphisms on $\pi_1(M)$-lower central quotients $\pi'/\Gamma'_n\xrightarrow{\cong}\pi/\Gamma_n$ for all $n$. 
Thus, for each $n$, $K'$ admits an $n$-basing $\phi_n$ relative to $K$, so $h_n(K',\phi_n)$ and $h_n(K')$ are defined for each $n$.

We now show $h_n(K',\phi_n)=h_n(K,\id)$ for all $n$. 
Since $g:S^3\to S^3$ is a degree 1 map, $g\simeq\id_{S^3}$. 
Thus, $f=g\#\id_M\simeq\id_{S^3}\#\id_M=\id_M$. 
Let $H:M\times[0,1]\to M$ be a homotopy from $\id_M$ to $f$.
Postcomposing with the map $\iota_n(K):M\hookrightarrow X_n(K)$ yields a homotopy from $h_n(K,\id)=\iota_n(K)$ to $\iota_n(K)\circ f$. 
It remains to see that $\iota_n(K)\circ f\simeq h_n(K',\phi_n)$. 
By construction, and by the fact that $f$ induces the $n$-basing $\phi_n$, we have the homotopy-commutative cube seen in Figure~\ref{fig:htpycube}.
\begin{figure}
\begin{center}
\begin{tikzcd}
E_{K'} \arrow[rr,"f|_{E_{K'}}"] \arrow[dr,hookrightarrow] \arrow[dd,hookrightarrow] & & E_K \arrow[dr,hookrightarrow] \arrow[dd,hookrightarrow]\\ 
& M_{p_n'}^\times \arrow[rr, crossing over,"\simeq" near start,"h_{\phi_n}"' near start] & & M_{p_n}^\times \\
M \arrow[dr,hookrightarrow,"\iota_n(K')"'] \arrow[rr,"f" near start] & & M \arrow[dr, hookrightarrow,"\iota_n(K)"] \\
& X_n(K') \arrow[from=uu, crossing over] \arrow[rr,"\simeq","\ol{h}_{\phi_n}"'] & & X_n(K). \arrow[from=uu, hookrightarrow]
\end{tikzcd}
\caption{The homotopy-commutative cube from the proof of Proposition~\ref{prop:localknot}.}
\label{fig:htpycube}
\end{center}
\end{figure}
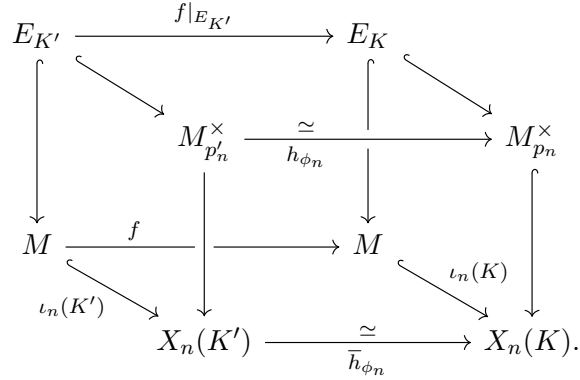
The bottom face of the cube shows that $h_n(K',\phi_n)=\ol{h}_{\phi_n}\circ\iota_n(K')\simeq\iota_n(K)\circ f$. Thus, $h_n(K',\phi_n)=h_n(K,\id)$, and $h_n(K')=h_n(K)$.
\end{proof}

\begin{corollary} \label{cor:acinv}
The invariants $h_n$, $\theta_n$, and $\ol{\mu}_n$ are almost-concordance invariants.
\end{corollary}

\noindent Note that Proposition~\ref{prop:localknot} holds more generally for links, where we allow local knotting on all components.


\section{Algebraic properties of knots in 3-manifolds} \label{sec:algprop}

Before producing families of knots which we distinguish up to almost-concordance using the lower central homology invariants $\theta_n$, we analyze some relevant algebraic properties of knots in 3-manifolds. In Section~\ref{subsec:algprop:lcq}, we analyze the $\pi_1(M)$-lower central quotients of homotopically essential knots in aspherical 3-manifolds. In Section~\ref{subsec:algprop:thetapair}, we prove some general statements about $\theta$-invariants of knots in 3-manifolds and do not place these restrictions on the 3-manifold or class of the knot.


\subsection{The $\pi_1(M)$-lower central quotients} \label{subsec:algprop:lcq}

When $M$ is aspherical and $[K]\in[S^1,M]$ is nontrivial, the lower central quotients $\Gamma/\Gamma_n$ admit a particular form that allows us to analyze $\theta$-invariants relative to $K$ using tools from the theory of Milnor's invariants of links in $S^3$ \cites{Milnor,Orr89,Cochran90,Cochran91}.
We prove that, in this setting, $\Gamma/\Gamma_n$ is isomorphic to the $n\upth$ lower central quotient of a free group. 
The $\pi_1(M)$-lower central quotients are extensions of these quotients by $\pi_1(M)$.

\begin{proposition} \label{prop:finfty}
Let $K\subset M$ be a homotopically essential knot in the aspherical 3-manifold $M$. 
Then \[\Gamma/\Gamma_n\cong F(\infty)/F(\infty)_n\] for all $n$, where $F(\infty)$ is a countably infinitely generated free group whose generators correspond to (left) cosets of $\langle[K]\rangle\leq\pi_1(M)$ or, equivalently, meridians of $p^{-1}(K)$, where $p:\wt{M}\to M$ is the universal covering map.
\end{proposition}

\begin{proof}
Let $\wt{X}$ denote the preimage of a space $X\subset M$ under the universal covering $p:\wt{M}\to M$. 
We analyze the Mayer-Vietoris sequence with coefficients in $\Z[\pi_1(M)]$ corresponding to the decomposition $M=E_K\cup_{\partial\nu K}\nu K$. 
Because $M$ is aspherical, $H_i(M;\Z[\pi_1(M)])=0$ for all $i>0$. 
Since $K$ is homotopically essential, $H_1(\nu K;\Z[\pi_1(M)])=0$ and $H_2(\partial\nu K;\Z[\pi_1(M)])=0$. 
The Mayer-Vietoris sequence therefore yields isomorphisms \[H_1(\partial\nu K;\Z[\pi_1(M)])\xrightarrow{\cong}H_1(E_K;\Z[\pi_1(M)])\]  and \[0=H_2(\partial\nu K;\Z[\pi_1(M)])\xrightarrow{\cong}H_2(E_K;\Z[\pi_1(M)])=0.\] 
The nontriviality of $[K]\in[S^1,M]$ also implies $H_1(\partial\nu K;\Z[\pi_1(M)])=H_1(\mu;\Z[\pi_1(M)])$, where $\mu$ is the meridian for $K$. 
Thus, the inclusion $\mu\hookrightarrow E_K$ induces isomorphisms on $H_i(-;\Z[\pi_1(M)])$ for $i=1,2$. 
Equivalently, the inclusion $\wt{\mu}\hookrightarrow\wt{E_K}$ induces isomorphisms on $H_i(-;\Z)$ for $i=1,2$. 
The space $\wt{\mu}$ consists of meridians for the components of $\wt{K}$ which are in (non-canonical) bijective correspondence with (left) cosets of $\langle[K]\rangle\leq\pi_1(M)$. 
We may turn $\wt{\mu}$ into a connected space by connecting the meridians to the basepoint of $\wt{E_K}$ by pairwise disjoint paths in $\wt{E_K}$. 
Note that the deck transformations on $\wt{M}$ act properly discontinuously on the meridians, so resulting space, which we still call $\wt{\mu}$, is a countably infinite wedge of circles with fundamental group $F(\infty)$ and generators corresponding to cosets of $\langle[K]\rangle\leq\pi_1(M)$. 
This modification does not alter the homology isomorphisms induced by the inclusion $\wt{\mu}\hookrightarrow\wt{E_K}$. 
As $\wt{\mu}$ is a $K(F(\infty),1)$ and \[0=H_2(\wt{\mu};\Z)\xrightarrow{\cong}H_2(\wt{E_K};\Z)\twoheadrightarrow H_2(\pi_1(\wt{E_K});\Z)=H_2(\Gamma;\Z),\] where the surjection follows from Hopf's Theorem (see \cite{Brown} Section II.5), we have isomorphisms $H_i(F(\infty);\Z)\xrightarrow{\cong}H_i(\Gamma;\Z)$ for $i=1,2$.
By the Stallings-Dwyer Theorem \cites{Stallings,Dwyer}, we obtain isomorphisms on lower central quotients $F(\infty)/F(\infty)_n\xrightarrow{\cong}\Gamma/\Gamma_n$.
\end{proof}

In Section~\ref{subsec:construct:Gamma}, we use finitely generated subgroups $F(\mu_\mathbf{g})/F(\mu_\mathbf{g})_n\leq F(\infty)/F(\infty)_n$, where the $F(\mu_\mathbf{g})=\langle\mu_{g_1},\dots\mu_{g_m}\rangle$ are free groups generated by finitely many generators of $F(\infty)$, as input data for constructing knots which are not almost-concordant to the fixed knot $K$.

\subsection{The space $X_n(K)$ and $\theta$-invariants for knots in 3-manifolds} \label{subsec:algprop:thetapair}

Once we construct families of knots which are candidates to satisfy Theorem~\ref{thm:family} in Section~\ref{subsec:construct:Gamma}, we analyze the differences of their $\theta$-invariants with the fixed class $\theta_n(K)\in H_3(X_n(K))/\Aut(\pi/\Gamma_n,\partial)$ in Section~\ref{subsec:construct:S3}.
In this section, we show that for a knot $K\subset M$ the action of the group of self-$n$-basings $\Aut(\pi/\Gamma_n,\partial)$ on $H_3(X_n(K))$ is trivial, and thus it is enough to understand these differences in $H_3(X_n(K))$.
More precisely, Proposition~\ref{prop:thetapair} and Corollary~\ref{cor:thetapair} below assert the lower central homology invariant $\theta_n(K',\phi)$ of the pair $(K',\phi)$ is independent of the choice of $n$-basing $\phi$ for $K'$ relative to $K$.
Thus, if $\theta_n(K',\phi')\neq\theta_n(K'',\phi'')\in H_3(X_n(K))$ \emph{for any choices of $\phi'$ and $\phi''$}, $K'$ and $K''$ are not almost-concordant. 
This fact does not generalize to links.
Indeed, for links in $S^3$ the invariant $\theta_n$ depends on the choice of $n$-basing; see \cites{Orr89,IgusaOrr}.
In this subsection, we do not assume that $M$ is aspherical or that $[K]\in [S^1,M]$ is nontrivial.

\begin{proposition} \label{prop:thetapair}
Fix $K\subset M$, and suppose $K',K''\subset M$ admit $n$-basings $\phi'$ and $\phi''$, respectively, relative to $K$. Then $\theta_n(K')=\theta_n(K'')$ if and only if $\theta_n(K',\phi')=\theta_n(K'',\phi'')$. In particular, $\theta_n(K')$ vanishes if and only if $\theta_n(K',\phi')$ vanishes.
\end{proposition}

\begin{corollary} \label{cor:thetapair}
Fix $K\subset M$, and suppose $\theta_n(K')$ is defined. Then $\theta_n(K',\phi')=\theta_n(K',\phi'')$ for any $n$-basings $\phi'$ and $\phi''$.
\end{corollary}

\begin{proof}
Take $K'=K''$ and let $\phi'$ and $\phi''$ be any two $n$-basings for $K'$ relative to $K$ in the statement of Proposition~\ref{prop:thetapair}.
\end{proof}

Before we prove Proposition~\ref{prop:thetapair}, we first recall our setup.
Fix a knot $K\subset M$. 
Recall from Section~\ref{sec:Milnor} that $X_n(K)=M_{p_n}^\times\cup_{E_K}M=M_{p_n}^\times\cup_{\partial E_K}\nu K$, where $M_{p_n}^\times$ is the reduced mapping cylinder of $p_n:E_K\to K(\pi/\Gamma_n,1)$.
Given a knot $K'\subset M$ which admits an $n$-basing $\phi$ relative to $K$, we have $\theta_n(K')=[\theta_n(K',\phi)]\in H_3(X_n(K))/\Aut(\pi/\Gamma_n,\partial)$, where $\theta_n(K',\phi)=h_n(K',\phi)_*[M]$, and where $h_n(K',\phi)$ is the homotopy class of maps $M\to X_n(K)$ induced by the $n$-basing whose existence is guaranteed by Proposition~\ref{prop:nbasing}.

Next, we analyze the group $H_3(X_n(K))$.

\begin{lemma} \label{lem:H3}
The Mayer-Vietoris sequence corresponding to the decomposition $X_n(K)=M_{p_n}^\times\cup_{E_K}M$ induces an isomorphism \[H_3(\pi/\Gamma_n)\oplus H_3(M)\xrightarrow{\cong} H_3(X_n(K)).\] Furthermore, if $K'\subset M$ admits an $n$-basing $\phi$ relative to $K$, then $\theta_n(K',\phi)\in H_3(X_n(K))$ corresponds under this isomorphism to a class of the form $(\eta,[M])$ for some $\eta\in H_3(\pi/\Gamma_n)$.
\end{lemma}

\begin{proof}
First consider the Mayer-Vietoris sequence corresponding to $M=E_K\cup_{\partial\nu K}\nu K$.
Observe that $H_3(E_K)=H_3(\nu K)=H_2(\nu K)=0$, so we obtain the following sequence.
\[0\to H_3(M)\to H_2(\partial\nu K)\to H_2(E_K)\to H_2(M)\to\cdots\]
Analysis of the connecting homomorphism $H_3(M)\to H_2(\partial\nu K)$ shows that it is an isomorphism.
Thus, $H_2(E_K)\to H_2(M)$ is injective.

Next consider the Mayer-Vietoris sequence corresponding to $X_n(K)=M_{p_n}^\times\cup_{E_K}M$.
Note that $M_{p_n}^\times$ is a $K(\pi/\Gamma_n,1)$, so $H_*(M_{p_n}^\times)=H_*(\pi/\Gamma_n)$.
Using $H_3(E_K)=0$ again, we obtain
\[0\to H_3(\pi/\Gamma_n)\oplus H_3(M)\to H_3(X_n(K))\to H_2(E_K)\to H_2(\pi/\Gamma_n)\oplus H_2(M)\to\cdots\]
By the injectivity of the inclusion-induced homomorphism $H_2(E_K)\hookrightarrow H_2(M)$, the homomorphism $H_2(E_K)\to H_2(\pi/\Gamma_n)\oplus H_2(M)$ is also injective.
Thus, the inclusions $M_{p_n}^\times\hookrightarrow X_n(K)$ and $M\hookrightarrow X_n(K)$ induce an isomorphism \[H_3(\pi/\Gamma_n)\oplus H_3(M)\xrightarrow{\cong} H_3(X_n(K)).\]

Now observe that the Mayer-Vietoris sequence corresponding to the alternate decomposition $X_n(K)=M_{p_n}^\times\cup_{\partial\nu K}\nu K$ yields the short exact sequence \[0\to H_3(\pi/\Gamma_n)\to H_3(X_n(K))\to H_2(\partial\nu K)\to 0\] because $H_3(X_n(K))\to H_2(\partial\nu K)$ is surjective with $\iota_n(K)_*[M]\mapsto[\partial\nu K]$.
Because $H_2(\partial\nu K)\cong\Z$, the homomorphism $H_3(X_n(K))\to H_2(\partial\nu K)$ splits, and we may choose this splitting in accordance with the above direct sum decomposition because $\iota_n(K)_*[M]\mapsto[\partial\nu K]$. 
(Recall that $\iota_n(K):M\hookrightarrow X_n(K)$ is the inclusion.) 
In other words, classes of the form $(\eta,[M])\in H_3(\pi/\Gamma_n)\oplus H_3(M)$ are precisely those whose images in $H_3(X_n(K))$ map to $[\partial\nu K]\in H_2(\partial\nu K)$.

Finally, suppose $K'\subset M$ admits an $n$-basing $\phi$ relative to $K$ so that $\theta_n(K',\phi)\in H_3(X_n(K))$ is defined.
By Theorem C of \cite{Stees24}, classes $\theta\in H_3(X_n(K))$ which are realized by knots in $M$ must satisfy $\theta\mapsto[\partial\nu K]$, so $\theta_n(K',\phi)$ corresponds to a class of the form $(x,[M])\in H_3(\pi/\Gamma_n)\oplus H_3(M)$.
\end{proof}

\begin{proof}[Proof of Proposition~\ref{prop:thetapair}]
Fix $K\subset M$, and suppose $K', K''\subset M$ admit $n$-basings $\phi'$ and $\phi''$, respectively, relative to $K$.
The backward direction is immediate by definition.
We first prove the forward direction in the case that $K''=K$ and $\phi''=\id_{\pi/\Gamma_n}$.
Suppose that $\theta_n(K')$ vanishes, that is, $\theta_n(K',\phi_0')=\theta_n(K,\id)$ for some $n$-basing $\phi_0'$. 
The composition $\psi=\phi'\circ(\phi_0')^{-1}$ is a self-$n$-basing for $K$, so $\theta_n(K',\phi'_0)=\theta_n(K,\id)$ implies \[\theta_n(K',\phi')=\psi\cdot\theta_n(K',\phi_0')=\psi\cdot\theta_n(K,\id)=\theta_n(K,\psi)\] (see Theorem A of \cite{Stees24}).
It therefore suffices to prove $\theta_n(K,\psi)=\theta_n(K,\id)$.

Recall from Section~\ref{sec:Milnor} that the action of $\Aut(\pi/\Gamma_n,\partial)$ on $H_3(X_n(K))$ is given by postcomposition with homotopy self-equivalences of the pair $(X_n(K),\nu K)$ which restrict to homotopy self-equivalences of the pair $(M_{p_n}^\times,\partial\nu K)$.
Thus, the action on $H_3(X_n(K))$ induces to an action on $H_3(\pi/\Gamma_n)$ by automorphisms.
By Lemma~\ref{lem:H3}, classes $\theta \in H_3(X_n(K))$ realized by knots in $M$ correspond to elements of the form $(\eta,[M])\in H_3(\pi/\Gamma_n)\oplus H_3(M)\cong H_3(X_n(K))$.
Because the action of an element in $\Aut(\pi/\Gamma_n,\partial)$ on a realizable class $\theta=(\eta,[M])$ produces another realizable class, we have $\psi\cdot\theta=\psi\cdot (\eta,[M])=(\psi\cdot\eta,[M])$, where the action $\psi\cdot\eta$ is the restriction-induced action on $H_3(\pi/\Gamma_n)$. 
In particular, as $\psi$ acts as an automorphism on $H_3(\pi/\Gamma_n)$, $\psi\cdot (0,[M])=(0,[M])\in H_3(\pi/\Gamma_n)\oplus H_3(M)$.

The ``trivial" invariant $\theta_n(K,\id)$ is by definition $\iota_n(K)_*[M]$, where $\iota_n(K):M\hookrightarrow X_n(K)$ is the inclusion.
Thus, $\theta_n(K,\id)$ corresponds to $(0,[M])$ under the isomorphism $H_3(\pi/\Gamma_n)\oplus H_3(M)\xrightarrow{\cong} H_3(X_n(K))$ from Lemma~\ref{lem:H3}.
Therefore, for any self-$n$-basing $\psi\in\Aut(\pi/\Gamma_n,\partial)$, \[\theta_n(K,\psi)=\psi\cdot\theta_n(K,\id)=\psi\cdot(0,[M])=(0,[M])=\theta_n(K,\id).\]
This proves the forward direction in this special case.

We finish by proving the forward direction in the general case.
Let $K',K''\subset M$ admit $n$-basings $\phi'$ and $\phi''$, respectively, relative to $K$, and suppose $\theta_n(K')=\theta_n(K'')$.
Then, by definition, $\theta_n(K',\phi')=\theta_n(K'',\phi''_0)$ for some $n$-basing $\phi''_0$ for $K''$ relative to $K$. 
Let $\psi_0=(\phi')^{-1}\circ\phi_0''$ and $\psi=(\phi')^{-1}\circ \phi''$, which are $n$-basings for $K''$ relative to $K'$.
Then $\theta_n(K'',\psi_0)=\theta_n(K',\id)$ (where these are now invariants relative to $K'$).
The composition $\psi\circ\psi_0^{-1}$ is a self-$n$-basing for $K'$.
By the special case proved above, \[\theta_n(K'',\psi)=(\psi\circ\psi_0^{-1})\cdot\theta_n(K'',\psi_0)=(\psi\circ\psi_0^{-1})\cdot \theta_n(K',\id)=\theta_n(K',\id).\]
Thus, $\theta_n(K'',\phi'')=\theta_n(K',\phi')$ (where these are again invariants relative to $K$).
Note that we have used that an $n$-basing $\phi$ for $K'$ relative to $K$ induces a bijection $\phi_*:\mathcal{R}_n(K')\xrightarrow{\cong}\mathcal{R}_n(K)$ between sets of realizable classes relative to $K'$ and $K$ given by $\theta_n(K'',\psi)\mapsto\theta_n(K'',\phi\circ\psi)$ \cite{Stees24}.
\end{proof}


\section{Construction} \label{sec:construct}

In this section, we develop a method for constructing, given some fixed knot $K\subset M$, large families of knots which are pairwise homotopic and which are candidates to satisfy Theorem~\ref{thm:family}. 
In particular, we hope to show these knots are pairwise not almost-concordant using the $\theta$-invariants.
We begin with a motivating example which serves as a prototype for our construction.
Then, we describe a type of satellite operation we call an \emph{ambient connected sum with a weakly Brunnian link}.
Starting with the fixed knot $K$, we perform an ambient connected sum using data prescribed by a link $L\subset S^3$ to form another knot $K'\subset M$.
For suitably chosen connected sums of this type, which we call \emph{$\Gamma$-ambient connected sums}, the invariant $\theta_n(K')$ relative to $K$ is defined and its difference from $\theta_n(K)$ is measured by $L$.


\subsection{A motivating example} \label{subsec:construct:ex0}

Before presenting our construction for homotopically essential knots in aspherical 3-manifolds, we begin with a prototypical method of constructing families of pairwise not almost-concordant knots in 3-manifolds $M\neq S^3$.
This construction appears in various forms in, for instance, \cites{MillerD, Schneiderman, Celoria, FNOP}.
Given $K\subset M$, one modifies $K$ into a new knot $K'$ by performing a finger move pushing part of $K$ around some nontrivial class $g\in\pi$ and forming a clasp.
We may view $K$, the operation we perform on $K$, and the resulting knot $K'$ as sitting inside an embedded genus 2 handlebody in $M$ whose 1-handle cores are $K$ and a curve representing $g$; see Figure~\ref{fig:whitehead}.
\begin{figure}
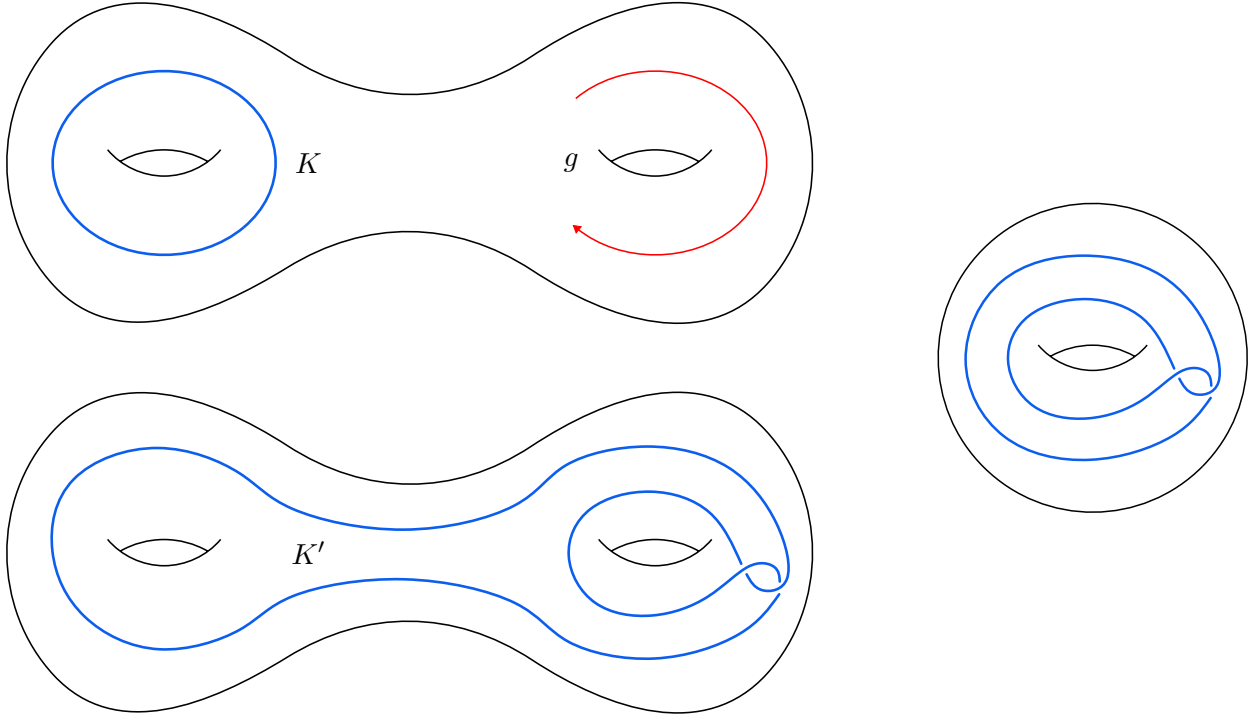

    \centering
    \begin{minipage}{0.65\textwidth}
    \begin{subfigure}{\textwidth}
    \centering
    \begin{tikzpicture}
        \node[anchor=south west,inner sep=0] (image) at (0,0) {\includesvg[width=\textwidth]{pictures/genus2Kg.svg}};
        \begin{scope}[x={(image.south east)},y={(image.north west)}]
        \node at (0.375,0.5) {$K$};
        \node at (0.7,0.5) {$g$};
        \end{scope}
    \end{tikzpicture}
    \end{subfigure}
    \vspace{1em}
    
    \begin{subfigure}{\textwidth}
    \centering
    \begin{tikzpicture}
        \node[anchor=south west,inner sep=0] (image) at (0,0) {\includesvg[width=\textwidth]{pictures/genus2Kprime.svg}};
        \begin{scope}[x={(image.south east)},y={(image.north west)}]
        \node at (0.375,0.5) {$K'$};
        \end{scope}
    \end{tikzpicture}
    \end{subfigure}
    \end{minipage}
    \hfill
    \begin{minipage}{0.25\textwidth}
    \begin{subfigure}{\textwidth}
    \centering
    \begin{tikzpicture}
        \node[anchor=south west,inner sep=0] (image) at (0,0) {\includesvg[width=\textwidth]{pictures/whitehead.svg}};
        \begin{scope}[x={(image.south east)},y={(image.north west)}]
        \end{scope}
    \end{tikzpicture}
    \end{subfigure}
    \end{minipage}
    \caption{Top left: The knot $K$ in a genus 2 handlebody embedded in $M$. Bottom left: The knot $K'$, an ambient connected sum of $K$ with the Whitehead link. Right: A knot whose exterior in the solid torus is the exterior of the Whitehead link in $S^3$.}
    \label{fig:whitehead}
\end{figure}
In the simplest example of this construction, $K$ is the local unknot and the resulting knot $K'$ is a Whitehead curve in some embedded solid torus in $M$.

One expects to distinguish the almost-concordance classes of $K$ and $K'$ because $K'$ exhibits \emph{self-linking} behavior which $K$ does not.
Indeed, such self-linking is detected in \cites{MillerD, Schneiderman, FNOP} using various methods.
Our construction which follows generalizes this example, replacing the Whitehead link with an arbitrary \emph{weakly Brunnian link}.


\subsection{Ambient connected sums with weakly Brunnian links} \label{subsec:construct:wBrunnian}

In this section, we describe an ambient connected sum operation which takes as input a knot $K\subset M$ and a link $L\subset S^3$ and produces another knot $K'\subset M$. One may think of this procedure as a type of satellite operation.

\begin{definition} \label{def:wbrunnian}
Let $L\subset S^3$ be an $m$-component link. 
We say $L$ is \emph{weakly Brunnian} if there is a component $L_0$ of $L$ such that the link $L-L_0$ is the $(m-1)$-component unlink. 
\end{definition}

\noindent Recall that if $L_0$ in the above definition may be chosen to be any component of $L$, we say $L$ is \emph{Brunnian}, so Brunnian links are weakly Brunnian with respect to any component.

Let $L\subset S^3$ be an $m$-component weakly Brunnian link with distinguished component $L_0$.
Since $L-L_0$ is the $(m-1)$-component unlink, we may view the components of $L-L_0$ as the cores of the 1-handles in a genus $m-1$ handlebody in the standard genus $m-1$ Heegaard splitting of $S^3$.
Thus, the exterior of $L-L_0$ is a  genus $m-1$ handlebody $H_0^-$ with 3-dimensional 2-handles $h_i$ attached, $2\leq i\leq m-1$; see Figure~\ref{fig:weaklyBrunniana}.
\begin{figure}
\centering
    \begin{minipage}{0.45\textwidth}
    \begin{subfigure}{\textwidth}
    \centering
    \begin{tikzpicture}
        \node[anchor=south west,inner sep=0] (image) at (0,0) {\includesvg[width=\textwidth]{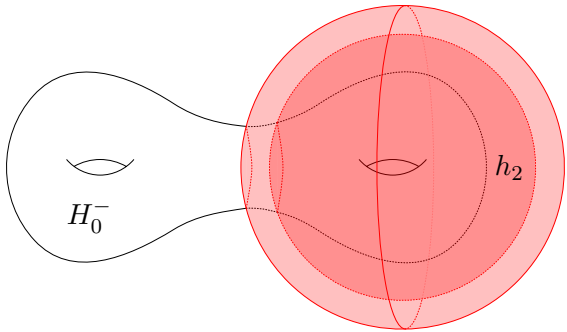}};
        \begin{scope}[x={(image.south east)},y={(image.north west)}]
        \node at (0.15,0.35) {$H_0^-$};
        \node at (0.9,0.5) {$h_2$};
        \end{scope}
    \end{tikzpicture}
    \caption{The exterior of $L-L_0$ when $m=3$.}
    \label{fig:weaklyBrunniana}
    \end{subfigure}
    \end{minipage}
    \hfill
    \begin{minipage}{0.45\textwidth}
    \begin{subfigure}{\textwidth}
    \centering
    \begin{tikzpicture}
        \node[anchor=south west,inner sep=0] (image) at (0,0) {\includesvg[width=\textwidth]{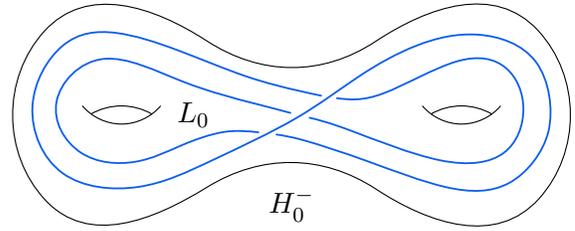}};
        \begin{scope}[x={(image.south east)},y={(image.north west)}]
        \node at (0.325,0.5) {$L_0$};
        \node at (0.5,0.1) {$H_0^-$};
        \end{scope}
    \end{tikzpicture}
    \caption{The knot $L_0\subset H_0^-$ in the case where $L$ is the Borromean rings.}
    \label{fig:weaklyBrunnianb}
    \end{subfigure}
    \end{minipage}
    \caption{Describing the exterior of a weakly Brunnian link $L$ with distinguished component $L_0$.}
    \label{fig:weaklyBrunnian}
\end{figure}
By an isotopy of $L_0$, we may ensure $L_0$ misses the $h_i$.
Consequently, the exterior $E_L$ may be viewed as the exterior of $L_0$ in $H_0^-$ along with the 2-handles $h_i$ which we now ignore; see Figure~\ref{fig:weaklyBrunnianb}. 
Form a genus $m$ handlebody $H_0$ via a boundary connected sum of a solid torus $V_0$ with $H_0^-$.

We consider two knots in $H_0$, as seen in Figure~\ref{fig:connsum1}: Let $K_0\subset H_0$ be the core of $V_0$, and form another knot $K'_0$ as an ambient connected sum of $K_0$ and $L_0$. 
More precisely, assume the boundary of the disk used in the boundary connected sum contains the basepoint of $H_0$. 
Choose basepoint paths for $K_0$ and $L_0$ in $V_0$ and $H_0^-$, respectively, and choose an arc $\alpha$ connecting $K_0$ and $L_0$ which, together with the basepoint paths, forms a nullhomotopic loop in $H_0$.
Use an untwisted band which follows $\alpha$ to ambiently surger $K_0$ and $L_0$ to form $K'_0$; see Figure~\ref{fig:connsum1}.
\begin{figure}
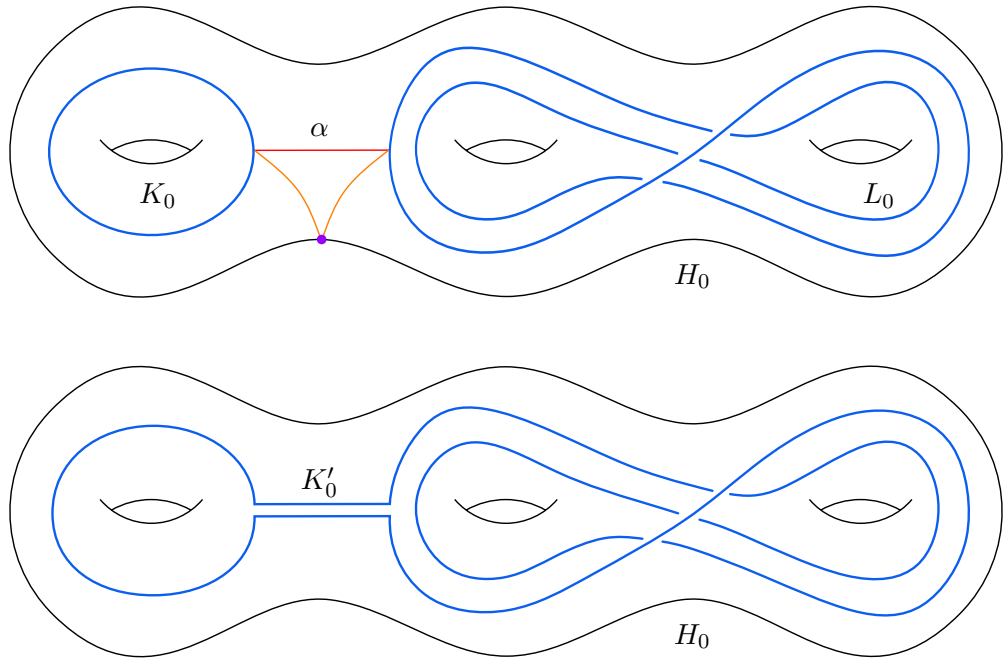

\centering
\begin{minipage}{0.8\textwidth}
    \begin{tikzpicture}
        \node[anchor=south west,inner sep=0] (image) at (0,0) {\includesvg[width=\textwidth]{pictures/genus3connsum.svg}};
        \begin{scope}[x={(image.south east)},y={(image.north west)}]
        \node at (0.15,0.35) {$K_0$};
        \node at (0.875,0.35) {$L_0$};
        \node at (0.3125,0.575) {$\alpha$};
        \node at (0.6875,0.075) {$H_0$};
        \end{scope}
    \end{tikzpicture}
\vspace{1em}

    \begin{tikzpicture}
        \node[anchor=south west,inner sep=0] (image) at (0,0) {\includesvg[width=\textwidth]{pictures/genus3Kprime.svg}};
        \begin{scope}[x={(image.south east)},y={(image.north west)}]
        \node at (0.3125,0.6) {$K_0'$};
        \node at (0.6875,0.075) {$H_0$};
        \end{scope}
    \end{tikzpicture}
\end{minipage}
\caption{Forming $K_0'$ as an ambient connected sum of $K_0$ and $L_0$ in the handlebody $H_0$.}
\label{fig:connsum1}
\end{figure}

A slightly different description of this ambient connected sum, which we now describe, will be useful in Section~\ref{subsec:construct:Gamma}; see Figure~\ref{fig:connsum2}.
\begin{figure}
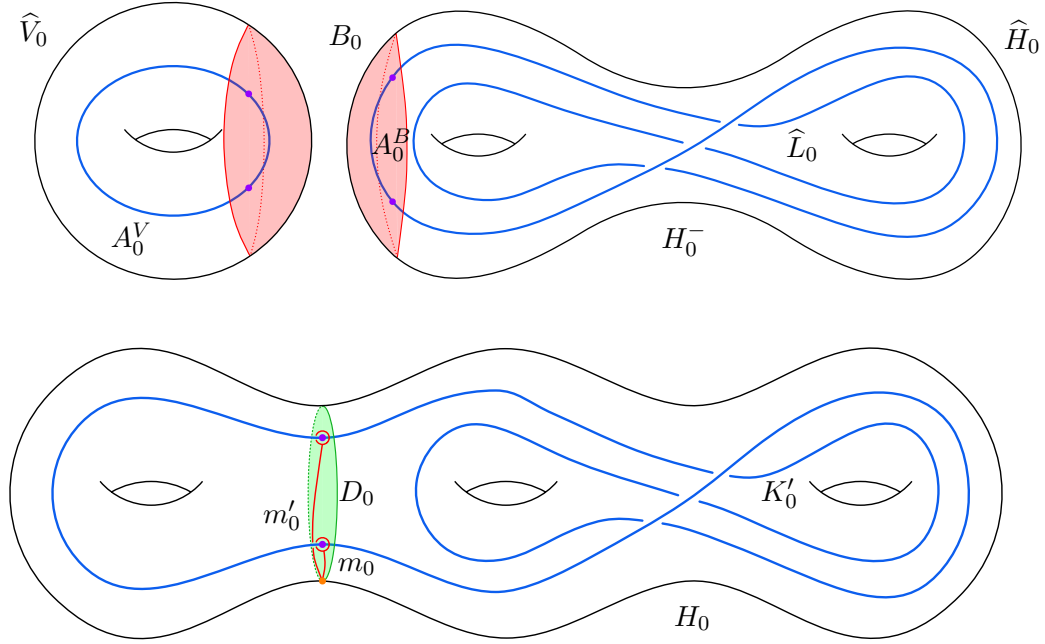

\centering
\begin{minipage}{0.8\textwidth}
    \begin{tikzpicture}
        \node[anchor=south west,inner sep=0] (image) at (0,0) {\includesvg[width=0.28\textwidth]{pictures/genus1Kconnsum.svg} };
        \begin{scope}[x={(image.south east)},y={(image.north west)}]
        \node at (0.35,0.15) {$A_0^V$};
        \node at (0,0.9) {$\wh{V}_0$};
        \end{scope}
    \end{tikzpicture}
\hfill
    \begin{tikzpicture}
        \node[anchor=south west,inner sep=0] (image) at (0,0) {\includesvg[width=0.68\textwidth]{pictures/genus2connsum.svg}};
        \begin{scope}[x={(image.south east)},y={(image.north west)}]
        \node at (0.5,0.15) {$H_0^-$};
        \node at (1,0.9) {$\wh{H}_0$};
        \node at (0,0.9) {$B_0$};
        \node at (0.0675,0.5) {$A_0^B$};
        \node at (0.675,0.5) {$\wh{L}_0$};
        \end{scope}
    \end{tikzpicture}
\vspace{1em}

    \begin{tikzpicture}
        \node[anchor=south west,inner sep=0] (image) at (0,0) {\includesvg[width=\textwidth]{pictures/genus3KprimeD.svg}};
        \begin{scope}[x={(image.south east)},y={(image.north west)}]
        \node at (0.35,0.5) {$D_0$};
        \node at (0.6875,0.075) {$H_0$};
        \node at (0.775,0.5) {$K_0'$};
        \node at (0.35,0.25) {$m_0$};
        \node at (0.275,0.425) {$m_0'$};
        \end{scope}
    \end{tikzpicture}
\end{minipage}
\caption{Forming $K_0'$ via a connected sum of pairs.}
\label{fig:connsum2}
\end{figure}
Consider a genus $m-1$ handlebody $\wh{H}_0$ which is a (non-proper) submanifold of the genus $m-1$ handlebody $H_0^-$: Remove from the pair $(H_0^-,L_0)$ a pair $(B_0,A_0^B)$, where $B_0$ is a 3-ball which meets $\partial H_0^-$ in a disk, and $A_0^B$ is a properly embedded trivial arc in $B_0$.
Denote the closure of the resulting pair by $(\wh{H}_0,\wh{L}_0)$.
Note that the boundary of $\wh{H}_0$ contains a disk $D_0$ meeting $\wh{L}_0$ in two points.
This disk $D_0$ is the closure of the part of $\partial B_0$ which does not meet $\partial H_0^-$.
The pair $(H_0^-,L_0)$ is the boundary connected sum of pairs \[(H_0^-,L_0)=(B_0,A_0^B)\,\natural\, (\wh{H}_0,\wh{L}_0).\]
Remove from $V_0$ and its core an analogous 3-ball-trivial-arc pair to obtain the pair $(\wh{V}_0,A_0^V)$, where $\wh{V}_0$ is a solid torus and $A_0^V$ is a properly embedded arc.
There is an analogous disk on $\partial \wh{V}_0$, which we also call $D_0$, meeting $A_0^V$ in two points.
By gluing the two copies of $D_0$, we obtain the boundary connected sum decomposition \[(H_0,K'_0)=(\wh{V}_0,A_0^V)\,\natural\,(\wh{H}_0,\wh{L}_0).\]

Now fix some knot $K\subset M$, and pick an embedding $\iota:H_0\hookrightarrow M$ which maps $K_0$ homeomorphically to $K$.
Consider the knot $K'\subset M$ which is the image $K'=\iota(K_0')$. 
We call $K'$ an \emph{ambient connected sum of $K$ with the weakly Brunnian link $L\subset S^3$}. 
We may view $K'$ as an ambient connected sum of the knot $K$ with the knot $\iota(L_0)$.


\subsection{$\Gamma$-ambient connected sums and $n$-basings} \label{subsec:construct:Gamma}

While ambient connected sum with a weakly Brunnian link is an operation we may consider for any $K$ in any $M$ with any embedding $\iota$, we now restrict this construction in order to study almost-concordance of homotopically essential knots in aspherical 3-manifolds. 
To this end, suppose $K$ is homotopically essential in aspherical $M$.
Assume for convenience that the basepoint of $M$ is on $\partial\nu K$, so that there is a distinguished based meridian $\mu_0$ for $K$ which is just the boundary of a meridional disk.

Recall from Proposition~\ref{prop:finfty} that the group $\Gamma/\Gamma_2$ is isomorphic to $\Z^\infty$, generated by the meridians of the preimage $p^{-1}(K)$ of $K$ under the universal covering map $p:\wt{M}\to M$.
These meridians correspond to left cosets of $\langle[K]\rangle\leq\pi_1(M)$. 
In addition to $\mu_0$, choose $m-1$ distinct such meridians $\mu_1,\dots,\mu_{m-1}$, and represent them by based loops in $E_K$. 
More specifically, a choice of basepoint $\wt{\ast}$ in the universal cover of $M$ determines a correspondence between components (hence meridians) of $p^{-1}(K)$ and cosets of $\langle [K]\rangle$.
The distinguished meridian $\mu_0$ lifts to a meridian of the component corresponding to the identity coset $\langle [K]\rangle$.
Each other (unbased) meridian $\mu_i$ represents an element of $\Gamma/\Gamma_2$ and corresponds to some coset $g_i\langle[K]\rangle$.
(Later in the paper, we will write $\mu_{g_i}$ for $\mu_i$, where $\mu_g=\mu_h$ if $g\langle[K]\rangle=h\langle[K]\rangle$; see Section~\ref{subsec:zpi:H1}.)
Adding a basepoint path from $\wt{\ast}$ to $g_i\cdot\wt{\ast}$ and taking the image of the resulting based meridian under the covering map $p$ yields a based loop in $E_K$, which we still call $\mu_i$, which is a conjugate of $\mu_0$ by some element of $\pi$ which maps to $g_i$ under the inclusion-induced homomorphism $\pi\twoheadrightarrow\pi_1(M)$.

Represent the classes $[\mu_i]\in\pi$ by disjoint embedded curves along with basepoint paths in $E_K$, which we still call $\mu_i$.
Henceforth we will use $\mu_i$ to denote both the based curve $\mu_i$ and its homotopy class.
We consider ambient connected sums with weakly Brunnian $L$ where the embedding $\iota$ sends the wedge of circles consisting of the 1-handle cores of $\wh{H}_0$ along with basepoint paths to the wedge of circles $\vee_{i=1}^{m-1}\mu_i$; see Figure~\ref{fig:embed}.
\begin{figure}[t]
\centering
    \begin{subfigure}{\textwidth}
    \centering
    \begin{tikzpicture}
        \node[anchor=south west,inner sep=0] (image) at (0,0) {\includesvg[width=0.6\textwidth]{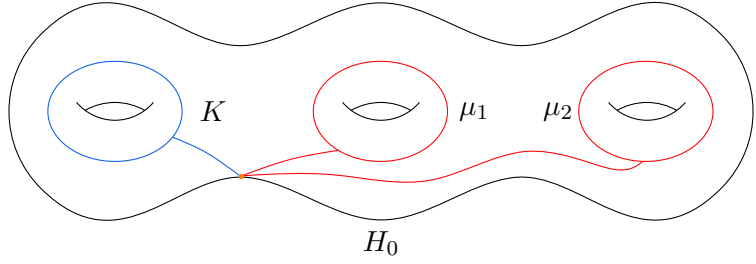} };
        \begin{scope}[x={(image.south east)},y={(image.north west)}]
        \node at (0.5,-0.1) {$H_0$};
        \node at (0.275,0.5) {$K$};
        \node at (0.625,0.5) {$\mu_1$};
        \node at (0.7375,0.5) {$\mu_2$};
        \end{scope}
    \end{tikzpicture}
    \caption{The handlebody $H_0$, labeled with the images of its 1-handle cores.}
    \label{fig:embeda}
    \end{subfigure}
\vspace{1em}
    
    \begin{subfigure}{0.45\textwidth}
    \centering
    \begin{tikzpicture}
        \node[anchor=south west,inner sep=0] (image) at (0,0) {\includesvg[width=\textwidth]{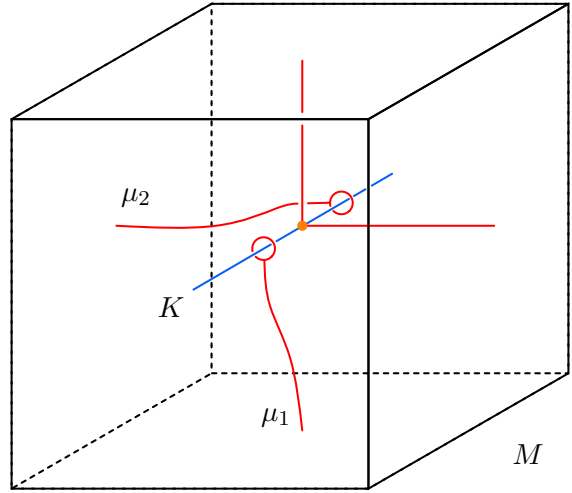}};
        \begin{scope}[x={(image.south east)},y={(image.north west)}]
        \node at (0.925,0.075) {$M$};
        \node at (0.475,0.15) {$\mu_1$};
        \node at (0.225,0.6) {$\mu_2$};
        \node at (0.2875,0.375) {$K$};
        \end{scope}
    \end{tikzpicture}
    \caption{The images of the 1-handle cores of $H_0$ in $M$.}
    \label{fig:embedb}
    \end{subfigure}
\caption{An example of an embedding $\iota:H_0\hookrightarrow M$, where in this case $M$ is a torus bundle over $S^1$. The front and back faces of the cube are identified using the monodromy of the bundle, and the other opposite pairs of faces are glued by the identity.}
\label{fig:embed}
\end{figure}
For conciseness, we call such an ambient connected sum a \emph{$\Gamma$-ambient connected sum}.
Note that any $\Gamma$-ambient connected sum $K'$ of the knot $K$ with a weakly Brunnian link is homotopic to $K$ because $\mu_i\mapsto 1$ under $\pi\twoheadrightarrow\pi_1(M)$ for all $i$.

The primary goal of this section is to determine a sufficient condition for a $\Gamma$-ambient connected sum $K'$ to admit an $n$-basing relative to the original knot $K$ (see Definition~\ref{def:nbasing}).

\begin{proposition} \label{prop:suminducesbasing}
Let $K\subset M$ be a homotopically essential knot in the aspherical 3-manifold $M$. Suppose $L\subset S^3$ is a weakly Brunnian link with vanishing Milnor invariants of lengths $\leq n$.
Then any $\Gamma$-ambient connected sum $K'$ of $K$ with $L$ admits an $n$-basing relative to $K$.
\end{proposition}

\begin{corollary} \label{cor:suminducestheta}
Let $K$, $L$, and $K'$ be as in the previous proposition. Then the invariants $h_n(K')$, $\theta_n(K')$, and $\ol{\mu}_n(K')$ relative to $K$ are defined.
\end{corollary}

Let $L\subset S^3$ be a weakly Brunnian link with distinguished component $L_0$.
Using the notation of Section~\ref{subsec:construct:wBrunnian}, let $\wh{E_L}$ denote the exterior of $\wh{L}_0$ in the genus $m-1$ handlebody $\wh{H}_0$; see Figure~\ref{fig:connsum2}.
In order to prove Proposition~\ref{prop:suminducesbasing}, we first prove that for weakly Brunnian $L\subset S^3$ with vanishing Milnor invariants of lengths $\leq n$ the submanifold $\wh{E_L}$ carries the $n\upth$ lower central quotients of the link group.

\begin{lemma} \label{lemma:ELhat}
Let $L\subset S^3$ be a weakly Brunnian link with distinguished component $L_0$ with vanishing Milnor invariants of lengths $\leq n$, and let $\pi^L=\pi_1(E_L)$. Let $\wh{E_L}$ be as above, and let $\wh{\pi}^L=\pi_1(\wh{E_L})$. Then the inclusion $\wh{E_L}\hookrightarrow E_L$ induces an isomorphism \[\wh{\pi}^L/\wh{\pi}^L_n\xrightarrow{\cong}\pi^L/\pi^L_n.\]
\end{lemma}

\begin{proof}
Recall from Section~\ref{subsec:construct:wBrunnian} the disk $D_0\subset\partial\wh{H}_0$ which meets $\wh{L}_0$ in two points (see Figure~\ref{fig:connsum2}). 
Assume for convenience that the basepoint of $\wh{E_L}$ is also the basepoint of $E_L$ and lies on $\partial D_0$.
Choose a basing of the link $L-L_0$ so that the (based) attaching curves of the 2-handles $h_i$ represent the commutators $[m_i,l_i]$, $2\leq i\leq m-1$, where $m_i$ and $l_i$ denote the meridians and longitudes, respectively, of $L-L_0$, $1\leq i\leq m-1$; see Figure~\ref{fig:weaklyBrunniana}.
We also use $m_i$ and $l_i$ to denote the classes of these (based) curves in $\pi^L$, and we use $\wh{m}_i$ and $\wh{l}_i$ to denote the classes of these meridians and longitudes in $\wh{\pi}^L$.
The intersection $D_0\cap\wh{E_L}$ is the disk $D_0$ with two meridional disks for $L_0$ removed.
Form two distinguished based meridians of the distinguished component $L_0$, denoted $m_0$ and $m_0'$, by basing the boundaries of these two meridional disks so that $[\partial D_0]=m_0(m_0')^{-1}\in\pi^L$; see Figure~\ref{fig:connsum2}. 
Denote the analogous classes in $\wh{\pi}^L$ by $\wh{m}_0$ and $\wh{m}_0'$.

Up to homotopy equivalence, the inclusion $\wh{E_L}\hookrightarrow E_L$ involves attaching $m-1$ disks to $\wh{E_L}$. 
Of these, $m-2$ are attached along the attaching curves for the 2-handles $h_i$ and, when based, represent the commutators $[\wh{m}_i,\wh{l}_i]$, $2\leq i\leq m-1$.
The remaining disk is attached along a curve representing $[\partial D_0]=\wh{m}_0(\wh{m}_0')^{-1}$.
Thus, the inclusion $\wh{E_L}\hookrightarrow E_L$ induces a surjective homomorphism $\wh{\pi}^L\twoheadrightarrow\pi^L$ whose kernel is normally generated by the $[\wh{m}_i,\wh{l}_i]$, $2\leq i\leq m-1$, and the element $[\partial D_0]=\wh{m}_0(\wh{m}_0')^{-1}$.
This homomorphism induces a surjection $\wh{\pi}^L/\wh{\pi}^L_n\twoheadrightarrow\pi^L/\pi^L_n$.
To prove this is an isomorphism, it suffices to show the attaching curves represent elements of $\wh{\pi}^L_n$. 
Note that the curve $\partial D_0$ is the boundary of a genus $m-1$ surface on $\partial\wh{E_L}$, and \[\wh{m}_0(\wh{m}_0')^{-1}=\prod_{i=1}^{m-1}[\wh{m}_i,\wh{l}_i].\]
Thus, $\ker(\wh{\pi}^L\twoheadrightarrow\pi^L)$ is normally generated by $[\wh{m}_i,\wh{l}_i]$, $1\leq i\leq m-1$, and it suffices to prove $[\wh{m}_i,\wh{l}_i]\in\wh{\pi}^L_n$ for $1\leq i\leq m-1$.

We in fact show $\wh{l}_i\in\wh{\pi}^L_n$ using induction on $n$.
Since $L$ has vanishing Milnor invariants of lengths $\leq n$, $l_i\in\pi^L_n$ for all $0\leq i\leq m-1$.
For the base case, it is apparent by construction that $\wh{\pi}^L/\wh{\pi}^L_2\twoheadrightarrow\pi^L/\pi^L_2$ is an isomorphism, so $\wh{l}_i\in\wh{\pi}^L_2$ since $\wh{l}_i\mapsto l_i\in\pi^L_n$.
Now suppose for some $2\leq k<n$ that $\wh{l}_i\in\wh{\pi}^L_k$ for all $1\leq i\leq m-1$.
Because $\wh{l}_i\mapsto l_i\in\pi^L_n$, we may express $\wh{l}_i$ as $\wh{l}_i=c_ik_i$ for some $c_i\in\wh{\pi}^L_n$ and some $k_i\in\ker(\wh{\pi}^L\twoheadrightarrow\pi^L)$. 
But the kernel of $\wh{\pi}^L\twoheadrightarrow\pi^L$ is normally generated by commutators $[\wh{m}_i,\wh{l}_i]$, and all $\wh{l}_i\in\wh{\pi}^L_k$ by assumption.
Thus, $k_i\in\wh{\pi}^L_{k+1}$, and therefore so is $\wh{l}_i$.
This completes the induction.
(Note that we may in fact achieve an isomorphism on $(n+1)\upst$ lower central quotients, but we do not need this.)
\end{proof}

We now leverage Lemma~\ref{lemma:ELhat} to prove Proposition~\ref{prop:suminducesbasing}.

\begin{proof}[Proof of Proposition~\ref{prop:suminducesbasing}]
In addition to the weakly Brunnian link $L$, consider the $m$-component unlink $U$ as a weakly Brunnian link with distinguished component $U_0$.
Recall the genus $m$ handlebody $H_0$ from Section~\ref{subsec:construct:wBrunnian} and the boundary connected sum decompositions \[(H_0,K_0)=(\wh{V}_0,A_0^V)\,\natural\,(\wh{H}_0,\wh{U}_0)\] and \[(H_0,K'_0)=(\wh{V}_0,A_0^V)\,\natural\,(\wh{H}_0,\wh{L}_0),\] following the notation of Section~\ref{subsec:construct:wBrunnian}; see Figure~\ref{fig:connsum2}.
Then the knots $K$ and $K'$ in $M$ are the images $\iota(K_0)$ and $\iota(K'_0)$, respectively.
Define $H=\iota(H_0)$ and $\wh{H}=\iota(\wh{H}_0)$.

Following the notation of Lemma~\ref{lemma:ELhat}, write $\wh{\pi}^U=\pi_1(\wh{E_U})$ and $\wh{\pi}^L=\pi_1(\wh{E_L})$.
Denote the exteriors of $K_0$ and $K_0'$ in the handlebody $H_0$ by $E_0=E_{K_0\subset H_0}$ and $E'_0=E_{K_0'\subset H_0}$, respectively.
Using the decompositions above and notation from Lemma~\ref{lemma:ELhat}, Van Kampen's Theorem implies \[\pi_1(E_0)\cong\frac{\wh{\pi}^U\ast\langle l_0\rangle}{\langle\langle [\wh{m}_0,l_0]=\prod_{i=1}^{m-1}[\wh{m}_i,\wh{l}_i]\rangle\rangle}\] and \[\pi_1(E'_0)\cong\frac{\wh{\pi}^L\ast\langle l_0\rangle}{\langle\langle [\wh{m}_0,l_0]=\prod_{i=1}^{m-1}[\wh{m}_i,\wh{l}_i]\rangle\rangle},\] where $l_0$ is a longitude for $K_0$.
For simplicity, we use the notation $\wh{m}_i$ and $\wh{l}_i$, $1\leq i\leq m-1$, to denote the meridians and longitudes of both $U$ and $L$ in $\wh{E_U}$ and $\wh{E_L}$, respectively.
Further, when we view the $\wh{m}_i$ and $\wh{l}_i$ as elements of $E_0$ or $E'_0$, we will now write $m_i$ and $l_i$.
Note that $m'_0=l_0m_0l_0^{-1}$ in the language of Lemma~\ref{lemma:ELhat}.

Since $L$ has vanishing Milnor invariants of lengths $\leq n$, $L$ admits an $n$-basing $\phi_0:\pi^L/\pi^L_n\xrightarrow{\cong}\pi^U/\pi^U_n$ relative to $U$ (see Lemma 8.1 of \cite{Stees24}).
Applying Lemma~\ref{lemma:ELhat} twice, we obtain an isomorphism $\wh{\pi}^L/\wh{\pi}^L_n\xrightarrow{\cong}\wh{\pi}^U/\wh{\pi}^U_n$ via the composition \[\wh{\pi}^L/\wh{\pi}^L_n\xrightarrow{\cong}\pi^L/\pi^L_n\xrightarrow[\phi_0]{\cong}\pi^U/\pi^U_n\xrightarrow{\cong}\wh{\pi}^U/\wh{\pi}^U_n\] which, by construction, preserves meridians and longitudes.

Let $E$ and $E'$ denote the respective images of $E_0$ and $E'_0$ under the embedding $\iota$.
This embedding induces isomorphisms $\pi_1(E_0)\xrightarrow{\cong}\pi_1(E)$ and $\pi_1(E'_0)\xrightarrow{\cong}\pi_1(E')$, and we will not distinguish between these isomorphic groups.
We will prove that $K'$ admits an $n$-basing relative to $K$ using Van Kampen's Theorem applied to the decompositions $E_{K'}=E_H\cup_{\partial H} E'$ and $E_K=E_H\cup_{\partial H} E$, where $E_H$ is the exterior of the handlebody $H$ in $M$. 
The $n$-basing we define will be the identity on the image of $\pi_1(E_H)$ and will leverage the above isomorphism $\wh{\pi}^L/\wh{\pi}^L_n\xrightarrow{\cong}\wh{\pi}^U/\wh{\pi}^U_n$ to relate $\pi_1(E')$ to $\pi_1(E)$.

Define the auxilliary groups $G$ and $G'$ by \[ G=\frac{\left(\wh{\pi}^U/\wh{\pi}^U_n\right)\ast\langle l_0\rangle}{\langle\langle [\wh{m}_0,l_0]\rangle\rangle}\] and \[ G'=\frac{\left(\wh{\pi}^L/\wh{\pi}^L_n\right)\ast\langle l_0\rangle}{\langle\langle [\wh{m}_0,l_0]\rangle\rangle}.\]
Using the isomorphism $\wh{\pi}^L/\wh{\pi}^L_n\xrightarrow{\cong}\wh{\pi}^U/\wh{\pi}^U_n$ above, we obtain an isomorphism $\wt{\phi}:G'\xrightarrow{\cong} G$.

Observe that we have surjective homomorphisms $\wt{p}:\pi_1(E)\twoheadrightarrow G$ and $\wt{p}\,':\pi_1(E')\twoheadrightarrow G'$: The element $\prod_{i=1}^{m-1}[\wh{m}_i,\wh{l}_i]$ is in the kernel of the composition of surjective homomorphisms \[\wh{\pi}^U\ast\langle l_0\rangle\twoheadrightarrow\left(\wh{\pi}^U/\wh{\pi}^U_n\right)\ast\langle l_0\rangle\twoheadrightarrow G\] since $\wh{l}_i\in\wh{\pi}^U_n$ for all $1\leq i\leq m-1$.
Thus, the surjective homomorphism $\wh{\pi}^U\ast\langle l_0\rangle\twoheadrightarrow G$ factors through $\pi_1(E)$ into two surjections.
In fact, the above argument shows that $G$ is the pushout
\begin{center}
\begin{tikzcd}
\wh{\pi}^U\ast\langle\l_0\rangle \arrow[r,twoheadrightarrow] \arrow[d,twoheadrightarrow] & \pi_1(E) \arrow[d,twoheadrightarrow,"\wt{p}"] \\
\wh{\pi}^U/\wh{\pi}^U_n\ast\langle l_0\rangle \arrow[r,twoheadrightarrow] & G.
\end{tikzcd}
\end{center}
A similar argument yields a surjective homomorphism $\wt{p}\,':\pi_1(E')\twoheadrightarrow G'$ which fits into an analogous pushout diagram.

Now consider the composite homomorphisms $p'\circ\iota_1':\pi_1(E')\to\pi'\twoheadrightarrow\pi'/\Gamma'_n$ and $p\circ\iota_1:\pi_1(E)\to \pi\twoheadrightarrow\pi/\Gamma_n$, where the initial maps are induced by the inclusions $E'\hookrightarrow E_{K'}$ and $E\hookrightarrow E_K$. 
These compositions factor through the surjective homomorphisms $\wt{p}\,'$ and $\wt{p}$ obtained above.
Indeed, the pushout squares for $G'$ and $G$ fit into the diagrams seen in Figure~\ref{fig:pushout}.
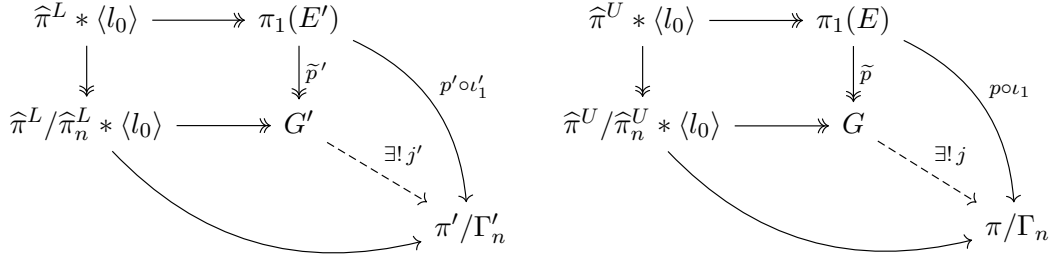
\begin{figure}
\centering
\begin{tabular}{cc}
\begin{tikzcd}
\wh{\pi}^L\ast\langle l_0\rangle \arrow[r,twoheadrightarrow] \arrow[d,twoheadrightarrow] & \pi_1(E') \arrow[d,twoheadrightarrow,"\wt{p}\,'"] \arrow[ddr, bend left=30,"p'\circ\iota'_1"] & \\
\wh{\pi}^L/\wh{\pi}^L_n\ast\langle l_0\rangle \arrow[r,twoheadrightarrow] \arrow[drr, bend right=30] & G' \arrow[dr, dashed, "\exists!\,j'"] & \\
& & \pi'/\Gamma'_n
\end{tikzcd}
&
\begin{tikzcd}
\wh{\pi}^U\ast\langle l_0\rangle \arrow[r,twoheadrightarrow] \arrow[d,twoheadrightarrow] & \pi_1(E) \arrow[d,twoheadrightarrow,"\wt{p}"] \arrow[ddr, bend left=30,"p\circ\iota_1"] & \\
\wh{\pi}^U/\wh{\pi}^U_n\ast\langle l_0\rangle \arrow[r,twoheadrightarrow] \arrow[drr, bend right=30] & G \arrow[dr, dashed, "\exists!\,j"] & \\
& & \pi/\Gamma_n
\end{tikzcd}
\end{tabular}
\caption{Pushout squares for $G'$ and $G$ yield homomorphisms $j':G'\to\pi'/\Gamma'_n$ and $j:G\to\pi/\Gamma_n$.}
\label{fig:pushout}
\end{figure}
The homomorphism at the bottom of each diagram exists because the inclusion-induced homormophisms $\wh{\pi}^L\to\pi'$ and $\wh{\pi}^U\to\pi$ factor through $\Gamma'$ and $\Gamma$, respectively, since $\wh{\pi}^L$ and $\wh{\pi}^U$ are normally generated by meridians which all map trivially to $\pi_1(M)$.
More specifically, the inclusion-induced homomorphism $\wh{\pi}^L\to\pi_1(M)$ factors through $\pi_1(\wh{H})$, which is generated by the cores of its 1-handles which, by construction, map to $\pi_1(M)$ trivially, and similarly for $\wh{\pi}^U$. 
This yields the composite $\wh{\pi}^L/\wh{\pi}^L_n\to\Gamma'/\Gamma'_n\hookrightarrow\pi'/\Gamma'_n$.
A similar argument holds for $\wh{\pi}^U/\wh{\pi}^U_n$.

We thus obtain the commutative diagram over $\pi_1(M)$ in Figure~\ref{fig:pushoutlg}. The left-front and left-back squares of this diagram are pushout squares which commute because all maps are induced by inclusions. The right-front and right-back squares commute by the factorizations described above. The top rectangle commutes by construction: The group $\pi_1(\partial H)$ is generated by $\{m_i,l_i\}_{i=0}^{m-1}$, and the isomorphism $\wt{\phi}$ preserves meridians and longitudes.
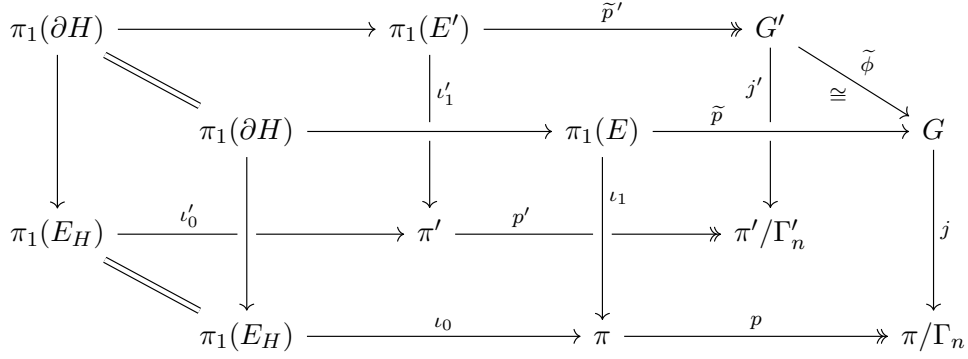
\begin{figure}
\centering
\begin{tikzcd}
\pi_1(\partial H) \arrow[rr] \arrow[dr, equal] \arrow[dd] & & \pi_1(E') \arrow[dd, "\iota_1'" near start] \arrow[rr, twoheadrightarrow, "\wt{p}\,'"] & & G' \arrow[dd, "j'"' near start] \arrow[dr, "\wt{\phi}", "\cong"'] & \\ 
& \pi_1(\partial H) \arrow[rr, crossing over] & & \pi_1(E) \arrow[rr, crossing over, "\wt{p}" near start] & & G \arrow[dd,"j"] \\
\pi_1(E_H) \arrow[dr, equal] \arrow[rr, "\iota_0'" near start] & & \pi' \arrow[rr,twoheadrightarrow, "p'" near start] & & \pi'/\Gamma'_n & \\
& \pi_1(E_H) \arrow[from=uu, crossing over] \arrow[rr, "\iota_0"] & & \pi \arrow[from=uu, crossing over, "\iota_1" near start] \arrow[rr, twoheadrightarrow, "p"] & & \pi/\Gamma_n
\end{tikzcd}
\caption{Obtaining isomorphisms $\phi:\pi'/\Gamma'_n\xrightarrow{\cong}\pi/\Gamma_n$ and $\phi':\pi/\Gamma_n\xrightarrow{\cong}\pi'/\Gamma'_n$.}
\label{fig:pushoutlg}
\end{figure}

Leveraging the left-front and left-back pushout squares of Figure~\ref{fig:pushoutlg}, we first obtain homomorphisms $\phi:\pi'/\Gamma'_n\to\pi/\Gamma_n$ and $\phi':\pi/\Gamma_n\to\pi'/\Gamma'_n$ which fit into the diagram. 
We then show they are mutually inverse isomorphisms. 
Because the left-back square is a pushout, the homomorphisms $p\circ\iota_0:\pi_1(E_H)\to\pi/\Gamma_n$ and $j\circ\wt{\phi}\circ\wt{p}\,':\pi_1(E')\to\pi/\Gamma_n$ induce a homomorphism $\pi'\to\pi/\Gamma_n$ which factors through $\pi'/\Gamma'_n$ because the diagram is over $\pi_1(M)$. 
We wish to show the resulting homomorphism $\phi:\pi'\Gamma'_n\to\pi/\Gamma_n$ fits into the diagram in Figure~\ref{fig:pushoutlg}. 
Since $\phi$ is defined using a pushout diagram, $p\circ\iota_0=\phi\circ p'\circ\iota'_0$, so the bottom rectangle commutes. 
Pushout properties also yield $j\circ\wt{\phi}\circ\wt{p}\,'=\phi\circ p'\circ\iota'_1$. 
This implies $\phi\circ j'\circ\wt{p}\,'=\phi\circ p'\circ\iota'_1=j\circ\wt{\phi}\circ\wt{p}\,'$. 
The surjectivity of $\wt{p}\,'$ implies $j\circ\wt{\phi}=\phi\circ j'$, and the rightmost square commutes. 
By a similar argument which uses $\wt{\phi}^{-1}$, we obtain a homomorphism $\phi':\pi/\Gamma_n\to\pi'/\Gamma'_n$ fitting into the diagram.

We now show that $\phi'=\phi^{-1}$. 
Observe that \[\phi'\circ\phi\circ p'\circ\iota'_0=\phi'\circ p\circ\iota_0=p'\circ\iota'_0\] and \[\phi'\circ\phi\circ p'\circ\iota'_1=\phi'\circ j\circ\wt{\phi}\circ\wt{p}\,'=j'\circ\wt{p}\,'=p'\circ\iota'_1.\]
Since $\phi'\circ\phi\circ p'$ and $p'$ are the same after precomposing with $\iota'_0$ or $\iota'_1$, $\phi'\circ\phi\circ p'=p'$ because the left-back square of Firgure~\ref{fig:pushoutlg} is a pushout.
The surjectivity of $p'$ implies $\phi'\circ\phi=\id_{\pi'/\Gamma'_n}$.
Reversing the respective roles of $\phi$ and $\phi'$ and the the back and front faces of the diagram, we obtain $\phi\circ\phi'=\id_{\pi/\Gamma_n}$.
Thus, $\phi'=\phi^{-1}$.

Finally, observe that $\phi$ satisfies the remaining condition to be an $n$-basing by construction. 
The isomorphism $\wt{\phi}$ maps the elements which become the images of the (based) meridian and longitude of $K'$ in $\pi'/\Gamma'_n$ to the elements which become the images of the (based) meridian and longitude of $K$ in $\pi/\Gamma_n$.
\end{proof}


\subsection{Relating $\theta$-invariants to links in $S^3$} \label{subsec:construct:S3}

By Proposition~\ref{prop:suminducesbasing}, any $\Gamma$-ambient connected sum $K'$ of a homotopically essential knot $K$ in the aspherical 3-manifold $M$ with a weakly Brunnian link $L\subset S^3$ with vanishing Milnor invariants of lengths $\leq n$ has a well-defined $n\upth$ lower central homotopy invariant $h_n(K')$ relative to $K$, hence a well-defined $n\upth$ lower central homology invariant $\theta_n(K')$ and Milnor invariant $\ol{\mu}_n(K')$. 
In this section, we analyze the invariant $\theta_n(K',\phi)$ using the specific $n$-basing $\phi$ constructed in the proof of Proposition~\ref{prop:suminducesbasing} and show that the difference $\theta_n(K',\phi)-\theta_n(K,\id)$ is measured by the (classical) Milnor invariants of $L$. 
Recall from Section~\ref{sec:Milnor} that the $n$-basing $\phi$ induces the homotopy class $h_n(K',\phi)\in[M,X_n(K)]_0$, and by definition $\theta_n(K',\phi)=h_n(K',\phi)_*[M]\in H_3(X_n(K))$. 
We say $\theta_n(K',\phi)$ \emph{vanishes} if $\theta_n(K',\phi)=\theta_n(K,\id)$. 

We first compare $h_n(K',\phi)$ to $h_n(K,\id)$.
By construction and the proof of Proposition~\ref{prop:suminducesbasing}, these maps may be taken to agree on the handlebody exterior $E_H$.
Additionally, $h_n(K',\phi)$ and $h_n(K,\id)$ map $\nu K'$ and $\nu K$, respectively, homeomorphically to $\nu K\subset X_n(K)$. 
In particular, the images of $\nu K'$ and $\nu K$ agree.
Taking these statements together, we see that the images of $h_n(K',\phi)$ and $h_n(K,\id)$ differ only on $\wh{E_{K'}}=\iota(\wh{E_{L}})$ and $\wh{E_K}=\iota(\wh{E_U})$.

Represent $[M]$ as a sum of simplices in two different ways so that the respective images in $X_n(K)$ of the simplices outside $\wh{E_{K'}}$ and $\wh{E_K}$ under the maps $h_n(K',\phi)$ and $h_n(K,\id)$ agree.
This is possible because the maps agree on those domains.
The difference $\theta_n(K',\phi)-\theta_n(K,\id)$ is then represented by the chain \[\sum_{\sigma'\subset \wh{E_{K'}}}(h_n(K',\phi)\circ\sigma')-\sum_{\sigma\subset\wh{E_K}}(h_n(K,\id)\circ\sigma),\] where the sums are taken over simplices which make up $\wh{E_{K'}}$ and $\wh{E_K}$, respectively.

We now examine the restrictions $h_n(K',\phi)|_{\wh{E_{K'}}}:\wh{E_{K'}}\to X_n(K)$ and $h_n(K,\id)|_{\wh{E_K}}:\wh{E_K}\to X_n(K)$ and show (after precomposing with the embedding $\iota$) that they can be extended to produce maps $\ol{h},\ol{h}\,':S^3\to X_n(K)$ such that \[\ol{h}\,'_*[S^3]-\ol{h}_*[S^3]=\theta_n(K',\phi)-\theta_n(K,\id)\in H_3(X_n(K)).\]
Moreover, the maps $\ol{h}$ and $\ol{h}\,'$ are directly related to the Milnor invariants of $U$ and $L$, respectively, via work of Orr \cite{Orr89}.

Since $h_n(K,\id)$ and $h_n(K',\phi)$ are maps of pairs from $(M,K)$ and $(M,K')$ to $(X_n(K),K)$, respectively, the restrictions in question map $\wh{E_K}$ and $\wh{E_{K'}}$ to the (reduced) mapping cylinder $M_{p_n}^\times$, which is a $K(\pi/\Gamma_n,1)$ (see Section~\ref{sec:Milnor}). 
Henceforth we denote this space by $K(\pi/\Gamma_n,1)$.
By construction and as seen in the proof of Proposition~\ref{prop:suminducesbasing}, the inclusion-induced homomorphisms $\wh{\pi}^U\to\pi$ and $\wh{\pi}^L\to\pi'$ factor through $\Gamma$ and $\Gamma'$, respectively.
The restrictions $h(K,\id)|_{\wh{E_K}}$ and $h(K',\phi)|_{\wh{E_{K'}}}$ therefore lift to $K(\Gamma/\Gamma_n,1)$ in such a way that their images agree on $\partial\wh{E_K}$ and $\partial\wh{E_{K'}}$. Precompose these lifts with the embedding $\iota$ to obtain maps $\wh{E_U}\to K(\Gamma/\Gamma_n,1)$ and $\wh{E_L}\to K(\Gamma/\Gamma_n,1)$.

Next, we extend these maps to the link exteriors in $S^3$ and obtain maps $h:E_U\to K(\Gamma/\Gamma_n,1)$ and $h':E_L\to K(\Gamma/\Gamma_n,1)$.
Extend first to $E_{U_0\subset H_0}$ and $E_{L_0\subset H_0}$, the exteriors of the distinguished components $U_0$ and $L_0$ in the genus $m-1$ handlebody $H_0^-$; see Figure~\ref{fig:weaklyBrunnianb}.
Borrowing notation from Sections~\ref{subsec:construct:wBrunnian} and~\ref{subsec:construct:Gamma}, these spaces are formed as boundary connected sums of $E_{A_0^B}$ with $\wh{E_U}$ and $\wh{E_L}$, respectively, taken along the disk $D_0$ with two meridional disks removed, where $E_{A_0^B}$ is the exterior of the arc $A_0^B$ in the 3-ball $B_0$; see Figure~\ref{fig:connsum2}.
Define the extensions to have the same image on $E_{A_0^B}$.
Such extensions are possible because the element $[\partial D_0]=\wh{m}_0(\wh{m}'_0)^{-1}$ maps trivially to $\pi/\Gamma_n$ (see Section~\ref{subsec:construct:Gamma}).
To complete the extensions to $E_U$ and $E_L$, we extend over the $(m-2)$ 2-handles $h_i$ (see Figure~\ref{fig:weaklyBrunniana}) whose attaching maps are nullhomotopic in $K(\Gamma/\Gamma_n,1)$, again defining the extensions to have the same image.
Thus, we have produced extensions $h:E_U\to K(\Gamma/\Gamma_n,1)$ and $h':E_L\to K(\Gamma/\Gamma_n,1)$.

These extensions induce homomorphisms $\pi^U/\pi^U_n\to\Gamma/\Gamma_n$ and $\pi^L/\pi^L_n\to\Gamma/\Gamma_n$ which by construction map $m_i\mapsto\mu_i$ for each $0\leq i\leq m-1$.
Since $U$ and $L$ have vanishing Milnor invariants of lengths $\leq n$, these homomorphisms factor through $F(\mu)/F(\mu)_n$, where $F(\mu)$ is the $m$-generator free group on $\mu=\{\mu_0,\dots,\mu_{m-1}\}$.
Note that $F(\mu)/F(\mu)_n$ includes into $\Gamma/\Gamma_n\cong F(\infty)/F(\infty)_n$, sending generators to generators, so these factorizations consist of an isomorphism (because Milnor's invariants vanish) followed by an inclusion.
Thus, $h$ and $h'$ factor through $K(F/F_n,1)$, where we now write $F=F(\mu)$.
Denote these factorizations by $h=k\circ h_0$ and $h'=k\circ h_0'$, where $h_0:E_U\to K(F/F_n,1)$, $h_0':E_L\to K(F/F_n,1)$, and $k:K(F/F_n,1)\to K(\Gamma/\Gamma_n,1)$.

Define $K_n=K(F/F_n,1)\cup_{(\vee\mu_i)}(\vee_{i=0}^{m-1}D^2)$ as in \cite{Orr89}, the space obtained by attaching $m$ disks $D^2$ to $K(F/F_n,1)$ along the generators of $F$, that is, the images of the meridians of $U$ (or $L$).
As in \cite{Orr89}, we may extend $h_0$ and $h_0'$ to $\ol{h_0}:S^3\to K_n$ and $\ol{h_0}\,':S^3\to K_n$, respectively, by projecting the components of $\nu U$ and $\nu L$ to the respective disks.
Such an extension is possible because the images of the longitudes of $U$ and $L$ are nullhomotopic in $K(F/F_n,1)$.
By construction, we have $\ol{h_0}=h^O_n(U,\id)$ and $\ol{h_0}\,'=h^O_n(L,\phi_0)\in\pi_3(K_n)$, where $h^O_n$ is Orr's invariant from \cite{Orr89}, and where $\phi_0$ is the $n$-basing for $L$ relative to $U$ from which the $n$-basing $\phi$ for $K'$ relative to $K$ is constructed as in the proof of Proposition~\ref{prop:suminducesbasing}.
Orr denotes the invariant $h^O_n$ by $\theta_n$, but we use the notation $h^O_n$ as in \cite{Stees24} to avoid confusion with the lower central homology invariant $\theta_n$ defined in Section~\ref{sec:Milnor} and to indicate the relationship with the lower central homotopy invariant $h_n$.
In fact, Proposition 8.7 of \cite{Stees24} implies that for links in $S^3$ and $n\geq 2$ the invariants $h_n$ relative to the unlink and the invariants $h^O_n$ are equivalent.
We choose to use the invariant $h^O_n$ in this situation because it simplifies the discussion below.
As in \cite{Stees24}, we write $\theta^O_n(L,\phi_0)=h^O_n(L,\phi_0)_*[S^3]=\ol{h_0}_*[S^3]\in H_3(K_n)=H_3(F/F_n)$.
The map $h^O_n(U,\id)$ is nullhomotopic, so $\theta^O_n(U,\id)=\ol{h_0}\,'_*[S^3]=0$.

Analogous to $K_n$, define $K_n^\Gamma=K(\Gamma/\Gamma_n,1)\cup_{(\vee\mu_i)}(\vee_{i=0}^{m-1}D^2)$, where we attach $m$ disks to the images of the same meridians in $K(\Gamma/\Gamma_n,1)$.
The map $k:K(F/F_n,1)\to K(\Gamma/\Gamma_n,1)$ extends to $\ol{k}:K_n\to K_n^\Gamma$ by the identity on the $m$ disks $D^2$, and the inclusion-induced map $i:K(\Gamma/\Gamma_n,1)\to K(\pi/\Gamma_n,1)$ extends to $\ol{i}:K_n^\Gamma\to X_n(K)$ because the meridians in $K_n^\Gamma$ to which the disks are attached are nullhomotopic in $X_n(K)$.
Our desired maps $\ol{h}:S^3\to X_n(K)$ and $\ol{h}\,':S^3\to X_n(K)$ are the compositions $\ol{h}=\ol{i}\circ\ol{k}\circ\ol{h_0}$ and $\ol{h}\,'=\ol{i}\circ\ol{k}\circ\ol{h_0}\,'$.

Consider $\ol{h}\,'_*[S^3]-\ol{h}_*[S^3]\in H_3(X_n(K))$.
Using a suitable decomposition of $[S^3]$ as a sum of simplices, and based on the construction of the maps $\ol{h}$ and $\ol{h}\,'$, we have 
\begin{align*}
\ol{h}\,'_*[S^3]-\ol{h}_*[S^3]&=\sum_{\sigma'\subset \wh{E_{K'}}}(h_n(K',\phi)\circ\sigma')-\sum_{\sigma\subset\wh{E_K}}(h_n(K,\id)\circ\sigma) \\
&=\theta_n(K',\phi)-\theta_n(K,\id)\in H_3(\pi/\Gamma_n)\leq H_3(X_n(K)).
\end{align*}
Note that we view $H_3(\pi/\Gamma_n)$ as a subgroup of $H_3(X_n(K))$ in light of Proposition~\ref{lem:H3}.
On the other hand,
\begin{align*}
\ol{h}\,'_*[S^3]-\ol{h}_*[S^3]&=\ol{i}_*\ol{k}_*\left(\theta^O_n(L,\phi_0)-\theta^O_n(U,\id)\right) \\
&=\ol{i}_*\ol{k}_*\left(\theta^O_n(L,\phi_0)\right),
\end{align*}
where $\theta^O_n(L,\phi_0)\in H_3(K_n)=H_3(F/F_n)$.
Thus, the difference of the $\theta$-invariants of $K$ and $K'$ is measured by the Milnor invariants of the link $L\subset S^3$.
We have proved the following.

\begin{proposition} \label{prop:difference}
Let $K'$ be a $\Gamma$-ambient connected sum of a homotopically essential knot $K$ in the aspherical 3-manifold $M$ with a weakly Brunnian link $L\subset S^3$ with vanishing Milnor invariants of lengths $\leq n$. Let $\phi$ be the $n$-basing for $K'$ relative to $K$ constructed in the proof of Proposition~\ref{prop:suminducesbasing}. Then \[(i_*\circ k_*)\left(\theta^O_n(L,\phi_0)\right)=\theta_n(K',\phi)-\theta_n(K,\id)\in H_3(\pi/\Gamma_n),\] where $\theta^O_n$ is Orr's invariant, and $k_*$ and $i_*$ are induced by the inclusions $F/F_n\hookrightarrow\Gamma/\Gamma_n$ and $\Gamma/\Gamma_n\hookrightarrow\pi/\Gamma_n$, respectively.
\end{proposition}

\noindent We use $k_*$ and $i_*$ instead of $\ol{k}_*$ and $\ol{i}_*$ because of the following commutative diagram.
The first two vertical arrows are canonical isomorphisms given by Mayer-Vietoris sequences.

\begin{center}
\begin{tikzcd}
H_3(F/F_n) \arrow[d,equal] \arrow[r,"k_*"] & H_3(\Gamma/\Gamma_n) \arrow[d,equal] \arrow[r, "i_*"] & H_3(\pi/\Gamma_n) \arrow[d, hookrightarrow] \\
H_3(K_n) \arrow[r,"\ol{k}_*"'] & H_3(K_n^\Gamma) \arrow[r,"\ol{i}_*"'] & H_3(X_n(K))
\end{tikzcd}
\end{center}

We have thus reduced the problem of constructing a family of knots which are not almost-concordant to $K$ to finding a family of links in $S^3$ with vanishing Milnor invariants of lengths $\leq n$ whose length $n+1$ Milnor invariants survive and are distinguished in the images of inclusion-induced homomorphisms $H_3(F/F_n)\to H_3(\pi/\Gamma_n)$. 


\section{Proofs of main results} \label{sec:main}

We now turn to the main results of this paper.
In Section~\ref{subsec:main:family}, we prove Theorem~\ref{thm:family}, which applies the $\theta$-invariants to obstruct pairwise almost-concordance among families of $\Gamma$-ambient connected sums.
These families are indexed by groups $\ol{C}_n$ which are quotients of the images of the homomorphisms $i_*:H_3(\Gamma/\Gamma_n)\to H_3(\pi/\Gamma_n)$ which appeared at the end of Section~\ref{subsec:construct:S3}.
The following sections,~\ref{subsec:main:large} and~\ref{subsec:main:ex}, place some restrictions on $M$ and $K$ in order to analyze the groups $\ol{C}_n$ in detail.
In Section~\ref{subsec:main:large}, we prove Theorem~\ref{thm:infinite}, which supplies conditions under which the groups $\ol{C}_n$ are, in some sense, as large as possible.
We then apply Theorem~\ref{thm:infinite} to a specific family of examples and prove Corollary~\ref{cor:examplecor} in Section~\ref{subsec:main:ex}.

The results in this section, particularly in Sections~\ref{subsec:main:large} and~\ref{subsec:main:ex}, rely heavily on an analysis of the $\Z[\pi_1(M)]$-module structure on $H_*(\Gamma/\Gamma_n)$.
We postpone this analysis until Section~\ref{sec:zpi}, referencing results from that section as needed.


\subsection{Families of knots with distinct $\theta$-invariants} \label{subsec:main:family}

The purpose of this section is to prove the following main structural result of the paper.

\family

\noindent Theorem~\ref{thm:family} answers the \hyperref[conj:ac]{Almost-Concordance Conjecture} for the class $x$ in the 3-manifold $M$ if the set $\cup_{n\geq 2}(\ol{C}_n-\{0\})$ is infinite.
In Section~\ref{subsec:main:large}, we will analyze the $\ol{C}_n$ and provide sufficient conditions under which these groups are, in some sense, as large as possible.
In particular, each $\ol{C}_n$ is infinite under these conditions.

\begin{proof}
Fix a knot $K$ in the aspherical 3-manifold $M$ representing the nontrivial class $x\in[S^1,M]$.
For convenience, assume the basepoint of $M$ is on $\partial\nu K$, and base $K$ via a straight line segment from $\ast$ to $K$ in a $D^2$ fiber of $\nu K$, so that $[K]\in\pi_1(M)$ is well-defined.
We first describe how to construct a knot $K^n_\alpha\subset M$ given $\alpha\in \ol{C}_n$.
The knot $K^n_\alpha$ will be a $\Gamma$-ambient connected sum of $K$ with a Brunnian link $L^n_\alpha\subset S^3$.
Following notation of \cite{Orr89}, let \[C_n=\coker\left(H_3(\Gamma/\Gamma_{n+1})\to H_3(\Gamma/\Gamma_n)\right).\]
The homomorphism $H_3(\Gamma/\Gamma_n)\to H_3(\pi/\Gamma_n)$ induces a surjective homomorphism $C_n\twoheadrightarrow\ol{C}_n$.
Lift $\alpha$ to some class $\theta\in C_n$.

Recall from Proposition~\ref{prop:finfty} that $\Gamma/\Gamma_n\cong F(\infty)/F(\infty)_n$ generated by meridians of the preimage $p^{-1}(K)$ of $K$ under the universal covering map $p:\wt{M}\to M$.
These meridians are in bijective correspondence with left cosets of $\langle[K]\rangle\leq \pi_1(M)$, and this correspondence is determined by a choice of basepoint in $\wt{M}$ and, for $n>2$, choices of basepoint paths, as discussed in Section~\ref{subsec:construct:Gamma} and in more detail in Section~\ref{subsec:zpi:H1}.
Denote the (unbased) meridian corresponding to $g\langle[K]\rangle$ by $\mu_g$, where we consider $\mu_g=\mu_h$ if $g\langle[K]\rangle=h\langle[K]\rangle$.
We discuss the $\Z[\pi_1(M)]$-module structure on these meridians in more detail in Section~\ref{subsec:zpi:H1}.
In short, there is a $\pi_1(M)$-action on $H_k(\Gamma/\Gamma_n)$ for all $k$, and in particular on $H_1(\Gamma/\Gamma_n)=\Gamma/\Gamma_2$, induced by the short exact sequence \[1\to\Gamma/\Gamma_n\hookrightarrow\pi/\Gamma_n\twoheadrightarrow\pi_1(M)\to 1.\]
In Proposition~\ref{prop:cosets}, we show that this action corresponds to the action on cosets of $\langle[K]\rangle$ by left multiplication.

Now fix a representative for each coset of $\langle[K]\rangle$, and fix a basing for each corresponding meridian.
For convenience, we choose $1\in\pi_1(M)$ to represent the identity coset and the boundary of a meridional disk for the corresponding component of $p^{-1}(K)$ to be the based meridian.
For a set $\mathbf{g}=\{g_1,\dots,g_m\}$ of elements of $\pi_1(M)$ representing distinct cosets of $\langle[K]\rangle$, let $\mu_\mathbf{g}=\{\mu_{g_1},\dots,\mu_{g_m}\}$ denote the corresponding set of (based or unbased) meridians.
In Proposition~\ref{prop:colimC}, we prove that \[C_n=\colim_{\mathbf{g}}C_n^\mathbf{g},\] where $C_n^\mathbf{g}=\coker\left(H_3(F(\mu_\mathbf{g})/F(\mu_\mathbf{g})_{n+1})\to H_3(F(\mu_\mathbf{g})/F(\mu_\mathbf{g})_n)\right)$, and where $F(\mu_\mathbf{g})$ is the free group on the based meridians $\mu_\mathbf{g}$.
Furthermore, this colimit is over a system of split inclusions.
In Lemma~\ref{lem:actioncn}, we prove that the $\Z[\pi_1(M)]$-module structure on $H_3(\Gamma/\Gamma_n)$ induces a $\Z[\pi_1(M)]$-module structure on $C_n$.
With respect to the colimit structure on $C_n$, the action of $h\in\pi_1(M)$ induces an isomorphism $C_n^\mathbf{g}\to C_n^{h\cdot\mathbf{g}}$, where $h\cdot\mathbf{g}=\{hg_1,\dots,hg_m\}$; see the discussion in Section~\ref{subsec:zpi:action} following the proof of Lemma~\ref{lem:actioncn}.
This isomorphism agrees with the one induced by a ``relabeling isomorphism" $F(\mu_\mathbf{g})\xrightarrow{\cong}F(\mu_{h\cdot\mathbf{g}})$ defined by $\mu_{g_i}\mapsto\mu_{hg_i}$; see Definition~\ref{def:relabel}.
Recall from above that we write $\mu_{g_i}=\mu_{g_j}$ if and only if $g_i\langle[K]\rangle=g_j\langle[K]\rangle$ even though we fixed a coset representative and basepoint path in order to define a based meridian corresponding to each coset.

View $\theta\in C_n$ as the image of some $\theta_\mathbf{g}\in C_n^\mathbf{g}$, $\mathbf{g}=\{g_1,\dots,g_m\}$, under the inclusion $C_n^\mathbf{g}\hookrightarrow C_n$.
As part of our analysis of the groups $\ol{C}_n$ in Section~\ref{subsec:main:large}, we will show that the homomorphism $C_n\to\ol{C}_n$ factors through the module of coinvariants $(C_n)_{\pi_1(M)}$, so we may assume that $g_1=1$.
We combine work of Cochran \cite{Cochran91} and Orr \cite{Orr89} to realize $\theta_\mathbf{g}$ as an invariant of an $m$-component Brunnian link $L^n_\alpha\subset S^3$.
By Theorem 4 of \cite{Orr89}, all elements in $C_n^\mathbf{g}$ are realized by $m$-component links in $S^3$ with vanishing Milnor invariants of lengths $\leq n$.
Conversely, every $m$-component link with vanishing Milnor invariants of lengths $\leq n$ represents some class in $C_n^\mathbf{g}$.
Furthermore, Theorem 13 of \cite{Orr89} states that any two such links represent the same class in $C_n^\mathbf{g}$ if and only if they have the same Milnor invariants of length $n+1$.
Milnor's invariants are invariants of ordered links, and in this setting the links in $S^3$ corresponding to elements of $C_n^\mathbf{g}$ have components which are labeled by distinct elements of $\mathbf{g}$.

Realize $\theta_\mathbf{g}$ by some link $L\subset S^3$ with vanishing Milnor invariants of lengths $\leq n$, that is, obtain $L$ such that $[\theta_n^O(L,\phi_0')]=\theta_\mathbf{g}$, where $\phi_0'$ is some $n$-basing for $L$ relative to the $m$-component unlink.
Consider the collection of length $n+1$ Milnor invariants of $L$.
By Theorem 3.3 of \cite{Cochran91}, there is a link $L^n_\alpha\subset S^3$ with vanishing Milnor invariants of lengths $\leq n$ and the same length $n+1$ invariants as $L$ such that $L^n_\alpha$ is a connected sum of Brunnian links.
In Cochran's notation, $L^n_\alpha\in\langle\mathbb{B}(m,n)\rangle$.
Recall that while connected sum is not a well-defined operation on links, connected sum is well-defined with respect to the first nonvanishing Milnor invariants \cites{Cochran90 ,Orr89}.
Thus, by choosing the bands for the connected sums of Brunnian links carefully, we may replace $L^n_\alpha$ with a Brunnian link, which we still call $L^n_\alpha$, whose length $\leq n+1$ Milnor invariants still agree with those of $L$.
By Theorem 13 of \cite{Orr89}, $L^n_\alpha$ represents the class $\theta_\mathbf{g}\in C_n^\mathbf{g}$, that is, $[\theta_n^O(L^n_\alpha,\phi_0)]=\theta_\mathbf{g}$ for some $n$-basing $\phi_0$ relative to the unlink.
Because $L^n_\alpha$ is a Brunnian link, we may form a $\Gamma$-ambient connected sum of $K$ with $L^n_\alpha$ to produce a knot $K^n_\alpha\subset M$.
The labels on the components of $L^n_\alpha$ determine the correspondence between the meridians of $L^n_\alpha$ and the 1-handles of the handlebody used in the $\Gamma$-ambient connected sum operation; see Section~\ref{subsec:construct:Gamma} and Figure~\ref{fig:embed}.
Arranging $g_1=1$ above ensures that the meridian of the distinguished component of $L^n_\alpha$ (in the sense of Definition~\ref{def:wbrunnian}) is identified with the distinguished meridian (previously called $\mu_0$) of $K^n_\alpha$.

By Proposition~\ref{prop:suminducesbasing}, $K^n_\alpha$ admits an $n$-basing $\phi$ relative to $K$, hence $\theta_n(K^n_\alpha,\phi)$ is defined. By Proposition~\ref{prop:difference}, \[\theta_n^O(L^n_\alpha,\phi_0)\mapsto\theta_n(K^n_\alpha,\phi)-\theta_n(K,\id)\] under the composite $H_3(F(\mu_\mathbf{g})/F(\mu_\mathbf{g})_n)\to H_3(\Gamma/\Gamma_n)\to H_3(\pi/\Gamma_n)$.
This implies $[\theta_n(K^n_\alpha,\phi)-\theta_n(K,\id)]=\alpha\in \ol{C}_n$, since $[\theta_n^O(L,\phi_0)]=\theta_\mathbf{g}\in C_n^\mathbf{g}$ and $\theta_\mathbf{g}\mapsto\alpha$ under the composite $C_n^\mathbf{g}\hookrightarrow C_n\twoheadrightarrow\ol{C}_n$.
Thus, for each $n\geq 2$ and each element $\alpha$ of $\ol{C}_n$, we have produced a knot $K^n_\alpha$ which is homotopic to $K$ (see Section~\ref{subsec:construct:Gamma}) such that $[\theta_n(K^n_\alpha,\phi)-\theta_n(K,\id)]=\alpha\in\ol{C}_n$.

It is apparent by construction that $\{K\}\in\{K^m_\alpha\}\cap\{K^n_\beta\}$ for $m\neq n$.
We will see that this intersection contains only $K$ by showing all other knots in $\{K^n_\alpha\}_{n\geq 2}$ are not pairwise almost-concordant, by which we complete the proof.
For $\alpha\neq\beta\in\ol{C}_n$, we have \[[\theta_n(K^n_\alpha,\phi_\alpha)-\theta_n(K^n_\beta,\phi_\beta)]=[\theta_n(K^n_\alpha,\phi_\alpha)-\theta_n(K,\id)]-[\theta_n(K^n_\beta,\phi_\beta)-\theta_n(K,\id)]=\alpha-\beta\neq 0,\] so $\theta_n(K^n_\alpha,\phi_\alpha)\neq\theta_n(K^n_\beta,\phi_\beta)$.
By Proposition~\ref{prop:thetapair}, the invariant $\theta_n$ is independent of the choice of $n$-basing, so $\theta_n(K^n_\alpha)\neq\theta_n(K^n_\beta)$ and thus $K^n_\alpha$ and $K^n_\beta$ are not almost-concordant.
By construction, for $n>m$ and any $\beta\in\ol{C}_n$, the link $L^n_\beta$ has vanishing Milnor invariants of length $m$.
So $[\theta_m^O(L^n_\beta,\phi_0)]=0\in C_m^\mathbf{g}$.
This implies $[\theta_m(K^n_\beta,\phi)-\theta_m(K,\id)]=0\in\ol{C}_m$.
Thus, $K^n_\beta$ is not almost-concordant to any $K^m_\alpha$ unless $K^n_\beta=K^m_\alpha=K$.
\end{proof}


\subsection{Cases where the family $\{K^n_\alpha\}$ is large} \label{subsec:main:large}

Theorem~\ref{thm:family} does not make any claim on the sizes of the groups $\ol{C}_n$. 
We expect that $\cup_{n\geq 2}(\ol{C}_n-\{0\})$, which is in bijective correspondence with the family of knots constructed in Theorem~\ref{thm:family}, is often (and possibly always) infinite; however, analyzing the $\ol{C}_n$ without placing any additional restrictions on the 3-manifold $M$ and the fixed  knot $K$ seems difficult.
As we will see in Section~\ref{sec:zpi}, the $\pi_1(M)$-action on the meridians $\{\mu_g\}$ exerts a large influence on the $\Z[\pi_1(M)]$-module structures of $H_2(\Gamma/\Gamma_n)$ and $H_3(\Gamma/\Gamma_n)$, both of which play an important role in understanding the $\ol{C}_n$.
This action can vary widely depending on $M$ and the class $[K]\in[S^1,M]$.

In this section, we impose some further restrictions on the 3-manifolds and knots we consider in order to provide instances where the groups $\ol{C}_n$ are, in some sense, as large as possible.
One can think of the following Theorem~\ref{thm:infinite} as providing sufficient conditions under which Milnor's invariants of links in $S^3$ ``embed" as invariants of almost-concordance of knots in the 3-manifold $M$ homotopic to the fixed knot $K$.

\infinite

\noindent Note that $\lim_{n\to\infty}\mathcal{M}(n)=\infty$.
The number of linearly independent Milnor invariants of length $k$ for $m$-component links grows on the order $m^{k+1}/k^2$ \cite{Orr89}.
Thus, $\mathcal{M}(n)$ grows on order at least \[\frac{n^{n+1}}{n!n^2}=\frac{n^{n-1}}{n!}.\]

\begin{proof}
We study the homomorphism $H_3(\Gamma/\Gamma_n)\to H_3(\pi/\Gamma_n)$ involved in the definition of $\ol{C}_n$.
Consider the Lyndon-Hochschild-Serre spectral sequence associated to the short exact sequence of groups \[1\to\Gamma/\Gamma_n\hookrightarrow\pi/\Gamma_n\twoheadrightarrow\pi_1(M)\to 1.\]
This is a homology spectral sequence with \[E^2_{p,q}\cong H_p(\pi_1(M);H_q(\Gamma/\Gamma_n))\implies H_{p+q}(\pi/\Gamma_n).\]
The short exact sequence of groups yields a $\pi_1(M)$-action on the homology of $\Gamma/\Gamma_n$, hence a $\Z[\pi_1(M)]$-module structure on $H_*(\Gamma/\Gamma_n)$.
The coefficients on the $E^2$ page of the spectral sequence are twisted and defined using this $\pi_1(M)$-action.
Because $M$ is aspherical, $H_p(\pi_1(M);H_q(\Gamma/\Gamma_n))=H_p(M;H_q(\Gamma/\Gamma_n))$, and therefore $E^2_{p,q}=0$ for $p>3$.

The homomorphism $H_3(\Gamma/\Gamma_n)\to H_3(\pi/\Gamma_n)$ factors as \[H_3(\Gamma/\Gamma_n)\to H_0(M;H_3(\Gamma/\Gamma_n))\to H_3(\pi/\Gamma_n),\] the composition of the canonical homomorphism from $H_3(\Gamma/\Gamma_n)$ to its group of $\pi_1(M)$-coinvariants $H_0(M;H_3(\Gamma/\Gamma_n))=H_3(\Gamma/\Gamma_n)_{\pi_1(M)}$ with an edge homomorphism from the spectral sequence (see, for instance, \cite{DavisKirk}).
Sometimes the above composite is called an edge homomorphism, but we reserve this term for the second map to distinguish it from the first.
We study the homomorphisms in this factorization separately, beginning with the edge homomorphism.

Consider the $E^2$ page of the spectral sequence, the relevant parts of which are pictured in Figure~\ref{fig:spectralsequence}.
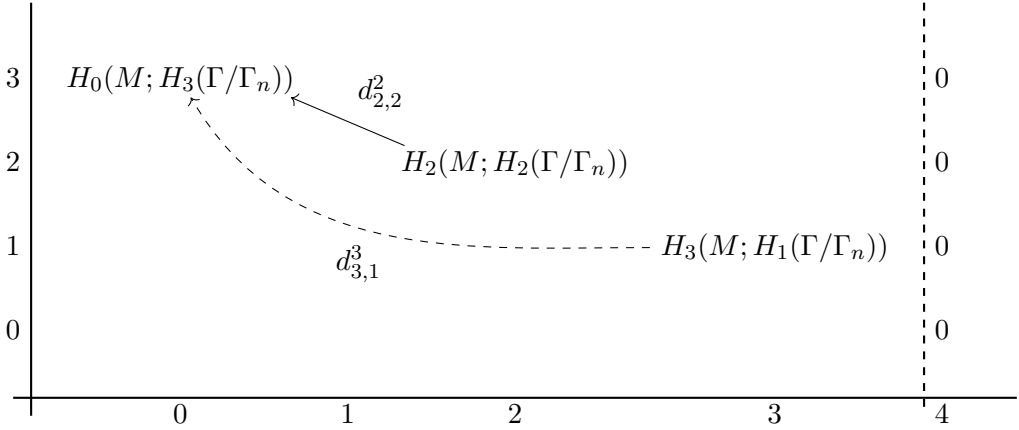
\begin{figure}
\centering
\begin{tikzpicture}
\matrix (m) [matrix of math nodes,nodes in empty cells,nodes={minimum width=5ex,minimum height=5ex,outer sep=-5pt},column sep=1ex,row sep=1ex]{
& & & & & \\
3 & H_0(M;H_3(\Gamma/\Gamma_n)) & & & & 0\\
2 & & & H_2(M;H_2(\Gamma/\Gamma_n)) & & 0\\
1 & & & & H_3(M;H_1(\Gamma/\Gamma_n)) & 0\\
0 & & & & & 0 \\
\quad\strut &   0  &  1  &  2  & 3 & 4 & \strut \\};
\draw[thick] (m-1-1.east) -- (m-6-1.east) ;
\draw[thick] (m-6-1.north) -- (m-6-7.north) ;
\draw[thick,dashed] (m-1-6.west) -- (m-6-6.west) ; 
\draw[->] (m-3-4.north west) to node[above right] {$d^2_{2,2}$} (m-2-2.south east);
\draw[shorten <=0.5em, ->, dashed] (m-4-5) to[out=180,in=300] node[below left] {$d^3_{3,1}$} (m-2-2);
\end{tikzpicture}
    \caption{The $E^2$ page of the Lyndon-Hochschild-Serre spectral sequence.}
    \label{fig:spectralsequence}
\end{figure}
We first observe that $H_3(M;H_1(\Gamma/\Gamma_n))=0$, so that the differential $d^3_{3,1}$ on the $E^3$ page provides no contribution to the computation of $E^\infty_{0,3}$.
We have isomorphisms \[H_3(M;H_1(\Gamma/\Gamma_n))\cong H^0(M;H_1(\Gamma/\Gamma_n))\cong \Hom_{\Z[\pi_1(M)]}\left(H_0(M;\mathbb{Z}[\pi_1(M)]),H_1(\Gamma/\Gamma_n)\right),\] the first following from Poincar\'{e} duality and the second following from the Universal Coefficient Spectral Sequence or standard facts about group cohomology \cite{Brown}.
But $H_0(M;\Z[\pi_1(M)])=\Z[\pi_1(M)]/\langle \{g-1\}_{g\in\pi_1(M)}\rangle\cong\Z$ is a module on which every $g\in\pi_1(M)$ acts trivially.
On the other hand, $H_1(\Gamma/\Gamma_n)=\Gamma/\Gamma_2$ is free abelian, generated by the (unbased) meridians of the preimage  $p^{-1}(K)$ of $K$ under the universal covering map $p:\wt{M}\to M$, which are in bijective correspondence with cosets of $\langle[K]\rangle$.
The $\pi_1(M)$-action on these meridians corresponds to left multiplication on cosets by Proposition~\ref{prop:cosets}.
To show $H_3(M;H_1(\Gamma/\Gamma_n))=0$, it suffices to see that all of $\pi_1(M)$ cannot act trivially on any nontrivial element of $H_1(\Gamma/\Gamma_n)$.

Suppose every $g\in \pi_1(M)$ acts trivially on the nontrivial element $m=\sum c_h\mu_h\in H_1(\Gamma/\Gamma_n)$, $c_h\in\Z$, where the sum is taken over distinct coset representatives and we assume each $c_h\neq 0$.
Because each $g\in\pi_1(M)$ permutes the meridians of $p^{-1}(K)$ and acts trivially on $m$, each $g$ must permute the meridians appearing in the sum.
Because the sum is finite, there exists $k\neq 0$ such that $g^k\cdot \mu_h=\mu_h$ for all $g\in\pi_1(M)$ and all $\mu_h$ appearing in the sum.
As $M$ is aspherical, $\pi_1(M)$ is torsion-free, so $g^k\neq 1$ for $g\neq 1$, $k\neq 0$.
Translating to cosets, we have $k\neq 0$ such that $g^kh\langle[K]\rangle=h\langle[K]\rangle$ for all $g\in\pi_1(M)$ and all $h$ appearing in the sum.
Thus $h^{-1}g^kh\in\langle[K]\rangle$ for all $g$, so for each $g$ we have $h^{-1}g^kh=[K]^i$ for some $i$.
But $h^{-1}g^kh=(h^{-1}gh)^k$, which implies $h^{-1}gh$ commutes with $[K]^i$.
By condition (2), for each $g$ we must have $h^{-1}gh=[K]^j$ for some $j$, hence $g\cdot h\langle[K]\rangle=gh\langle[K]\rangle=h\langle[K]\rangle$ for all $g\in\pi_1(M)$.
This cannot be, as the $\pi_1(M)$-action on meridians of $p^{-1}(K)$, and therefore cosets of $\langle[K]\rangle$, is transitive, and there is more than one coset of $\langle[K]\rangle$ because $\pi_1(M)$ is not cyclic.
Thus, there is no nontrivial $\Z[\pi_1(M)]$-module homomorphism $H_0(M;\Z[\pi_1(M)])\to H_1(\Gamma/\Gamma_n)$, and $H_3(M;H_1(\Gamma/\Gamma_n))=0$.

As $H_3(M;H_1(\Gamma/\Gamma_n))=0$, the edge homomorphism factors as \[H_0(M;H_3(\Gamma/\Gamma_n))\twoheadrightarrow H_0(M;H_3(\Gamma/\Gamma_n))/\im(d^2_{2,2})\hookrightarrow H_3(\pi/\Gamma_n),\] where the differential $d^2_{2,2}$ is pictured in Figure~\ref{fig:spectralsequence}.
By Corollary~\ref{cor:freeH2}, which we prove in Section~\ref{subsec:zpi:basic}, conditions (1)-(3) imply $H_2(\Gamma/\Gamma_n)$ is a free $\Z[\pi_1(M)]$-module.
Thus, $H_2(M;H_2(\Gamma/\Gamma_n))=0$ and so $\im(d^2_{2,2})=0$.
Therefore, we have an injection \[H_3(\Gamma/\Gamma_n)_{\pi_1(M)}=H_0(M;H_3(\Gamma/\Gamma_n))\hookrightarrow H_3(\pi/\Gamma_n).\]

We now turn our attention to the homomorphism $H_3(\Gamma/\Gamma_n)\twoheadrightarrow H_3(\Gamma/\Gamma_n)_{\pi_1(M)}$.
Moving closer to the definition of the set $\ol{C}_n$, we instead study $C_n=\coker(H_3(\Gamma/\Gamma_{n+1})\to H_3(\Gamma/\Gamma_n))$.
Lemma~\ref{lem:actioncn}, which we prove in Section~\ref{subsec:zpi:action}, states that the $\Z[\pi_1(M)]$-module structure on $H_3(\Gamma/\Gamma_n)$ descends to a $\Z[\pi_1(M)]$-module structure on $C_n$.
Thus, we may consider the homomorphism $C_n\twoheadrightarrow (C_n)_{\pi_1(M)}$.

Using the commutative diagram below, we see that $(C_n)_{\pi_1(M)}$ maps isomorphically onto $\ol{C}_n$.
Note that $H_3(\Gamma/\Gamma_{n+1})_{\pi_1(M)}\hookrightarrow H_3(\pi/\Gamma_{n+1})$ is injective by Corollary~\ref{cor:freeH2}, as conditions (1)-(3) do not depend on $n$.
Additionally, \[\im\left(H_3(\Gamma/\Gamma_{n+1})_{\pi_1(M)}\to H_3(\Gamma/\Gamma_n)_{\pi_1(M)}\right)\cong\im\left(H_3(\Gamma/\Gamma_{n+1})\to H_3(\pi/\Gamma_n)\right).\]
\begin{center}
\begin{tikzcd}
H_3(\Gamma/\Gamma_{n+1}) \arrow[r,twoheadrightarrow] \arrow[d] & H_3(\Gamma/\Gamma_{n+1})_{\pi_1(M)} \arrow[r,"\cong"] \arrow[d] & \im\left(H_3(\Gamma/\Gamma_{n+1})\to H_3(\pi/\Gamma_{n+1})\right) \arrow[d] \\
H_3(\Gamma/\Gamma_n) \arrow[r,twoheadrightarrow] \arrow[d,twoheadrightarrow] & H_3(\Gamma/\Gamma_n)_{\pi_1(M)} \arrow[d,twoheadrightarrow] \arrow[r,"\cong"] & \im\left(H_3(\Gamma/\Gamma_n)\to H_3(\pi/\Gamma_n)\right) \arrow[d,twoheadrightarrow] \\
C_n \arrow[r,twoheadrightarrow] & (C_n)_{\pi_1(M)} \arrow[r] & \ol{C}_n
\end{tikzcd}
\end{center}
By Corollary~\ref{cor:freeC}, which we prove in Section~\ref{subsec:zpi:action}, conditions (1)-(3) imply $C_n$ is a free $\Z[\pi_1(M)]$-module of rank at least $\mathcal{M}(n+1)$.
Thus, $\ol{C}_n\cong (C_n)_{\pi_1(M)}$ is a free abelian group whose rank is equal to the rank of $C_n$ as a $\Z[\pi_1(M)]$-module, and therefore at least $\mathcal{M}(n+1)$.
\end{proof}


\subsection{Examples: Surface bundles} \label{subsec:main:ex}

We now provide a family of examples of distinct $M$ and $x\in[S^1,M]$ which satisfy the conditions of Theorem~\ref{thm:infinite}.
As discussed in Section~\ref{sec:intro}, the examples consist of genus $g\geq 1$ surface bundles over $S^1$, where the fixed knot $K$ representing $x$ is a section of the bundle corresponding to a fixed point of the monodromy.
More precisely, \[M=\frac{\Sigma_g\times[0,1]}{(p,1)\sim(\psi(p),0)},\] where $\psi$ is the monodromy.
Without loss of generality, we may assume $\psi(p_0)=p_0$ for some $p_0\in\Sigma_g$.
Let $x$ be the class of the fixed knot $K=(\{p_0\}\times[0,1])/\sim$.

\examplecor

\begin{example} \label{ex:torusbundleanosov}
The torus bundle $M$ with monodromy \[\begin{pmatrix} 2 & 3 \\ 1 & 2 \end{pmatrix}\in SL(2,\Z)\] satisfies the conditions of Corollary~\ref{cor:examplecor}.  
Note that \[\pi_1(M)=\langle a,b,c\,|\,cac^{-1}=a^2b, cbc^{-1}=a^3b^2,[a,b]=1\rangle,\] where conjugation by the generator $c=[K]$ corresponds to the monodromy of the bundle.
As $c=[K]$ normally generates $\pi_1(M)$, there is no regular cover of $M$ to which $K$ lifts.
Thus, typical covering link techniques used to prove other cases of the \hyperref[conj:ac]{Almost-Concordance Conjecture}, as seen in \cites{MillerD,FNOP}, cannot be applied immediately to construct infinitely many almost-concordance classes within the homotopy class $[K]$.
As $[K]\in H_1(M)$ is primitive and infinite order, other techniques such as those seen in \cites{FNOP,NOPP} also cannot be applied immediately in this setting.
\end{example}

\begin{proof}[Proof of Corollary~\ref{cor:examplecor}]
It suffices to show that conditions (1)-(3) of Theorem~\ref{thm:infinite} are satisfied for the 3-manifold $M$ and class $x$ in question.
The fundamental group $\pi_1(M)$ is an extension \[1\to\pi_1(\Sigma_g)\hookrightarrow\pi_1(M)\twoheadrightarrow\Z\to 1,\]
where $x\mapsto 1$ under the surjection $\pi_1(M)\twoheadrightarrow\Z$.
Thus, condition (1) is immediate from the choice of $M$ and $x$.

We next prove condition (2) holds.
Suppose the centralizer of some power $x^j$ of $x$ is not cyclic.
We treat the cases $g=1$ and $g\geq 2$ separately.
Let $g=1$, and let $h\neq 1\in\pi_1(\Sigma_1)$ such that $x^j$ commutes with $hx^k$.
Then $[x^j,h]=[x^j,hx^k]=1$, so $x^jhx^{-j}=h$.
Thus, $\psi_*^j(h)=h$, where $\psi_*$ is the automorphism induced by $\psi$ on $\pi_1(\Sigma_1)=H_1(\Sigma_1)$.
Since $h$ is nontrivial in $H_1(\Sigma_1)$, some eigenvalue of $\psi^j_*$ is equal to 1.
Therefore, since the eigenvalues of $\psi_*$ are assumed to be real and positive, $\psi_*$ has 1 as an eigenvalue.
This contradicts the assumption that $\psi$ is Anosov, which implies no eigenvalue of $\psi_*$ is equal to 1 (see Section 11.4 of \cite{Martelli}).

Now let $g\geq 2$.
We employ results from 3-manifold theory summarized in Theorem 2.5.1 of \cite{AFW}.
Because the centralizer of $x^j$ is non-cyclic and $M$ is a closed 3-manifold, this theorem implies $M$ is Seifert-fibered or there exists a JSJ torus $T$ such that $\pi_1(T)$ is conjugate to the centralizer of $x^j$.
Since $\psi$ is pseudo-Anosov, $M$ is simple, hence has trivial JSJ decomposition, and is not a Seifert 3-manifold (see Section 11.4 of \cite{Martelli}), contradicting the theorem from \cite{AFW}.
Thus, condition (2) holds in all cases.

Condition (3) follows from known results concerning orderability of groups.
Because the eigenvalues of $\psi_*$ are real and positive, a result of Perron-Rolfsen implies that there is a total ordering on $\pi_1(\Sigma_g)$ preserved by $\psi$ \cite{PerronRolfsen}.
The Perron-Rolfsen result in fact implies that $\pi_1(M)$ is biorderable.
The left cosets of $\langle x\rangle\leq\pi_1(M)$ are in bijective correspondence with $\pi_1(\Sigma_g)$, and this correspondence preserves the action of $\pi_1(\Sigma_g)$ by left multiplication.
The ordering on $\pi_1(\Sigma_g)$ preserved by $\psi$ thus induces an ordering on left cosets of $\langle x\rangle$.
To see that this ordering on cosets is invariant under left multiplication by elements in $\pi_1(M)$, it suffices to prove that it is preserved by left multiplication by $x$.
For any coset $g\langle x\rangle$, \[x\cdot g\langle x\rangle=xgx^{-1}\langle x\rangle=\psi_*(g)\langle x\rangle,\] where $\psi_*$ now denotes the induced automorphism of $\pi_1(\Sigma_g)$.
Because $\psi_*$ preserves the ordering on $\pi_1(\Sigma_g)$, it preserves the induced ordering on cosets.
Thus, condition (3) holds.
The result now follows from Theorem~\ref{thm:infinite}.
\end{proof}

To conclude this section, we present one example which contrasts the conditions implied by assumptions (1)-(3) of Theorem~\ref{thm:infinite}.
This example indicates the large extent to which the $\pi_1(M)$-action on the cosets of $\langle[K]\rangle$ exerts an influence on which Milnor invariants of links in $S^3$ ``survive" as almost-concordance invariants of knots in $M$ homotopic to $K$.
In particular, we will see that $C_n$ is not always a free $\Z[\pi_1(M)]$-module.

\begin{example} \label{ex:torusbundleI}
Consider the torus bundle $M$ over $S^1$ whose monodromy is $-I\in SL(2,\Z)$, and let the fixed knot $K\subset M$ be a section of this bundle corresponding to a fixed point, as in Corollary~\ref{cor:examplecor}.
The meridians $\mu_g$ which generate $\Gamma/\Gamma_2$ are in bijective correspondence with $\Z^2$.
The action of $[K]$ on $\Gamma/\Gamma_2$ fixes $\mu_{(0,0)}$ and interchanges $\mu_{(1,0)}$ and $\mu_{(-1,0)}$.
Using notation from the proof of Theorem~\ref{thm:family} and Section~\ref{sec:zpi}, let $\mathbf{g}=\{(0,0),(1,0),(0,1)\}$, so that $\mu_\mathbf{g}=\{\mu_{(0,0)},\mu_{(1,0)},\mu_{(-1,0)}\}$.
Consider $C^2_\mathbf{g}\cong\Z$, which corresponds to Milnor invariants of length 3 for 3-component links labeled with elements of $\mathbf{g}$.
For any ordering on $\mathbf{g}$, the action of $[K]$ on $C_2$ fixes $\im(C^2_\mathbf{g}\hookrightarrow C_2)$ setwise and sends the generator corresponding to the Milnor triple linking invariant $\ol{\mu}(123)$ to its negative, the element corresponding to $\ol{\mu}(132)=-\ol{\mu}(123)$.
Thus, this generator has order at most 2 in $(C_2)_{\pi_1(M)}$, hence has order at most 2 in $\ol{C}_2$.
\end{example}


\section{The $\Z[\pi_1(M)]$-module structure on $H_*(\Gamma/\Gamma_n)$} \label{sec:zpi}

Recall from Section~\ref{subsec:main:large} that the short exact sequence \[1\to\Gamma/\Gamma_n\hookrightarrow\pi/\Gamma_n\twoheadrightarrow\pi_1(M)\to 1\] yields a left action of $\pi_1(M)$ on $H_*(\Gamma/\Gamma_n)$, endowing $H_*(\Gamma/\Gamma_n)$ with the structure of a $\Z[\pi_1(M)]$-module.
As we have seen, understanding this module structure allows one to compute the groups $\ol{C}_n$ from Theorem~\ref{thm:family} and, in particular, prove cases of the \hyperref[conj:ac]{Almost-Concordance Conjecture}, as in Theorem~\ref{thm:infinite} and Corollary~\ref{cor:examplecor}.
In this section, we analyze the $\Z[\pi_1(M)]$-module structures on $H_i(\Gamma/\Gamma_n)$ for $1\leq i\leq 3$ and prove technical statements upon which the results of Section~\ref{sec:main} stand.


\subsection{The $\pi_1(M)$-action on meridians} \label{subsec:zpi:H1}

We first analyze the $\Z[\pi_1(M)]$-module structure on $H_1(\Gamma/\Gamma_n)=\Gamma/\Gamma_2$.
As in Section~\ref{subsec:construct:Gamma}, assume that the basepoint $\ast$ of $E_K$ is on $\partial\nu K$, so there is a canonical choice of element of $\pi_1(M)$ which $K$ represents given by any longitude of $K$ passing through $\ast$.
Recall that the components or, equivalently, meridians of the preimage $p^{-1}(K)$ of $K$ under the universal covering map $p:\wt{M}\to M$ are in bijective correspondence with the left cosets of $\langle[K]\rangle\leq\pi_1(M)$.
We label these (unbased) meridians $\mu_g$, $g\in\pi_1(M)$, where we consider $\mu_g=\mu_h$ if and only if $g\langle[K]\rangle=h\langle[K]\rangle$.
These meridians generate $H_1(\Gamma/\Gamma_n)=\Gamma/\Gamma_2$, which is also the first homology of the $\pi_1(M)$-cover of $E_K$.
In this section, we explicitly describe the $\pi_1(M)$-action on the meridians $\mu_g$, hence the $\Z[\pi_1(M)]$-module structure on $H_1(\Gamma/\Gamma_n)$.
More specifically, we label the meridians $\mu_g$ in a prescribed way and show that the $\pi_1(M)$-action on meridians translates to left multiplication on left cosets of $\langle[K]\rangle$.

The identification \[\{\text{meridians of } p^{-1}(K)\}\longleftrightarrow\{\text{left cosets of } \langle[K]\rangle\}\] is not canonical.
Fix some basepoint $\wt{\ast}$ for the $\pi_1(M)$-cover $\wt{E_K}$ of $E_K$ which is a lift of the basepoint $\ast\in E_K$, and label the corresponding meridian with the identity coset $\langle[K]\rangle$.
Now consider some other meridian based at a translate of $\wt{\ast}$ by some deck transformation $\varphi_g$, $g\in\pi_1(M)$.
This translate of $\wt{\ast}$ is the endpoint of the lift of a loop representing $g\in\pi_1(M)$ to $\wt\ast$.
Label this meridian $\mu_g$, identifying it with the coset $g\langle[K]\rangle$.

\begin{proposition} \label{prop:cosets}
The identification $\mu_g\leftrightarrow g\langle[K]\rangle$ is well-defined. Moreover, the $\pi_1(M)$-action on $H_1(\Gamma/\Gamma_n)$ corresponds under this identification to the action of left multiplication on the left cosets of $\langle[K]\rangle$.
\end{proposition}

\begin{proof}
Write $x=[K]$, and suppose $h\in g\langle x\rangle$. Then $h=gx^k$ for some $k\in\Z$.
The deck transformation $\varphi_h$ decomposes as $\varphi_h=\varphi_g\varphi_{x^k}$.
Deck transformations induce bijections on components of $p^{-1}(K)$, and the deck transformation $\varphi_{x^k}$ fixes the distinguished component containing $\wt{\ast}$.
Equivalently, the lift of a loop representing $gx^k$ first moves from the distinguished component to the one containing the meridian $\mu_g$ and then travels along that component.
Thus, $\mu_h=\mu_g$.
Conversely, two meridians $\mu_g$ and $\mu_h$ lying on the same component of $p^{-1}(K)$ and based at two translates of $\wt{\ast}$ differ by a lift of $x^k$.
So $\mu_h$ corresponds to the lift of a loop representing $gx^k$, and $g\langle x\rangle=h\langle x\rangle$.
Thus, the identification $\mu_g\leftrightarrow g\langle[K]\rangle$ is well-defined.

Recall that the $\pi_1(M)$-action on $H_*(\Gamma/\Gamma_n)$ is induced by conjugation of $\pi/\Gamma_n$ on $\Gamma/\Gamma_n$ (see, for instance, Section III.8 of \cite{Brown}).
In the case of $H_1(\Gamma/\Gamma_n)=\Gamma/\Gamma_2$, this conjugation-induced action is conjugation itself: $\Gamma/\Gamma_2$ is normal in $\pi/\Gamma_2$, which acts by conjugation.
Since $\Gamma/\Gamma_2$ is abelian, this descends to a conjugation action by $\pi_1(M)$.
If we think of $\mu_h$ as an element of $\pi$, basing the meridian of $K$ using a loop representing $h\in\pi_1(M)$, then $g\cdot\mu_h$ can be computed by representing $g$ by some element $\wt{g}\in\pi$ and viewing $\wt{g}\mu_h\wt{g}^{-1}$ as an element of $\Gamma/\Gamma_2$.
In the $\pi_1(M)$-cover $\wt{E_K}$, the resulting meridian is based at the translate of $\wt{\ast}$ obtained by lifting a loop representing $gh$.
Thus, $g\cdot\mu_h=\mu_{gh}$.
\end{proof}


\subsection{A colimit structure on $H_3(\Gamma/\Gamma_n)$} \label{subsec:zpi:colimH}

Our next goal, which we accomplish over Sections~\ref{subsec:zpi:colimH}-\ref{subsec:zpi:action}, is to analyze the $\pi_1(M)$-action on the groups $C_n$ which map onto the groups $\ol{C}_n$ from Theorem~\ref{thm:family}.
As we saw in Section~\ref{subsec:main:large}, the surjection $C_n\twoheadrightarrow\ol{C}_n$ factors through the coinvariants $(C_n)_{\pi_1(M)}$, so understanding this action is crucial for understanding $\ol{C}_n$.
The $\pi_1(M)$-action on the $C_n$ is closely related to the $\Z[\pi_1(M)]$-module structure on $H_3(\Gamma/\Gamma_n)$.
In order to parse these $\pi_1(M)$-actions, we describe colimit structures on $H_3(\Gamma/\Gamma_n)$ and $C_n$ in this section and Section~\ref{subsec:zpi:colimC}, respectively.
Each group is a colimit over a system of split inclusions, and in the case of $C_n$ the $\pi_1(M)$-action behaves nicely with respect to this colimit structure.

For this section through Section~\ref{subsec:zpi:basic}, fix coset representatives for the left cosets of $\langle[K]\rangle\leq\pi_1(M)$.
As in Section~\ref{subsec:zpi:H1}, label the meridians of $p^{-1}(K)$ using these coset representatives.
The resulting set of unbased meridians $\{\mu_g\}$ generate $\Gamma/\Gamma_2$; however, to obtain generators of $\Gamma/\Gamma_n$ for $n>2$ we choose basepoint paths in the $\pi_1(M)$-cover $\wt{E_K}$.
For the remainder of this section through Section~\ref{subsec:zpi:basic}, we fix a basepoint path for each meridian and still call the resulting set of meridians $\{\mu_g\}$.
For $\mathbf{g}=\{g_1,\dots,g_k\}$ a finite set of representatives for distinct cosets, we let $\mu_{\mathbf{g}}=\{\mu_{g_1},\dots,\mu_{g_k}\}$ denote the corresponding set of based meridians.

Consider the directed system consisting of groups $H_\mathbf{g}^n=H_3(F(\mu_{\mathbf{g}})/F(\mu_{\mathbf{g}})_n)$, where $F(\mu_\mathbf{g})$ is the free group on the set of meridians $\mu_\mathbf{g}$, along with inclusion-induced homomorphisms $\psi_{\mathbf{g},\mathbf{h}}:H_\mathbf{g}^n\to H_\mathbf{h}^n$ for $\mathbf{g}\subset\mathbf{h}$.
We will omit the superscript $n$ from our notation when it is understood from context, simply writing $H_\mathbf{g}$.

\begin{proposition} \label{prop:colimH}
The directed system $\left(\{H_\mathbf{g}\},\{\psi_{\mathbf{g},\mathbf{h}}\}\right)$ consists entirely of split inclusions and has colimit \[H_3(\Gamma/\Gamma_n)=\colim_{\mathbf{g}}H_\mathbf{g}.\]
\end{proposition}

\begin{proof}
Suppose $\mathbf{g}\subset\mathbf{h}$.
Then the composition $F(\mu_\mathbf{g})\hookrightarrow F(\mu_\mathbf{h})\twoheadrightarrow F(\mu_\mathbf{g})$ is the identity, where the first map is the inclusion and the second map kills the generators $\mu_{\mathbf{h}-\mathbf{g}}$.
This induces a splitting \[F(\mu_\mathbf{g})/F(\mu_\mathbf{g})_n\hookrightarrow F(\mu_\mathbf{h})/F(\mu_\mathbf{h})_n\twoheadrightarrow F(\mu_\mathbf{g})/F(\mu_\mathbf{g})_n\] which, in turn, induces a splitting $H_\mathbf{g}\hookrightarrow H_\mathbf{h}\twoheadrightarrow H_\mathbf{g}$.
Because homology commutes with colimits, we have \[H_3(\Gamma/\Gamma_n)=H_3\left(\colim_\mathbf{g}F(\mu_\mathbf{g})/F(\mu_\mathbf{g})_n\right)=\colim_\mathbf{g} H_\mathbf{g}.\]
\end{proof}

We next examine the colimit structure on $H_3(\Gamma/\Gamma_n)$ more closely.
In particular, we consider the intersections of the images of $H_\mathbf{g}\hookrightarrow H_{\mathbf{g}\cup\mathbf{h}}$ and $H_\mathbf{h}\hookrightarrow H_{\mathbf{g}\cup\mathbf{h}}$ for arbitrary finite $\mathbf{g},\mathbf{h}$.

\begin{definition} \label{def:initialH}
For $\theta_\mathbf{g}\mapsto\theta$ under $H_\mathbf{g}\hookrightarrow H_3(\Gamma/\Gamma_n)$, we say $\theta_\mathbf{g}$ is \emph{initial} for $\theta$ if there does not exist $\theta_\mathbf{h}\in H_\mathbf{h}$ with $\mathbf{h}\subsetneq \mathbf{g}$ and $\theta_\mathbf{h}\mapsto\theta$.
\end{definition}

\noindent Note that any such $\theta_\mathbf{h}$ as in the above definition satisfies $\psi_{\mathbf{g},\mathbf{h}}(\theta_\mathbf{h})=\theta_\mathbf{g}$ because the colimit structure on $H_3(\Gamma/\Gamma_n)$ consists only of inclusions.

\begin{proposition} \label{prop:initialH}
For each element $\theta\in H_3(\Gamma/\Gamma_n)$, there exists a unique initial element.
\end{proposition}

\begin{proof}
Fix $\theta\in H_3(\Gamma/\Gamma_n)$.
First, note that $\theta$ has at least one initial element.
Because $H_3(\Gamma/\Gamma_n)=\colim_\mathbf{g}H_\mathbf{g}$, there exists $\mathbf{g}$ and $\theta_\mathbf{g}\in H_\mathbf{g}$ with $\theta_\mathbf{g}\mapsto\theta$.
Since $\mathbf{g}$ is a finite set, we are guaranteed to find an initial element mapping to $\theta$ after searching for one inductively based on the size of $\mathbf{g}$.

To show uniqueness, it suffices to show that \[\im(H_\mathbf{g}\hookrightarrow H_{\mathbf{g}\cup\mathbf{h}})\cap\im(H_\mathbf{h}\hookrightarrow H_{\mathbf{g}\cup\mathbf{h}})=\im(H_{\mathbf{g}\cap\mathbf{h}}\hookrightarrow H_{\mathbf{g}\cup\mathbf{h}})\] for all $\mathbf{g},\mathbf{h}$.
For suppose $\theta_\mathbf{g}$ and $\theta_\mathbf{h}$ were both initial for $\theta$.
Because $H_{\mathbf{g}\cup\mathbf{h}}\hookrightarrow H_3(\Gamma/\Gamma_n)$ is injective, $\theta_\mathbf{g}$ and $\theta_\mathbf{h}$ must have a common image $\theta_{\mathbf{g}\cup\mathbf{h}}$ in $H_{\mathbf{g}\cup\mathbf{h}}$.
If the above equality of sets holds, then there exists some $\theta_{\mathbf{g}\cap\mathbf{h}}\in H_{\mathbf{g}\cap\mathbf{h}}$ with $\theta_{\mathbf{g}\cap\mathbf{h}}\mapsto\theta_\mathbf{g}\in H_\mathbf{g}$ and $\theta_{\mathbf{g}\cap\mathbf{h}}\mapsto\theta_\mathbf{h}\in H_\mathbf{h}$.
Since $\theta_\mathbf{g}$ and $\theta_\mathbf{h}$ are initial, we must have $\mathbf{g}=\mathbf{g}\cap\mathbf{h}=\mathbf{h}$.
We obtain $\theta_\mathbf{g}=\theta_{\mathbf{g}\cap\mathbf{h}}=\theta_\mathbf{h}$ as the directed system consists only of inclusions.

The reverse inclusion $\supseteq$ for the desired equality of sets is immediate from the structure of the directed system.
We prove the forward inclusion $\subseteq$ by leveraging an interpretation of the $H_\mathbf{g}$ due to Orr \cite{Orr89}.
We give a geometric interpretation of the splitting $H_\mathbf{k'}\hookrightarrow H_\mathbf{k}\twoheadrightarrow H_\mathbf{k'}$ for $\mathbf{k'}\subset\mathbf{k}$.

Given $\theta_\mathbf{k'}\in H_\mathbf{k'}$, represent $\theta_\mathbf{k'}$ by a $|\mathbf{k'}|$-component based link $L_\mathbf{k'}\subset S^3$, which is possible by Theorem 4 of \cite{Orr89}.
In other words, in the notation of Section~\ref{subsec:construct:S3}, $\theta_\mathbf{k'}=\theta_n^O(L_\mathbf{k'},\phi_0)$ for some $n$-basing $\phi_0$ relative to the $|\mathbf{k'}|$-component unlink.
(Technically, we must choose an ordering on the elements of $\mathbf{k'}$, which produces the ordered based link $L_\mathbf{k'}$ whose components can be thought of as labeled by the elements of $\mathbf{k'}$.
For the purposes of this argument, we only need an ordering on $\mathbf{k}$ which restricts to one on $\mathbf{k'}$, although we could choose a total ordering on all coset representatives if desired.)

The image $\theta_\mathbf{k}$ of $\theta_\mathbf{k'}$ in $H_\mathbf{k}$ can be represented by $L_\mathbf{k'}$ along with a split unlink whose components correspond to elements of $\mathbf{k}-\mathbf{k'}$.
This describes the injection $H_\mathbf{k'}\hookrightarrow H_\mathbf{k}$.
To describe the surjection $H_\mathbf{k}\twoheadrightarrow H_\mathbf{k'}$, consider some $\theta_\mathbf{k}\in H_\mathbf{k}$, and represent it by a $|\mathbf{k}|$-component link $L_\mathbf{k}$.
The surjection sends $\theta_\mathbf{k}$ to an element of $H_\mathbf{k'}$ which is represented by the sublink obtained from $L_\mathbf{k}$ by deleting components corresponding to elements of $\mathbf{k}-\mathbf{k'}$.

Now consider $\theta_\mathbf{k}\in\im(H_\mathbf{k'}\hookrightarrow H_\mathbf{k})$ with $\theta_\mathbf{k}=\psi_{\mathbf{k'},\mathbf{k}}(\theta_\mathbf{k'})$.
Represent $\theta_\mathbf{k}$ by some link $L_\mathbf{k}$.
Since the splitting implies $\theta_\mathbf{k}\mapsto\theta_\mathbf{k'}$ under $H_\mathbf{k}\twoheadrightarrow H_\mathbf{k'}$, the element $\theta_\mathbf{k'}$ is represented by the sublink $L_\mathbf{k'}$ of $L_\mathbf{k}$ whose components correspond to the elements of $\mathbf{k'}\subset\mathbf{k}$.
Considering again that $\theta_\mathbf{k'}\mapsto\theta_\mathbf{k}$ under the inclusion $H_\mathbf{k'}\hookrightarrow H_\mathbf{k}$, we see that $\theta_\mathbf{k}$ is also represented by the link obtained from $L_\mathbf{k}$ by replacing components corresponding to elements of $\mathbf{k}-\mathbf{k'}$ with a split unlink.

To complete the proof, suppose $\theta_{\mathbf{g}\cup\mathbf{h}}=\psi_{\mathbf{g},\mathbf{g}\cup\mathbf{h}}(\theta_\mathbf{g})=\psi_{\mathbf{h},\mathbf{g}\cup\mathbf{h}}(\theta_\mathbf{h})$.
By the above argument, we may represent $\theta_{\mathbf{g}\cup\mathbf{h}}$ by a link $L_{\mathbf{g}\cup\mathbf{h}}$ such that the components which correspond to elements of $(\mathbf{g}-(\mathbf{g}\cap\mathbf{h}))\cup(\mathbf{h}-(\mathbf{g}\cap\mathbf{h}))=(\mathbf{g}\cup\mathbf{h})-(\mathbf{g}\cap\mathbf{h})$ form a split unlink.
Thus, $\theta_{\mathbf{g}\cup\mathbf{h}}\in\im({H_{\mathbf{g}\cap\mathbf{h}}\hookrightarrow H_{\mathbf{g}\cup\mathbf{h}}})$, represented by the sublink of $L_{\mathbf{g}\cup\mathbf{h}}$ whose components correspond to elements of $\mathbf{g}\cap\mathbf{h}$.
\end{proof}


\subsection{A colimit structure on $C_n$} \label{subsec:zpi:colimC}

In this section, we apply many of the ideas from Section~\ref{subsec:zpi:colimH} to study the groups $C_n$.
Recall from Section~\ref{subsec:main:family} that we define \[C_n=\coker\left(H_3(\Gamma/\Gamma_{n+1})\to H_3(\Gamma/\Gamma_n)\right)\] and \[C^n_\mathbf{g}=\coker\left(H_3(F(\mu_\mathbf{g})/F(\mu_\mathbf{g})_{n+1})\to H_3(F(\mu_\mathbf{g})/F(\mu_\mathbf{g})_n)\right)=\coker(H^{n+1}_\mathbf{g}\to H^n_\mathbf{g}).\]
As with the $H^n_\mathbf{g}$, we will simplify our notation and write $C_\mathbf{g}$ for $C^n_\mathbf{g}$ when the superscript $n$ is understood from context.
As in Proposition~\ref{prop:colimH}, we prove that $C_n$ is the colimit of the $C_\mathbf{g}$ over a directed system of split inclusions.

\begin{lemma} \label{lem:colimC}
The directed system $\left(\{H^n_\mathbf{g}\},\{\psi_{\mathbf{g},\mathbf{h}}\}\right)$ induces a directed system $\left(\{C_\mathbf{g}\},\{\ol{\psi}_{\mathbf{g},\mathbf{h}}\}\right)$.
\end{lemma}

\begin{proof}
Suppose $\mathbf{g}\subset\mathbf{h}$.
By Proposition~\ref{prop:colimH}, we have inclusions $\psi^n_{\mathbf{g},\mathbf{h}}:H^n_\mathbf{g}\hookrightarrow H^n_\mathbf{h}$ and $\psi^{n+1}_{\mathbf{g},\mathbf{h}}:H^{n+1}_\mathbf{g}\hookrightarrow H^{n+1}_\mathbf{h}$.
The following square commutes.
\begin{center}
\begin{tikzcd}
H^{n+1}_\mathbf{g} \arrow[r,hookrightarrow] \arrow[d] & H^{n+1}_\mathbf{h} \arrow[d] \\
H^n_\mathbf{g} \arrow[r,hookrightarrow] & H^n_\mathbf{h}
\end{tikzcd}
\end{center}
Thus, the composite $H^n_\mathbf{g}\hookrightarrow H^n_\mathbf{h}\twoheadrightarrow C_\mathbf{h}$ factors through $C_\mathbf{g}$.
We denote the induced homomorphism $C_\mathbf{g}\to C_\mathbf{h}$ by $\ol{\psi}_{\mathbf{g},\mathbf{h}}$.
\end{proof}

\begin{proposition} \label{prop:colimC}
The directed system $\left(\{C_\mathbf{g}\},\{\ol{\psi}_{\mathbf{g},\mathbf{h}}\}\right)$ consists entirely of split inclusions and has colimit \[C_n=\colim_\mathbf{g} C_\mathbf{g}.\]
\end{proposition}

\begin{proof}
Suppose $\mathbf{g}\subset\mathbf{h}$, and consider the following diagram, where both rows are splittings.
\begin{center}
\begin{tikzcd}
H^{n+1}_\mathbf{g} \arrow[r,hookrightarrow] \arrow[d] & H^{n+1}_\mathbf{h} \arrow[r,twoheadrightarrow] \arrow[d] & H^{n+1}_\mathbf{g} \arrow[d] \\
H^n_\mathbf{g} \arrow[r,hookrightarrow] & H^n_\mathbf{h} \arrow[r,twoheadrightarrow] & H^n_\mathbf{g}
\end{tikzcd}
\end{center}
As in the proof of Lemma~\ref{lem:colimC}, this diagram induces maps $C_\mathbf{g}\to C_\mathbf{h}\to C_\mathbf{g}$.
The composite homomorphism is $\id_{C_\mathbf{g}}$.
Thus, $\ol{\psi}_{\mathbf{g},\mathbf{h}}$ is injective and splits.

To conclude the proof, we show that $C_n=\colim_\mathbf{g}C_\mathbf{g}$.
By definition, each $C_\mathbf{g}$ is the colimit of $0\leftarrow H^{n+1}_\mathbf{g}\to H^n_\mathbf{g}$.
We may therefore enlarge the directed system $\{C_\mathbf{g}\}$ to include the directed systems $\{H^{n+1}_\mathbf{g}\}$, $\{H^n_\mathbf{g}\}$, and $\{0_\mathbf{g}\}$ (in which each object is the trivial abelian group), as in the inner part of the diagram in Figure~\ref{fig:colim},
without altering the colimit of the system.
By properties of colimits, and because $H_3(\Gamma/\Gamma_{n})=\colim_\mathbf{g}H^n_\mathbf{g}$ by Proposition~\ref{prop:colimH}, we obtain the outer square of Figure~\ref{fig:colim} which fits into the enlarged directed system.

To show $C_n=\colim_\mathbf{g}C_\mathbf{g}$, it suffices to prove that the outer square of Figure~\ref{fig:colim} is a pushout square, as $C_n$ is the pushout of $0\leftarrow H_3(\Gamma/\Gamma_{n+1}) \to H_3(\Gamma/\Gamma_n)$.
Let $A$ be an abelian group, and suppose there exist homomorphisms $H_3(\Gamma/\Gamma_n)\to A$ and $0\to A$ as in Figure~\ref{fig:colim}.
We show there exists a unique homomorphism $\colim_\mathbf{g}C_\mathbf{g}\to A$ fitting into the diagram.
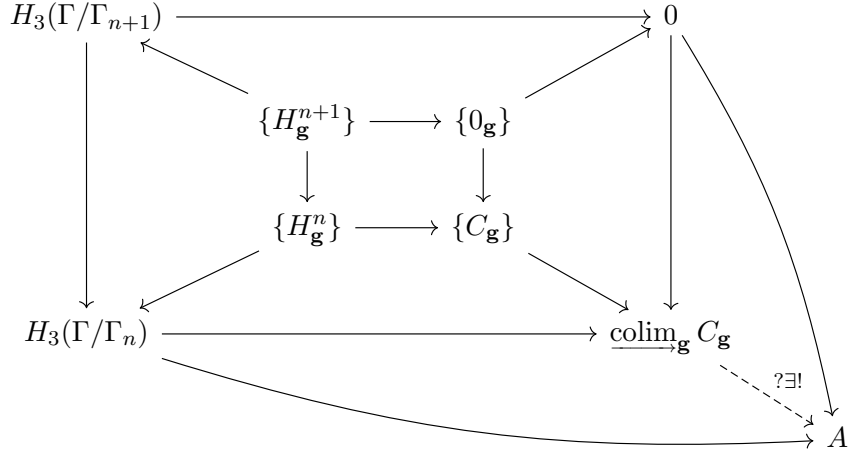
\begin{figure}
\begin{center}
\begin{tikzcd}
H_3(\Gamma/\Gamma_{n+1}) \arrow[rrr] \arrow[ddd] & & & 0 \arrow[ddd] \arrow[ddddr, bend left=10] \\
& \{H^{n+1}_\mathbf{g}\} \arrow[r] \arrow[d] \arrow[ul] & \{0_\mathbf{g}\} \arrow[d] \arrow[ur] \\
& \{H^n_\mathbf{g}\} \arrow[r] \arrow[dl] & \{C_\mathbf{g}\} \arrow[dr] \\
H_3(\Gamma/\Gamma_n) \arrow[rrr] \arrow[drrrr, bend right=10] & & & \colim_\mathbf{g}C_\mathbf{g} \arrow[dr, dashed, "?\exists!"] \\
& & & & A
\end{tikzcd}
\caption{The enlarged directed system from the proof of Proposition~\ref{prop:colimC}.}
\label{fig:colim}
\end{center}
\end{figure}
As the inner squares are a system of pushout squares, we obtain a collection of unique homomorphisms $C_\mathbf{g}\to A$ compatible with the directed systems and the homomorphisms $H_3(\Gamma/\Gamma_n)\to A$ and $0\to A$.
Thus, we obtain the desired homomorphism $\colim_\mathbf{g}C_\mathbf{g}\to A$, proving $C_n=\colim_\mathbf{g}C_\mathbf{g}$.
\end{proof}

Similarly to Definition~\ref{def:initialH}, we define initial elements in the setting of the $C_n$.

\begin{definition} \label{def:initialC}
For $\theta_\mathbf{g}\mapsto\theta$ under $C_\mathbf{g}\hookrightarrow C_n$, we say $\theta_\mathbf{g}$ is \emph{initial} for $\theta$ if there does not exist $\theta_\mathbf{h}\in C_\mathbf{h}$ with $\mathbf{h}\subsetneq \mathbf{g}$ and $\theta_\mathbf{h}\mapsto\theta$.
\end{definition}

\noindent The following proposition is analogous to Proposition~\ref{prop:initialH}.

\begin{proposition} \label{prop:initialC}
For each element $\theta\in C_n$, there exists a unique initial element.
\end{proposition}

\begin{proof}
The proof is similar to the proof of Proposition~\ref{prop:initialH}.
Fix $\theta\in C_n$.
By Proposition~\ref{prop:colimC}, $C_n=\colim_\mathbf{g}C_\mathbf{g}$, so $\theta$ has at least one initial element.

As in the proof of Proposition~\ref{prop:initialH}, to show uniqueness is suffices to show that \[\im(C_\mathbf{g}\hookrightarrow C_{\mathbf{g}\cup\mathbf{h}})\cap\im(C_\mathbf{h}\hookrightarrow C_{\mathbf{g}\cup\mathbf{h}})=\im(C_{\mathbf{g}\cap\mathbf{h}}\hookrightarrow C_{\mathbf{g}\cup\mathbf{h}})\] for all $\mathbf{g},\mathbf{h}$.
The reverse inclusion $\supseteq$ is again immediate, and we prove the forward inclusion $\subseteq$ by leveraging an interpretation of the $C_\mathbf{g}$ due to Orr \cite{Orr89}.

Recall from the proof of Proposition~\ref{prop:initialH} that an element $\theta_\mathbf{k}\in\im(H_\mathbf{k'}\hookrightarrow H_\mathbf{k})$, $\mathbf{k'}\subset\mathbf{k}$, can be represented by a $|\mathbf{k}|$-component link $L_\mathbf{k}$ labeled by elements of $\mathbf{k}$, where the components corresponding to elements of $\mathbf{k}-\mathbf{k'}$ comprise a split unlink.
More precisely, we have $\theta_\mathbf{k}=\theta^O_n(L_\mathbf{k},\phi_0)$, where $\phi_0$ is an $n$-basing of $L_\mathbf{k}$ relative to the unlink.

In this setting, we may use the same argument, relying on the splitting $C_\mathbf{k'}\hookrightarrow C_\mathbf{k}\twoheadrightarrow C_\mathbf{k'}$, to show that an element $\theta_\mathbf{k}\in\im(C_\mathbf{k'}\hookrightarrow C_\mathbf{k})$, $\mathbf{k'}\subset\mathbf{k}$, can be represented by the same kind of link $L_\mathbf{k}$.
More precisely, we have $\theta_\mathbf{k}=[\theta^O_n(L_\mathbf{k},\phi_0)]$, the value of the class $\theta^O_n(L_\mathbf{k},\phi_0)\in H^n_\mathbf{k}$ in $C_\mathbf{k}=\coker(H^{n+1}_\mathbf{k}\to H^n_\mathbf{k})$.
By the same argument as in the proof of Proposition~\ref{prop:initialH}, given $\theta_{\mathbf{g}\cup\mathbf{h}}=\ol{\psi}_{\mathbf{g},\mathbf{g}\cup\mathbf{h}}(\theta_\mathbf{g})=\ol{\psi}_{\mathbf{h},\mathbf{g}\cup\mathbf{h}}(\theta_\mathbf{h})$, we may represent $\theta_{\mathbf{g}\cup\mathbf{h}}$ by a link $L_{\mathbf{g}\cup\mathbf{h}}$ labeled with elements of $\mathbf{g}\cup\mathbf{h}$ such that the components corresponding to $(\mathbf{g}\cup\mathbf{h})-(\mathbf{g}\cap\mathbf{h})$ form a split unlink.
Thus, $\theta_{\mathbf{g}\cup\mathbf{h}}\in\im(C_{\mathbf{g}\cap\mathbf{h}}\hookrightarrow C_{\mathbf{g}\cup\mathbf{h}})$, represented by the sublink of $L_{\mathbf{g}\cup\mathbf{h}}$ whose components are labeled by elements of $\mathbf{g}\cap\mathbf{h}$.
\end{proof}

\begin{corollary} \label{cor:colimC}
The abelian group $C_n$ is isomorphic to $\Z^\infty$, that is, $C_n$ is a countably infinitely generated free abelian group.
\end{corollary}

\begin{proof}
By \cite{Orr89} Lemma 14, the group $C_\mathbf{g}$ is free abelian of finite rank for each $\mathbf{g}$.
There are countably many such $\mathbf{g}$, as the group $\pi_1(M)$ is finitely generated, hence countable, so the set of cosets of $\langle[K]\rangle$ is countable.
By Proposition~\ref{prop:colimC}, $C_n$ is therefore the colimit over countably many finitely generated free abelian groups, and this colimit is over a system of split inclusions.
There are infinitely many nontrivial summands in this colimit as there are infinitely many $\mathbf{g}$ with $|\mathbf{g}|=2$, and no two $C_\mathbf{g}$ with $|\mathbf{g}|=2$ have nontrivial intersection in $C_n$.
Thus, $C_n\cong\Z^\infty$.
\end{proof}


\subsection{Bases for $C_n$} \label{subsec:zpi:bases}

In this section, we build a specific kind of basis for $C_n\cong\Z^\infty$ given a total ordering on the cosets of $\langle[K]\rangle$.
Such an ordering restricts to an ordering on each set $\mathbf{g}$ of coset representatives.
The type of basis we construct is compatible with \emph{order-preserving relabeling isomorphisms}, which we now define.

\begin{definition} \label{def:relabel}
Fix a total ordering on the cosets of $\langle[K]\rangle$, thus fixing an ordering on each set of coset representatives. Let $\mathbf{g}=\{g_1,\dots,g_k\}$ and $\mathbf{h}=\{h_1,\dots,h_k\}$ be sets of coset representatives of the same size with $g_1<\cdots<g_k$ and $h_1<\cdots<h_k$. The \emph{order-preserving relabeling isomorphism} from $C_\mathbf{g}$ to $C_\mathbf{h}$ is the isomorphism $\rho_{\mathbf{g},\mathbf{h}}:C_\mathbf{g}\xrightarrow{\cong} C_\mathbf{h}$ induced by the isomorphism of free groups \[\langle\mu_{g_1},\dots,\mu_{g_k}\rangle\xrightarrow{\cong}\langle\mu_{h_1},\dots,\mu_{h_k}\rangle\] sending $\mu_{g_i}\mapsto\mu_{h_i}$ for all $1\leq i\leq k$.
\end{definition}

\begin{definition} \label{def:basis}
Fix a total ordering on the cosets of $\langle[K]\rangle$, thus fixing an ordering on each set of coset representatives. A basis $\beta$ for the free abelian group $C_n$ is \emph{colimit-compatible} if the following hold.
\begin{itemize}
\item For each set of coset representatives $\mathbf{g}$, the basis $\beta$ is an extension of the image of a basis $\beta_\mathbf{g}$ for $C_\mathbf{g}$ under the split inclusion $C_\mathbf{g}\hookrightarrow C_n$.
\item Every order-preserving relabeling isomorphism $\rho_{\mathbf{g},\mathbf{h}}:C_\mathbf{g}\xrightarrow{\cong} C_\mathbf{h}$ sends the basis $\beta_\mathbf{g}$ to the basis $\beta_\mathbf{h}$.
\end{itemize}
\end{definition}

\begin{proposition} \label{prop:basis}
Given any total ordering on the cosets of $\langle[K]\rangle$, there exists a colimit-compatible basis for $C_n$.
\end{proposition}

\begin{proof}
We will inductively build a colimit-compatible basis for $C_n$, inducting on the size of $\mathbf{g}$.
Choose any set of coset representatives of size 2, say $\mathbf{g_2}=\{g^2_1,g^2_2\}$ with $g^2_1<g^2_2$, and choose any basis $\beta_{\mathbf{g_2}}$ for $C_\mathbf{g_2}$.
Given some other set of coset representatives $\mathbf{h}=\{h_1,h_2\}$ of size 2 with $h_1<h_2$, let the basis $\beta_\mathbf{h}$ for $C_\mathbf{h}$ be the image of $\beta_{\mathbf{g_2}}$ under the order-preserving relabeling isomorphism $\rho_{\mathbf{g_2},\mathbf{h}}:C_\mathbf{g_2}\xrightarrow{\cong}C_\mathbf{h}$.
We may do so for each set of coset representatives of size 2.
Note that for $\mathbf{h}\neq\mathbf{h'}$ with $|\mathbf{h}|=|\mathbf{h'}|=2$, we have \[\im(C_\mathbf{h}\hookrightarrow C_{\mathbf{h}\cup\mathbf{h'}})\cap\im(C_\mathbf{h'}\hookrightarrow C_{\mathbf{h}\cup\mathbf{h'}})=\im(C_{\mathbf{h}\cap\mathbf{h'}}\hookrightarrow C_{\mathbf{h}\cup\mathbf{h'}})=0,\] where the first equality follows from Proposition~\ref{prop:initialC} and the second equality holds because $C_{\mathbf{h}\cap\mathbf{h'}}=0$ as $|\mathbf{h}\cap\mathbf{h'}|\leq 1$.
That this intersection is trivial is necessary so we can guarantee that the basis we eventually define for all of $C_n$ extends the basis $\beta_\mathbf{h}$ for each $C_\mathbf{h}$, $|\mathbf{h}|=2$.
This completes the base case, and we have a basis for $\bigoplus_{|\mathbf{h}|=2}\im(C_\mathbf{h}\hookrightarrow C_n)$.
This basis extends the (empty) basis for $\bigoplus_{|\mathbf{h}|=1}\im(C_\mathbf{h}\hookrightarrow C_n)$ and, by construction, is compatible with order-preserving relabeling isomorphisms $\rho_{\mathbf{h},\mathbf{h'}}:C_\mathbf{h}\xrightarrow{\cong} C_\mathbf{h'}$ for $|\mathbf{h}|=|\mathbf{h'}|=2$.

Now suppose for some $j\geq 2$ we have chosen a basis for $\plus_{|\mathbf{h}|=j}\im(C_\mathbf{h}\hookrightarrow C_n)$ which extends the bases for $\plus_{|\mathbf{h}|=i}\im(C_\mathbf{h}\hookrightarrow C_n)$, $i<j$, and is compatible with order-preserving relabeling isomorphisms $\rho_{\mathbf{h},\mathbf{h'}}$ for $|\mathbf{h}|=|\mathbf{h'}|\leq j$.
We extend to a basis for $\plus_{|\mathbf{h}|=j+1}\im(C_\mathbf{h}\hookrightarrow C_n)$.
Note that we use the notation $\plus$ instead of $\bigoplus$ for $|\mathbf{h}|>2$ as the images of the associated $C_\mathbf{h}$ no longer have trivial intersection as in the case $|\mathbf{h}|=2$.
Again, choose any set of coset representatives of size $j+1$, say $\mathbf{g_{j+1}}=\{g^{j+1}_1,\dots,g^{j+1}_{j+1}\}$ with $g^{j+1}_1<\cdots<g^{j+1}_{j+1}$.
Let \[C_+=\plus_{\mathbf{k}\subsetneq\mathbf{g_{j+1}}}\im(C_\mathbf{k}\hookrightarrow C_\mathbf{g_{j+1}}),\] where the sum is taken over all proper subsets of $\mathbf{g_{j+1}}$.
Note that the sum does not change if we only consider proper subsets $\mathbf{k}$ of size $j$.

We may inductively show that $C_+\leq C_\mathbf{g_{j+1}}$ is a direct summand of the free abelian group $C_\mathbf{g_{j+1}}$ as follows.
Choose some $\mathbf{k}\subsetneq\mathbf{g_{j+1}}$ of size $j$.
By Proposition~\ref{prop:colimC}, $C_\mathbf{k}\hookrightarrow C_\mathbf{g_{j+1}}$ is a split inclusion, hence its image is a direct summand.
Given another proper subset $\mathbf{k'}$ of size $j$, $C_\mathbf{k'}\hookrightarrow C_\mathbf{g_{j+1}}$ is also a split inclusion, and Proposition~\ref{prop:initialC} implies that \[\im(C_\mathbf{k}\hookrightarrow C_\mathbf{g_{j+1}})\cap\im(C_\mathbf{k'}\hookrightarrow C_\mathbf{g_{j+1}})=\im(C_{\mathbf{k}\cap\mathbf{k'}}\hookrightarrow C_\mathbf{g_{j+1}}).\]
By Proposition~\ref{prop:colimC}, $C_{\mathbf{k}\cap\mathbf{k'}}\hookrightarrow C_\mathbf{k'}$ is also a split inclusion, so \[\im(C_\mathbf{k}\hookrightarrow C_\mathbf{g_{j+1}})+\im(C_\mathbf{k'}\hookrightarrow C_\mathbf{g_{j+1}})\leq C_\mathbf{g_{j+1}}\] decomposes as a direct sum of direct summands \[\im(C_\mathbf{k}\hookrightarrow C_\mathbf{g_{j+1}})\oplus \im((C_\mathbf{k'}/C_{\mathbf{k}\cap\mathbf{k'}})\hookrightarrow C_\mathbf{g_{j+1}})\leq C_\mathbf{g_{j+1}},\] hence is itself a direct summand.
We may continue inductively, with this process ending after finitely many steps.

By the inductive hypothesis, we have already chosen a basis for the direct summand $C_+\leq C_\mathbf{g_{j+1}}$.
Any element not in $C_+$ must be initial in $C_\mathbf{g_{j+1}}$, so we have not yet fixed any basis elements for the summand $C_\mathbf{g_{j+1}}/C_+$.
Extend the basis for $C_+$ in any way to a basis $\beta_\mathbf{g_{j+1}}$ for $C_\mathbf{g_{j+1}}$.
Given some other set of coset representatives $\mathbf{h}=\{h_1,\dots,h_{j+1}\}$, again use the order-preserving relabeling isomorphism $\rho_{\mathbf{g_{j+1}},\mathbf{h}}:C_\mathbf{g_{j+1}}\xrightarrow{\cong} C_\mathbf{h}$ to define a basis $\beta_\mathbf{h}$ for $C_\mathbf{h}$.
By hypothesis, the image in $C_n$ of the induced basis on $\rho_{\mathbf{g_{j+1}},\mathbf{h}}(C_+)$ agrees with the basis already constructed for for $\plus_{|\mathbf{k}|=j}\im(C_\mathbf{k}\hookrightarrow C_n)$.
We may do this for all $C_\mathbf{h}$ with $|\mathbf{h}|=j+1$ without conflict as all elements in $C_\mathbf{h}/\rho_{\mathbf{g_{j+1},\mathbf{h}}}(C_+)$ are initial in $C_\mathbf{h}$.
Thus, we have constructed a basis for $\plus_{\mathbf{h}=j+1}\im(C_\mathbf{h}\hookrightarrow C_n)$ which extends the bases for $\plus_{\mathbf{h}=i}\im(C_\mathbf{h}\hookrightarrow C_n)$, $i<j+1$ and is compatible with order-preserving relabeling isomorphisms $\rho_{\mathbf{h},\mathbf{h'}}$ for $|\mathbf{h}|=|\mathbf{h'}|\leq j+1$.
As $C_n=\colim_\mathbf{g}C_\mathbf{g}$, we obtain a basis for $C_n$ which is colimit-compatible.
\end{proof}

We conclude this section by discussing a useful interpretation of colimit-compatible bases for $C_n$ which relies on an interpretation of the $C^n_\mathbf{g}$ due to Orr \cite{Orr89}.
Suppose $\mathbf{g}$ is a set of coset representatives of size $|\mathbf{g}|=k$.
We have already seen in Section~\ref{subsec:zpi:colimC} that an element $\theta_\mathbf{g}\in C^n_\mathbf{g}$ is represented by a $k$-component link in $S^3$ in the sense that there exists a link $L\subset S^3$ which admits an $n$-basing $\phi_0$ relative to the $k$-component unlink such that $\theta=[\theta^O_n(L,\phi_0)]\in C^n_\mathbf{g}$.
By \cite{Orr89} Lemma 14, the group $C^n_\mathbf{g}$ is free abelian of rank $kN_n(k)-N_{n+1}(k)$, where $N_w(k)$ is the number of basic commutators of weight $w$ in the free group on $k$ letters, computed as \[N_w(k)=\frac{1}{w}\sum_{d|w}\phi(d)(k^{w/d}),\] where $\phi$ is the M\"{o}bius function (see \cite{MKS}).
The group $C^n_\mathbf{g}$ represents Milnor invariants of length $n+1$ of $k$-component links in $S^3$ with vanishing Milnor invariants of lengths $\leq n$.
In other words, any basis for $C^n_\mathbf{g}$ consists of a maximal collection of linearly independent Milnor invariants of length $n+1$ for $k$-component links with vanishing Milnor invariants of lengths $\leq n$.
Milnor's invariants are concordance invariants of ordered links, so elements of $C^n_\mathbf{g}$ can be thought of as invariants of links labeled with the elements of $\mathbf{g}$.

In the proof of Proposition~\ref{prop:basis}, we built a colimit-compatible basis for $C_n$ by inductively iterating a process in which we chose a basis for a single $C_\mathbf{g_k}$ with $|\mathbf{g_k}|=k$ and used this basis to define compatible bases for all other $C_\mathbf{h}$, $|\mathbf{h}|=k$.
The construction ensures that for $C_\mathbf{h}$ with $|\mathbf{h}|=k$ the bases chosen are compatible with previously chosen bases for $C_\mathbf{h}$ with $|\mathbf{h}|<k$.
In the language of Milnor's invariants, this inductive process works as follows:
Choose a maximal collection of linearly independent Milnor invariants of length $n+1$ for ordered 2-component links labeled by elements of $\mathbf{g_2}$ with the ordering induced by the ordering on cosets of $\langle[K]\rangle$.
This yields a basis for $C_{\mathbf{g_2}}$.
Then choose the ``same" basis for any $C_\mathbf{h}$, $|\mathbf{h}|=2$, but for links labeled by elements of $\mathbf{h}$.
Next, extend the bases defined on $C_\mathbf{h}$ for all proper subsets $\mathbf{h}\subsetneq\mathbf{g_3}$ to a maximal collection of linearly independent Milnor invariants of length $n+1$ for ordered 3-component links labeled by elements of $\mathbf{g_3}$.
This yields a basis for $C_{\mathbf{g_3}}$.
Choose the ``same" basis for any $C_\mathbf{h}$, $|\mathbf{h}|=3$, but for links labeled by elements of $\mathbf{h}$, and so on.

Although it is not apparent in the proofs of Propositions~\ref{prop:colimC} and~\ref{prop:basis}, this interpretation of colimit-compatible bases for $C_n$ shows that the colimit $C_n=\colim_\mathbf{g}C_\mathbf{g}$ stabilizes at or before $\mathbf{g}$ with $|\mathbf{g}|=n+1$.
Any Milnor invariant of length $n+1$ for any link can involve at most $n+1$ components of the link.
Thus, any element of $\im(C_\mathbf{h}\hookrightarrow C_n)$ with $|\mathbf{h}|>n+1$ is a linear combination of basis elements which are not initial in $C_\mathbf{h}$.
In other words, for $\mathbf{h}$ with $|\mathbf{h}|>n+1$ we have \[C_\mathbf{h}=\plus_{\mathbf{k}\subset\mathbf{h},|\mathbf{k}|=n+1}\im(C_\mathbf{k}\hookrightarrow C_\mathbf{h}),\] so for $k>n+1$ the summand $C_+\leq C_{\mathbf{g_k}}$ seen in the proof of Proposition~\ref{prop:basis} is, in fact, all of $C_\mathbf{g_k}$.


\subsection{Orderability and the $\pi_1(M)$-action on $C_n$} \label{subsec:zpi:action}

This section concludes our analysis of $H_3(\Gamma/\Gamma_n)$ and $C_n$ which began in Section~\ref{subsec:zpi:colimH}.
We will show that the $\Z[\pi_1(M)]$-module structure on $H_3(\Gamma/\Gamma_n)$ descends to a $\Z[\pi_1(M)]$-module structure on $C_n$ and then understand the module structure on $C_n$ given certain conditions on $M$ and the class $x=[K]\in[S^1,M]$ of the fixed knot $K$.

\begin{lemma} \label{lem:actioncn}
The $\Z[\pi_1(M)]$-module structure on $H_3(\Gamma/\Gamma_n)$ induces a $\Z[\pi_1(M)]$-module structure on $C_n$.
\end{lemma}

\begin{proof}
Consider the diagram of short exact sequences of groups
\begin{center}
\begin{tikzcd}
1 \arrow[r] & \Gamma/\Gamma_{n+1} \arrow[r,hookrightarrow] \arrow[d,twoheadrightarrow] & \pi/\Gamma_{n+1} \arrow[r,twoheadrightarrow] \arrow[d,twoheadrightarrow] & \pi_1(M) \arrow[r] \arrow[d,equals] & 1 \\
1 \arrow[r] & \Gamma/\Gamma_n \arrow[r,hookrightarrow] & \pi/\Gamma_n \arrow[r,twoheadrightarrow] & \pi_1(M) \arrow[r] & 1,
\end{tikzcd}
\end{center}
where the rows induce the $\Z[\pi_1(M)]$-module structures on $H_3(\Gamma/\Gamma_{n+1})$ and $H_3(\Gamma/\Gamma_n)$ given by the conjugation actions of $\pi/\Gamma_{n+1}$ and $\pi/\Gamma_n$, respectively, on $\Gamma/\Gamma_{n+1}$ and $\Gamma/\Gamma_n$.
Because conjugation by elements of $\pi/\Gamma_{n+1}$ becomes conjugation by elements of $\pi/\Gamma_n$ when we pass from $\Gamma/\Gamma_{n+1}$ to $\Gamma/\Gamma_n$ via the surjective homomorphism in the diagram and the surjection $\pi/\Gamma_{n+1}\twoheadrightarrow\pi/\Gamma_n$ is a homomorphism over $\pi_1(M)$, the induced map $H_3(\Gamma/\Gamma_{n+1})\to H_3(\Gamma/\Gamma_n)$ is a morphism of $\Z[\pi_1(M)]$-modules.

Now let $\theta\in C_n$ and $g\in\pi_1(M)$, and define $g\cdot\theta$ to be the image of $g\cdot\wt{\theta}$ in $C_n$, where $\wt{\theta}\in H_3(\Gamma/\Gamma_n)$ is any element whose image in $C_n$ is equal to $\theta$.
We show this is a well-defined group action.
Suppose that $\wt{\theta},\wt{\theta}'\mapsto\theta$.
Then $\wt{\theta}-\wt{\theta}'\in\im(H_3(\Gamma/\Gamma_{n+1})\to H_3(\Gamma/\Gamma_n))$.
Let $\eta\in H_3(\Gamma/\Gamma_{n+1})$ be any element whose image in $H_3(\Gamma/\Gamma_n)$ is $\wt{\theta}-\wt{\theta}'$.
Then, because $H_3(\Gamma/\Gamma_{n+1})\to H_3(\Gamma/\Gamma_n)$ is a morphism of $\Z[\pi_1(M)]$-modules, $g\cdot\eta\mapsto g\cdot(\wt{\theta}-\wt{\theta}')=g\cdot\wt{\theta}-g\cdot\wt{\theta}'$.
Thus, the images of $g\cdot\wt{\theta}$ and $g\cdot\wt{\theta}'$ in $C_n$ agree, and $g\cdot\theta$ is well-defined.
That this yields a group action follows from properties of the action of $\pi_1(M)$ on $H_3(\Gamma/\Gamma_n)$.
\end{proof}

We now combine the $\pi_1(M)$-actions on $H_3(\Gamma/\Gamma_n)$ and $C_n$ with the colimit structures given in Sections~\ref{subsec:zpi:colimH} and~\ref{subsec:zpi:colimC}.
Recall from Section~\ref{subsec:zpi:colimH} that $H_3(\Gamma/\Gamma_n)=\colim_\mathbf{g}H^n_\mathbf{g}$, where $H_\mathbf{g}=H^n_\mathbf{g}=H_3(F(\mu_\mathbf{g})/F(\mu_\mathbf{g})_n)$.
We will first study the $\pi_1(M)$-action on the $H_\mathbf{g}$.
If we consider unbased meridians $\mu_\mathbf{h}$ of the preimage $p^{-1}(K)$ of $K$ under the universal covering map $p:\wt{M}\to M$, which correspond to cosets of $\langle[K]\rangle\leq\pi_1(M)$, the action by an element of $\pi/\Gamma_n$ representing $g\in\pi_1(M)$ corresponds to left multiplication on cosets by Proposition~\ref{prop:cosets}.
As seen in Section~\ref{subsec:zpi:H1}, $g\cdot\mu_h=\mu_{gh}$.
So for unbased meridians, $g\cdot\mu_\mathbf{h}=\mu_{g\cdot\mathbf{h}}$, where if $\mathbf{h}=\{h_1,\dots,h_k\}$ we define $g\cdot\mathbf{h}=\{gh_1,\dots,gh_k\}$.
However, in the setting of studying $H_3(\Gamma/\Gamma_n)$, we fix a basing for each generator of $\Gamma/\Gamma_n$, that is, for each meridian of $p^{-1}(K)$.
Thus, the statements above for unbased meridians hold for based meridians only up to conjugation by elements of $\Gamma$.
For each $i$, we have $g\cdot\mu_{h_i}=\gamma_i\mu_{gh_i}\gamma_i^{-1}$ for some $\gamma_i\in\Gamma$.
The action of $g\in\pi_1(M)$ therefore induces an isomorphism \[H_\mathbf{h}=H_3(F(\mu_\mathbf{h})/F(\mu_\mathbf{h})_n)\xrightarrow{\cong}H_3(F(\mu_{g\cdot\mathbf{h}}^\gamma)/F(\mu_{g\cdot\mathbf{h}}^\gamma)_n),\] where $F(\mu_{g\cdot\mathbf{h}}^\gamma)$ is the free group generated by the $\gamma_i\mu_{gh_i}\gamma_i^{-1}$.
See \cite{Brown} Section III.8 for further details.

One motivation for passing from $H_3(\Gamma/\Gamma_n)$ to the cokernel $C_n$ is that the $\pi_1(M)$-action simplifies significantly.
Recall from Section~\ref{subsec:zpi:colimC} that $C_n=\colim_\mathbf{g}C_\mathbf{g}$, where $C_\mathbf{g}=C^n_\mathbf{g}=\coker(H^{n+1}_\mathbf{g}\to H^n_\mathbf{g})$.
As in Section~\ref{subsec:zpi:bases}, we rely on the interpretation of $C_\mathbf{g}$, $|\mathbf{g}|=k$, as Milnor invariants of length $n+1$ for $k$-component links in $S^3$ with vanishing Milnor invariants of lengths $\leq n$ \cite{Orr89}.
Theorem 11 of \cite{Orr89} states that such Milnor invariants are independent of a choice of basing for the link.
In other words, such Milnor invariants do not change if the isomorphism chosen between the $n\upth$ lower central quotients of the link group and those of the unlink is modified by conjugation of individual meridians.
Thus, the isomorphism $H_3(F(\mu_\mathbf{h})/F(\mu_\mathbf{h})_n)\xrightarrow{\cong}H_3(F(\mu_{g\cdot\mathbf{h}}^\gamma)/F(\mu_{g\cdot\mathbf{h}}^\gamma)_n)$ described above which is induced by the action of $g\in\pi_1(M)$ becomes a \emph{relabeling isomorphism} $C_\mathbf{h}\xrightarrow{\cong}C_{g\cdot\mathbf{h}}$ (see Definition~\ref{def:relabel}) when we pass to cokernels.
If the relabeling isomorphisms induced by the $\pi_1(M)$-action are order-preserving for some ordering on left cosets of $\langle[K]\rangle$, then the $\pi_1(M)$-action on $C_n$ will preserve a colimit-compatible basis for $C_n$.
We show that under this assumption $C_n$ is a \emph{permutation module}, that is, a free abelian group with a basis preserved by the $\pi_1(M)$-action (see, for instance, \cite{Brown} Section I.3).

\begin{proposition} \label{prop:permmodule}
If the left cosets of $\langle[K]\rangle\leq\pi_1(M)$ admit a total ordering which is preserved by the $\pi_1(M)$-action by left multiplication, then $C_n$ is a permutation module.
\end{proposition}

\begin{proof}
Suppose the left cosets of $\langle[K]\rangle$ admit such a total ordering, and fix such a total ordering.
By Proposition~\ref{prop:basis}, we may construct a colimit-compatible basis $\beta$ for $C_n$ with respect to the fixed total ordering.
By the above discussion, the $\pi_1(M)$-action on $C_n$ restricts to relabeling isomorphisms on the $C_\mathbf{h}$, where the action of $g\in\pi_1(M)$ induces a relabeling isomorphism $C_\mathbf{h}\xrightarrow{\cong}C_{g\cdot\mathbf{h}}$.
By assumption, these relabeling isomorphisms are order-preserving, sending the basis $\beta_\mathbf{h}$ to the basis $\beta_{g\cdot\mathbf{h}}$, both of which extend to the basis $\beta$.
Thus, since $C_n=\colim_\mathbf{g}C_\mathbf{g}$, the action of any $g\in\pi_1(M)$ induces a bijection $\beta\to\beta$.
By definition, as $C_n$ is a free abelian group with a $\pi_1(M)$-action that permutes the elements of a basis, $C_n$ is a permutation module.
\end{proof}

Under some further assumptions, we show that $C_n$ is a free $\Z[\pi_1(M)]$-module.
Recall from the statement of Theorem~\ref{thm:infinite} that $\mathcal{M}(k)$ denotes the number of linearly independent Milnor invariants of $k$-component links of length $k$ up to relabeling components.

\begin{corollary} \label{cor:freeC}
Suppose conditions (1)-(3) of Theorem~\ref{thm:infinite} are satisfied:
\begin{enumerate}[label=(\arabic*)]
\item $[K]$ is primitive and infinite order in $H_1(M)$.
\item The centralizer of any nontrivial power of $[K]\in\pi_1(M)$ is cyclic.
\item The left cosets of $\langle[K]\rangle\leq\pi_1(M)$ admit a total ordering which is invariant under the $\pi_1(M)$-action by left multiplication.
\end{enumerate}
Then $C_n$ is a free $\Z[\pi_1(M)]$-module of rank at least $\mathcal{M}(n+1)$.
\end{corollary}

\begin{proof}
Assume conditions (1)-(3) hold.
By Proposition~\ref{prop:permmodule}, condition (3) implies that $C_n$ is a permutation module with respect to a colimit-compatible basis $\beta$.
In order to see that $C_n$ is a free module, it suffices to show that the stabilizer of each basis element is trivial (see, for instance, \cite{Brown} Section I.3).

Suppose, for the sake of contradiction, that there exists a basis element $\theta\in\beta$ and a nontrivial element $g\in\pi_1(M)$ such that $g\cdot\theta=\theta$.
By construction, $\theta$ is the image of some initial element $\theta_\mathbf{h}\in C_\mathbf{h}$ which belongs to a basis $\beta_\mathbf{h}$ whose image in $C_n$ extends to $\beta$.
The action of $g$ on $C_\mathbf{h}$ corresponds to the order-preserving relabeling isomorphism $C_\mathbf{h}\xrightarrow{\cong} C_{g\cdot\mathbf{h}}$ and sends $\theta_\mathbf{h}$ to another (initial) basis element since $\beta$ is colimit-compatible.
Because $\theta_\mathbf{h}$ and $g\cdot\theta_\mathbf{h}$ both have image $\theta\in C_n$, we must have $g\cdot\theta_\mathbf{h}=\theta_\mathbf{h}$ and $C_{g\cdot\mathbf{h}}=C_\mathbf{h}$ by Proposition~\ref{prop:initialC}.

Let $x=[K]\in\pi_1(M)$.
Since the action of $g$ is order-preserving on $\mathbf{h}$ and $|\mathbf{h}|\geq 2$, there exist two representatives of distinct cosets $h_1, h_2\in\mathbf{h}$ which are fixed by the action of $g$.
In other words, $gh_1\langle x\rangle=h_1\langle x\rangle$ and $gh_2\langle x\rangle=h_2\langle x\rangle$ by Proposition~\ref{prop:cosets}.
This implies $h_1^{-1}gh_1=x^i$ and $h_2^{-1}gh_2=x^j$ for some nonzero $i,j\in\Z$.
Equivalently, $g=h_1x^ih_1^{-1}=h_2x^jh_2^{-1}$.
Let $h=h_2^{-1}h_1\neq 1$.
Then $hx^ih^{-1}=x^j$, so that $[h,x^i]=x^{j-i}$.

We now leverage conditions (1) and (2).
Note that conditions (1) and (2) combine to imply that the centralizer $C(x^k)$ of any nontivial power of $x$ is $\langle x\rangle$.
If $j\neq i$, then $(j-i)[x]=0\in H_1(M)$ because the commutator $[h,x^i]$ is trivial in $H_1(M)$.
This contradicts condition (1).
On the other hand, if $j=i$, then $[h,x^i]=1$.
This implies $h\in C(x^i)=\langle x\rangle$, so that $h\langle x\rangle=\langle x\rangle$, hence $h_1\langle x\rangle=h_2\langle x\rangle$, a contradiction.
Thus, our assumption that the basis element $\theta\in\beta$ has nontrivial stabilizer is false, and $C_n$ is a free $\Z[\pi_1(M)]$-module.

To see that $C_n$ has rank at least $\mathcal{M}(n+1)$, choose any $\mathbf{h}$ with $|\mathbf{h}|=n+1$.
Then the basis $\beta_\mathbf{h}$ for $C_\mathbf{h}$ consists of linearly independent Milnor invariants of length $n+1$ for links labeled with $\mathbf{h}$.
In particular, $\beta_\mathbf{h}$ contains a subset $\beta_0\subseteq\beta_\mathbf{h}$ consisting of $\mathcal{M}(n+1)$ elements which represent the $\mathcal{M}(n+1)$ distinct classes of linearly independent Milnor invariants modulo relabeling components.
No two elements of $\beta_0$ can be in the same $\pi_1(M)$-orbit since $\pi_1(M)$ acts by order-preserving relabeling isomorphisms which on basis elements correspond to relabeling components involved in individual Milnor invariants.
Thus, $C_n$ has rank at least $\mathcal{M}(n+1)$ as a $\Z[\pi_1(M)]$-module.
\end{proof}


\subsection{Basic commutators and $H_2(\Gamma/\Gamma_n)$} \label{subsec:zpi:basic}

We now turn our attention to $H_2(\Gamma/\Gamma_n)$ and its accompanying $\Z[\pi_1(M)]$-module structure.
In this section, we will show that $H_2(\Gamma/\Gamma_n)$ can be described in terms of \emph{basic commutators} of the meridians $\mu_g$ of the preimage $p^{-1}(K)$ of $K$ under the universal covering $p:\wt{M}\to M$.
Similarly to Section~\ref{subsec:zpi:action}, we will then exhibit conditions under which $H_2(\Gamma/\Gamma_n)$ is a free $\Z[\pi_1(M)]$-module.
As in Sections~\ref{subsec:zpi:colimH}-\ref{subsec:zpi:bases}, our approach is to understand $H_2(\Gamma/\Gamma_n)$ as a colimit of $H_2(F(\mu_\mathbf{g})/F(\mu_\mathbf{g})_n)$. Then, we reinterpret each member of the resulting directed system in terms of basic commutators.
Before describing the colimit structure on $H_2(\Gamma/\Gamma_n)$, we express a canonical identification of $H_2(\Gamma/\Gamma_n)$ with $\Gamma_n/\Gamma_{n+1}$ which will enable our reinterpretation.

\begin{lemma} \label{lem:H2basic}
There is a canonical isomorphism of $\Z[\pi_1(M)]$-modules $H_2(\Gamma/\Gamma_n)\xrightarrow{\cong}\Gamma_n/\Gamma_{n+1}$.
\end{lemma}

\begin{proof}
Consider the diagram of groups in Figure~\ref{fig:groups} with exact rows and columns.
\begin{figure} 
\centering
\begin{tikzcd}
& 1 \arrow[d] & 1 \arrow[d] & 1 \arrow[d] & \\
1 \arrow[r] & \Gamma_n \arrow[r,hookrightarrow] \arrow[d,equal] & \Gamma \arrow[r,twoheadrightarrow] \arrow[d,hookrightarrow] & \Gamma/\Gamma_n \arrow[r] \arrow[d,hookrightarrow] & 1 \\
1 \arrow[r] & \Gamma_n \arrow[r,hookrightarrow] \arrow[d] & \pi \arrow[r,twoheadrightarrow] \arrow[d,twoheadrightarrow] & \pi/\Gamma_n \arrow[r] \arrow[d,twoheadrightarrow] & 1 \\
& 1 \arrow[r] & \pi_1(M) \arrow[r,equal] \arrow[d] & \pi_1(M) \arrow[r] \arrow[d] & 1 \\
& & 1 & 1 & \\
\end{tikzcd}
\caption{The diagram of groups in the proof of Lemma~\ref{lem:H2basic}.}
\label{fig:groups}
\end{figure}
The two longest columns induce compatible actions of $\pi_1(M)$ on the homology of $\Gamma$ and $\Gamma/\Gamma_n$ given by conjugation by $\pi$ and $\pi/\Gamma_n$, respectively.
Thus, homomorphisms on homology which are induced by the top row of the diagram are, in fact, $\Z[\pi_1(M)]$-module homomorphisms.

The desired isomorphism arises from the Stallings exact sequence (see \cite{Brown} Section VII.6) applied to the top row of Figure~\ref{fig:groups}, which yields \[H_2(\Gamma)\to H_2(\Gamma/\Gamma_n)\to H_1(\Gamma_n)_{\Gamma/\Gamma_n}\to H_1(\Gamma)\xrightarrow{\cong} H_1(\Gamma/\Gamma_n).\]
By the above discussion, this is an exact sequence of $\Z[\pi_1(M)]$-modules.
To see that we obtain the claimed isomorphism, note that $H_2(\wt{E_K})\twoheadrightarrow H_2(\Gamma)$ is surjective by Hopf's Theorem (see \cite{Brown} Section II.5), where $\wt{E_K}$ is the $\pi_1(M)$-cover of $E_K$.
However, as seen in the proof of Proposition~\ref{prop:finfty}, $H_2(\wt{E_K})=H_2(E_K;\Z[\pi_1(M)])=0$.
Thus, $H_2(\Gamma/\Gamma_n)\to H_1(\Gamma_n)_{\Gamma/\Gamma_n}$ is an isomorphism.
Furthermore, $H_1(\Gamma_n)_{\Gamma/\Gamma_n}$ is canonically isomorphic to $\Gamma_n/[\Gamma,\Gamma_n]=\Gamma_n/\Gamma_{n+1}$ (again see \cite{Brown} Section II.5).
This canonical identification is an isomorphism of $\Z[\pi_1(M)]$-modules as the $\pi_1(M)$ action is given by conjugation in each case.
\end{proof}

We now endow $H_2(\Gamma/\Gamma_n)$ with a colimit structure similar to the colimit structures on $H_3(\Gamma/\Gamma_n)$ and $C_n$ seen in Sections~\ref{subsec:zpi:colimH} and~\ref{subsec:zpi:colimC}.
For a set $\mathbf{g}$ of representatives of distinct cosets of $\langle[K]\rangle$, define $B_\mathbf{g}=B^n_\mathbf{g}=H_2(F(\mu_\mathbf{g})/F(\mu_\mathbf{g})_n)$, where we again fix a basing for each meridian $\mu_g$.
As in Sections~\ref{subsec:zpi:colimH} and~\ref{subsec:zpi:colimC}, we describe a colimit structure, this time on $H_2(\Gamma/\Gamma_n)=\Gamma_n/\Gamma_{n+1}$ in terms of the $B_\mathbf{g}$.

\begin{proposition} \label{prop:colimB}
The directed system $\left(\{B_\mathbf{g}\},\{\psi_{\mathbf{g},\mathbf{h}}\}\right)$ consists entirely of split inclusions and has colimit \[H_2(\Gamma/\Gamma_n)=\colim_\mathbf{g}B_\mathbf{g}.\]
\end{proposition}

\begin{proof}
This proof is entirely similar to the proof of Proposition~\ref{prop:colimH}.
\end{proof}

Note that $B_\mathbf{g}$ is canonically isomorphic to $F(\mu_\mathbf{g})_n/F(\mu_\mathbf{g})_{n+1}$ by the Stallings exact sequence.
For each $\mathbf{g}$, we have the following commutative square.
\begin{center}
\begin{tikzcd}
B_\mathbf{g} \arrow[r,"\cong"] \arrow[d,hookrightarrow] & F(\mu_\mathbf{g})_n/F(\mu_\mathbf{g})_{n+1} \arrow[d,hookrightarrow] \\
H_2(\Gamma/\Gamma_n) \arrow[r,"\cong"] & \Gamma_n/\Gamma_{n+1}
\end{tikzcd}
\end{center}
As seen in the proof of Lemma~\ref{lem:H2basic}, the isomorphism $H_2(\Gamma/\Gamma_n)\xrightarrow{\cong}\Gamma_n/\Gamma_{n+1}$ also arises from the Stallings exact sequence.
Thus, an equivalent interpretation of the colimit structure on $H_2(\Gamma/\Gamma_n)$ from Proposition~\ref{prop:colimB} is \[\Gamma_n/\Gamma_{n+1}=\colim_\mathbf{g}F(\mu_\mathbf{g})_n/F(\mu_\mathbf{g})_{n+1}.\]

One consequence of this reinterpretation is that the basings for individual meridians $\mu_g$ are irrelevant in the study of $H_2(\Gamma/\Gamma_n)$.
Any $n$-fold commutator in $\Gamma_n$ remains invariant, modulo $\Gamma_{n+1}$, under conjugation of individual generators by elements of $\Gamma$.
Thus, the action of $\pi_1(M)$ on $H_2(\Gamma/\Gamma_n)=\Gamma_n/\Gamma_{n+1}$ induces \emph{relabeling isomorphisms} $B_\mathbf{h}\xrightarrow{\cong}B_{g\cdot\mathbf{h}}$ as seen in Sections~\ref{subsec:zpi:bases}-\ref{subsec:zpi:action}.
Recall that this was not the case for $H_3(\Gamma/\Gamma_n)$, which partially motivated our study of the groups $C_n$.

We will often adopt the above reinterpretation of $H_2(\Gamma/\Gamma_n)$ as it allows us to apply the theory of \emph{basic commutators} to better understand $H_2(\Gamma/\Gamma_n)$.
Basic commutators arise in combinatorial group theory as bases for successive quotients of lower central subgroups, as we now discuss \cite{Hall, MKS}. \vspace{0.5em}

\noindent \textbf{Basic commutators.}
Let $F=\langle x_1,\dots x_k\rangle$ be the free group on $k$ letters.
M. Hall \cite{Hall} developed the \emph{commutator collection process} which, in particular, provides a totally ordered distinguished collection of commutators in $x_1,\dots,x_k$ called \emph{basic commutators}, each with a \emph{weight} $w\in\N$, such that the basic commutators of weight $w$ are a basis for the free abelian group $F_w/F_{w+1}$.
This group has rank \[N_w(k)=\frac{1}{w}\sum_{d|w}\phi(d)(k^{w/d}),\] where $\phi$ is the M\"{o}bius function \cite{MKS} (see also Section~\ref{subsec:zpi:bases}).

We define the basic commutators inductively by weight.
The total ordering is chosen so that if $b$ and $b'$ are basic of weights $w(b)$ and $w(b')$, respectively, with $w(b)<w(b')$, then $b<b'$.
Within each fixed weight $w$, we fix some ordering.
The basic commutators of weight 1 are $x_1,\dots,x_k$, and we choose the ordering $x_1<\cdots<x_k$.
Now suppose we have defined basic commutators of weights $\leq w$, totally ordered by weight and some ordering within each weight.
The basic commutators of weight $w+1$ are commutators $[a,b]$ such that (1) $a$ and $b$ are basic with $a<b$, (2) if $b=[c,d]$, then $a\geq c$, and (3) $w(a)+w(b)=w+1$.
Again, we choose some ordering on the basic commutators of weight $w+1$. \vspace{0.5em}

We now apply the interpretation of $H_2(\Gamma/\Gamma_n)=\Gamma_n/\Gamma_{n+1}$ as the colimit of the $F(\mu_\mathbf{g})_n/F(\mu_\mathbf{g})_{n+1}$ in order to discuss initial elements and to build a \emph{colimit-compatible basis} for $H_2(\Gamma/\Gamma_n)$, analogous to our work in Sections~\ref{subsec:zpi:colimC} and~\ref{subsec:zpi:bases}.
From now on, we will write $B_\mathbf{g}$ to mean both $H_2(F(\mu_\mathbf{g})/F(\mu_\mathbf{g})_n)$ and $F(\mu_\mathbf{g})_n/F(\mu_\mathbf{g})_{n+1}$.
Each $B_\mathbf{g}$ admits a basis of basic commutators, so our eventual goal is to construct a colimit-compatible basis for $H_2(\Gamma/\Gamma_n)$ using basic commutators.

We begin by reinterpreting the split inclusions $B_\mathbf{g}\hookrightarrow B_\mathbf{h}$ for $\mathbf{g}\subset\mathbf{h}$ from Proposition~\ref{prop:colimB}.
Consider a total ordering on the cosets of $\langle[K]\rangle$, hence a total ordering on the elements of $\mathbf{h}$ which restricts to a total ordering on the elements of $\mathbf{g}$.
The groups $B_\mathbf{g}$ and $B_\mathbf{h}$ are generated by basic commutators of length $n+1$ in the meridians $\mu_\mathbf{g}$ and $\mu_\mathbf{h}$, respectively.
Because the total order on $\mathbf{h}$ restricts to the total order on $\mathbf{g}$, any basic commutator in $\mu_\mathbf{g}$ is automatically a basic commutator in $\mu_\mathbf{h}$.
Thus, the split inclusion $B_\mathbf{g}\hookrightarrow B_\mathbf{h}$ is determined by the inclusion of the basis of basic commutators for $B_\mathbf{g}$ into the basis of basic commutators for $B_\mathbf{h}$.

Next, we define \emph{initial elements} in a similar manner to Sections~\ref{subsec:zpi:colimH} and~\ref{subsec:zpi:colimC}.

\begin{definition} \label{def:initialB}
For $c_\mathbf{g}\mapsto c$ under $B_\mathbf{g}\hookrightarrow H_2(\Gamma/\Gamma_n)$, we say $c_\mathbf{g}$ is \emph{initial} for $c$ if there does not exist $c_\mathbf{h}\in C_\mathbf{h}$ with $\mathbf{h}\subsetneq\mathbf{g}$ and $c_\mathbf{h}\mapsto c$.
\end{definition}

Analogous to Propositions~\ref{prop:initialH} and~\ref{prop:initialC}, we prove the following.

\begin{proposition} 
\label{prop:initialB}
For each element $c\in H_2(\Gamma/\Gamma_n)$, there exists a unique initial element.
\end{proposition}

\begin{proof}
As in the proofs of Propositions~\ref{prop:initialH} and~\ref{prop:initialC}, it suffices to prove that \[\im(B_\mathbf{g}\hookrightarrow B_{\mathbf{g}\cup\mathbf{h}})\cap\im(B_\mathbf{h}\hookrightarrow B_{\mathbf{g}\cup\mathbf{h}})=\im(B_{\mathbf{g}\cap\mathbf{h}}\hookrightarrow B_{\mathbf{g}\cup\mathbf{h}})\] for all $\mathbf{g},\mathbf{h}$.
As before, the reverse inclusion $\supseteq$ is immediate, and we prove the forward inclusion $\subseteq$ using our reinterpretation of the split inclusions in the directed system $\{B_\mathbf{g}\}$.

Any element in the intersection of the images of $B_\mathbf{g}$ and $B_\mathbf{h}$ in $B_{\mathbf{g}\cup\mathbf{h}}$ can be expressed as a linear combination of basic commutators in each of $\mu_\mathbf{g}$ and $\mu_\mathbf{h}$.
This is only possible if it can be expressed as a linear combination of basic commutators in $\mu_{\mathbf{g}\cap\mathbf{h}}$ and therefore lies in the image of $B_{\mathbf{g}\cap\mathbf{h}}$.
\end{proof}

\begin{corollary}\label{cor:colimB}
The abelian group $H_2(\Gamma/\Gamma_n)$ is isomorphic to $\Z^\infty$, that is, $H_2(\Gamma/\Gamma_n)$ is a countably infinitely generated free abelian group.
\end{corollary}

\begin{proof}
The proof is entirely similar to the proof of Corollary~\ref{cor:colimC}. In this case, we use that $B_\mathbf{g}$ is free abelian of finite rank for each $\mathbf{g}$.
\end{proof}

Similar to our analysis of $C_n$ in Section~\ref{subsec:zpi:bases}, we now construct a \emph{colimit-compatible basis} for $H_2(\Gamma/\Gamma_n)$.
In contrast to Section~\ref{subsec:zpi:bases}, in this instance we will require an \emph{ordered} basis, as basic commutators of a given weight rely on the ordering chosen for basic commutators  of lower weights.

\begin{definition} \label{def:basisB}
Fix a total ordering on the cosets of $\langle[K]\rangle$, thus fixing an ordering on each set of coset representatives. We will say that an ordered basis $\beta$ for $H_2(\Gamma/\Gamma_n)$ is \emph{colimit-compatible} if the following hold.
\begin{itemize}
\item For each set of coset representatives $\mathbf{g}$, the basis $\beta$ is an order-preserving extension of the image of an ordered basis $\beta_\mathbf{g}$ for $B_\mathbf{g}$ under the split inclusion $B_\mathbf{g}\hookrightarrow H_2(\Gamma/\Gamma_n)$.
\item Every order-preserving relabeling isomorphism $\rho_{\mathbf{g},\mathbf{h}}:B_\mathbf{g}\xrightarrow{\cong} B_\mathbf{h}$ sends the ordered basis $\beta_\mathbf{g}$ to the ordered basis $\beta_\mathbf{h}$.
\end{itemize}
\end{definition}

\begin{proposition} \label{prop:basisB}
Given any total ordering on the cosets of $\langle[K]\rangle$, there exists a colimit-compatible ordered basis of basic commutators for $H_2(\Gamma/\Gamma_n)$.
\end{proposition}

\begin{proof}
Our strategy is similar the proof of Proposition~\ref{prop:basis}, except now we inductively impose an order on the basis we construct.
Basis elements will be ordered first by $|\mathbf{g}|$, then by the lexicographic ordering on sets $\mathbf{g}$ of the same size induced by the ordering on cosets, and then by some fixed ordering which we ensure is compatible with the action of $\pi_1(M)$.

Because basic commutators of a given weight rely on the ordering chosen for basic commutators of lower weights, we apply nested induction arguments, using induction to find a colimit-compatible ordered basis of basic commutators for $H_2(\Gamma/\Gamma_k)=\Gamma_k/\Gamma_{k+1}$ for each $2\leq k\leq n$.
We begin with the basic commutators of weight 1, which are the generators $\mu_g$, ordered by the ordering on cosets.
Just as with the colimit structure on $C_n$ from Section~\ref{subsec:zpi:colimC}, the colimit structure on $H_2(\Gamma/\Gamma_n)$ stabilizes, this time at $\mathbf{g}$ with $|\mathbf{g}|=n$, as any basic commutator of length $n$ can involve at most $n$ generators.
In other words, \[H_2(\Gamma/\Gamma_n)=\colim_{|\mathbf{g}|\leq n}B_\mathbf{g}.\]

We begin with the base case $n=2$.
All $B^2_\mathbf{g}$ with $|\mathbf{g}|=2$ are generated by a single basic commutator.
If $\mathbf{g}=\{g_1,g_2\}$ with $g_1<g_2$, then $B^2_\mathbf{g}$ is generated by the basic commutator $[\mu_{g_1},\mu_{g_2}]$.
For each such $\mathbf{g}$, this element is determined by the ordering on cosets.
Any order-preserving relabeling isomorphism respects the single-element bases consisting of these commutators.
As discussed above, when $n=2$ the colimit stabilizes at $|\mathbf{g}|=2$.
We can order the resulting basis for $H_2(\Gamma/\Gamma_2)$ using the lexicographic order induced by the ordering on cosets.
This proves the base case $n=2$.

Next, suppose we have constructed an ordered colimit-compatible basis of basic commutators for $H_2(\Gamma/\Gamma_k)$.
We will construct such a basis for $H_2(\Gamma/\Gamma_{k+1})$ by inducting on $|\mathbf{g}|$, mirroring our argument from the proof of Proposition~\ref{prop:basis}.
Consider the base case $|\mathbf{g}|=2$ for this nested induction.
Fix some $\mathbf{g_2}$ with $|\mathbf{g_2}|=2$.
Using the colimit-compatible ordered basis for $H_2(\Gamma/\Gamma_k)$, which includes the image of a basis of basic commutators of length $k$ for $B^k_\mathbf{g_2}$, define a basis of basic commutators of length $k+1$ for $B^{k+1}_\mathbf{g_2}$.
Choose any ordering for this basis, and use order-preserving relabeling isomorphisms to define ordered bases for each $B^{k+1}_\mathbf{h}$ with $|\mathbf{h}|=2$.
Proposition~\ref{prop:initialB} implies the images of $B^{k+1}_\mathbf{g}$ and $B^{k+1}_\mathbf{h}$ intersect trivially for $\mathbf{g}\neq\mathbf{h}$ with $|\mathbf{g}|=|\mathbf{h}|=2$, so we can define these bases without conflict.
Order all basis elements using first the lexicographic ordering on the $\mathbf{g}$ and then the ordering on the basis for each $B^{k+1}_\mathbf{g}$.
This completes the base case, as we have produced an ordered basis of basic commutators for $\bigoplus_{|\mathbf{g}|=2}\im(B^{k+1}_\mathbf{g}\hookrightarrow H_2(\Gamma/\Gamma_{k+1}))$ which is compatible with order-preserving relabeling isomorphisms and extends the (empty) ordered basis for $\bigoplus_{|\mathbf{g}|=1}\im(B^{k+1}_\mathbf{g}\hookrightarrow H_2(\Gamma/\Gamma_{k+1}))$.

Now suppose we have chosen an ordered basis of basic commutators for $\plus_{|\mathbf{g}|=j}\im(B^{k+1}_\mathbf{g}\hookrightarrow H_2(\Gamma/\Gamma_{k+1}))$ which extends the bases of basic commutators for $\plus_{|\mathbf{g}|=i}\im(B^{k+1}_\mathbf{g}\hookrightarrow H_2(\Gamma/\Gamma_{k+1}))$, $i<j$, and is compatible with order-preserving relabeling isomorphisms.
We extend to a basis of basic commutators for $\plus_{|\mathbf{g}|=j+1}\im(B^{k+1}_\mathbf{g}\hookrightarrow H_2(\Gamma/\Gamma_{k+1}))$.
Choose some $\mathbf{g_{j+1}}$ with $|\mathbf{g_{j+1}}|=j+1$, and use basic commutators of length $k$ given by a basis for $B^k_\mathbf{g_{j+1}}$ to construct a basis for $B^{k+1}_\mathbf{g_{j+1}}$.
Because the split inclusions $B_\mathbf{g}\hookrightarrow B_\mathbf{h}$, $\mathbf{g}\subset\mathbf{h}$, are determined by inclusions of basic commutators, the basis for $B^{k+1}_\mathbf{g_{j+1}}$ is compatible with the basis already constructed for $B_+=\plus_{\mathbf{h}\subsetneq\mathbf{g_{j+1}}}\im(B^{k+1}_\mathbf{h}\hookrightarrow B^{k+1}_{\mathbf{g_{j+1}}})$.
In other words, the basis for $B^{k+1}_\mathbf{g_{j+1}}$ given by basic commutators is an extension of ordered bases already constructed.

Similarly to the proof of Proposition~\ref{prop:basis}, $B_+$ forms a direct summand of $B^{k+1}_\mathbf{g_{j+1}}$.
We extend the ordered basis for $B_+$ to a basis for $B^{k+1}_\mathbf{g_{j+1}}$, ordering the additional basis elements in any way but requiring that they are greater than all basis elements of $B_+$.
As in the proof of Proposition~\ref{prop:basis}, these additional basis elements are initial in $B^{k+1}_\mathbf{g_{j+1}}$.
Now use order-preserving relabeling isomorphisms to define ordered bases for all $B^{k+1}_\mathbf{h}$, $|\mathbf{h}|=j+1$.
Note that these isomorphisms send basic commutators to basic commutators.
Like the proof of Proposition~\ref{prop:basis}, we may do this without conflict by the induction hypothesis and because initial elements are sent to initial elements under relabeling isomorphisms.
Order all resulting initial basis elements using the lexicographic order, followed by the ordering induced by the ordering on initial basis elements of $B^{k+1}_\mathbf{g_{j+1}}$.
Order all basis elements of $\plus_{|\mathbf{g}|={j+1}}\im(B^{k+1}_\mathbf{g}\hookrightarrow H_2(\Gamma/\Gamma_{k+1}))$ by requiring that all basis elements which are images of initial basis elements in some $B^{k+1}_\mathbf{h}$, $|\mathbf{h}|=j+1$, are greater than all basis elements which are initial in some $B^{k+1}_\mathbf{h}$, $|\mathbf{h}|<j+1$.
This achieves the desired ordering, first by the size of the $\mathbf{g}$, then by the lexicographic ordering on elements of the $\mathbf{g}$, and finally by the ordering fixed for each $B^{k+1}_\mathbf{g_{i}}$ and induced on all $B^{k+1}_\mathbf{h}$ with $|\mathbf{h}|=i$.

Thus, we have extended the ordered basis for $\plus_{|\mathbf{g}|=j}\im(B^{k+1}_\mathbf{g}\hookrightarrow H_2(\Gamma/\Gamma_{k+1}))$ to an ordered basis of basic commutators for $\plus_{|\mathbf{g}|={j+1}}\im(B^{k+1}_\mathbf{g}\hookrightarrow H_2(\Gamma/\Gamma_{k+1}))$ which is preserved by order-preserving relabeling isomorphisms.
Continuing inductively, we achieve a colimit-compatible ordered basis for $H_2(\Gamma/\Gamma_{k+1})$ which completes the initial induction.
\end{proof}

To conclude our analysis of $H_2(\Gamma/\Gamma_n)$, we state results analogous to Proposition~\ref{prop:permmodule} and Corollary~\ref{cor:freeC}, culminating in conditions that guarantee $H_2(\Gamma/\Gamma_n)$ is a free $\Z[\pi_1(M)]$-module.
We omit the proofs as they are entirely similar to those seen in Section~\ref{subsec:zpi:action}.
Recall that, unlike the case of $H_3(\Gamma/\Gamma_n)$, the $\pi_1(M)$-action on $H_2(\Gamma/\Gamma_n)$ induces relabeling isomorphisms $B_\mathbf{h}\xrightarrow{\cong}B_{g\cdot\mathbf{h}}$, so each $g\in\pi_1(M)$ permutes the colimit-compatible basis for $H_2(\Gamma/\Gamma_n)$.

\begin{proposition} \label{prop:permmoduleB}
If the left cosets of $\langle[K]\rangle\leq\pi_1(M)$ admit a total ordering which is preserved by the $\pi_1(M)$-action by left multiplication, then $H_2(\Gamma/\Gamma_n)$ is a permutation module.
\end{proposition}

\begin{corollary} \label{cor:freeH2}
Suppose conditions (1)-(3) of Theorem~\ref{thm:infinite} are satisfied:
\begin{enumerate}[label=(\arabic*)]
\item $[K]$ is primitive and infinite order in $H_1(M)$.
\item The centralizer of any nontrivial power of $[K]\in\pi_1(M)$ is cyclic.
\item The left cosets of $\langle[K]\rangle\leq\pi_1(M)$ admit a total ordering which is invariant under the $\pi_1(M)$-action by left multiplication.
\end{enumerate}
Then $H_2(\Gamma/\Gamma_n)$ is a free $\Z[\pi_1(M)]$-module.
\end{corollary}


\section{Alternate methods} \label{sec:alt}

While in principle the methods developed in this paper could be used to confirm the \hyperref[conj:ac]{Almost-Concordance Conjecture} for any nontrivial homotopy class $x$ in any aspherical $M$, the intricacies explored in Section~\ref{sec:zpi} suggest some extensive casework might be needed to realize this result.
To conclude this paper, we suggest a number of alternate methods and further lines of study for proving the \hyperref[conj:ac]{Almost-Concordance Conjecture} in remaining open cases.


\subsection*{Understanding the differential $d^2_{2,2}$}

The proof of Theorem~\ref{thm:infinite} reveals that understanding the $\Z[\pi_1(M)]$-module structure on $H_*(\Gamma/\Gamma_n)$ is beneficial for showing that Milnor's invariants of links $L\subset S^3$ ``survive" as almost-concordance invariants of knots in $M\neq S^3$.
Our analysis in Section~\ref{sec:zpi} focuses on showing $H_2(\Gamma/\Gamma_n)$ and $C_n$ are free $\Z[\pi_1(M)]$-modules under the assumptions of Theorem~\ref{thm:infinite}.
When $H_2(\Gamma/\Gamma_n)$ is free, $H_2(M;H_2(\Gamma/\Gamma_n))=0$, and the image of the differential $d^2_{2,2}$ in the Lyndon-Hochschild-Serre spectral sequence seen in Section~\ref{subsec:main:large} does not affect the inclusion-induced homomorphism $H_3(\Gamma/\Gamma_n)\to H_3(\pi/\Gamma_n)$.

Even if the module $H_2(\Gamma/\Gamma_n)$ is not free, it may be the case that $d^2_{2,2}=0$ or has small image, so that one may still use Theorem~\ref{thm:family} to construct knots representing infinitely many almost-concordance classes within a fixed homotopy class.
Further analysis is needed to determine when this is the case.


\subsection*{Realization theorems}

Some alternate methods of proving the \hyperref[conj:ac]{Almost-Concordance Conjecture} using the invariants from Section~\ref{sec:Milnor} rely on realization theorems.
Theorem C of \cite{Stees24} provides necessary and sufficient conditions under which a \emph{homotopy class} $f\in[M,X_n(K)]_0$ is realized as the lower central \emph{homotopy} invariant $h_n(K')$ of some knot $K'$ relative to the fixed knot $K\subset M$.
When $M$ is aspherical, Corollary 6.2 of \cite{Stees24} implies the technical third condition of the realization theorem is always satisfied, simplifying the analysis of the set of realizable homotopy classes.
In this work, we chose to pursue the \hyperref[conj:ac]{Almost-Concordance Conjecture} using instead the lower central \emph{homology} invariants $\theta_n$ because it seems more straightforward to prove results like Theorems~\ref{thm:family} and~\ref{thm:infinite} involving large families of 3-manifolds and homotopy classes.
Nevertheless, the realization theorem for lower central homotopy invariants is a potentially useful tool toward confirming the conjecture.

As discussed in \cite{Stees24}, a realization theorem for the $\Z[\pi_1(M)]$-lower central \emph{homology} invariants may rely on extending a 3-dimensional homology surgery result of Turaev \cite{Turaev84}.
Another avenue toward verifying the \hyperref[conj:ac]{Almost-Concordance Conjecture} would be to extend Turaev's result, prove such a realization theorem, and then execute algebraic arguments to show there are infinitely many realizable classes.
Theorem E of \cite{Stees24} is a realization theorem for $\Z$-lower central homology invariants; however, these invariants, which concern the classical lower central series of the knot group $\pi$, may not be strong enough to detect nontrivial almost-concordance classes.
For instance, if the homotopy class $x$ in question has infinite order in $\pi_1(M)$, it may be the case that the meridian of any fixed knot $K$ representing $x$ is nullhomologous in $E_K$, in which case it belongs to $\pi_2=[\pi,\pi]$.
It is unclear whether this prevents the $\Z$-lower central homology invariants from providing almost-concordance information, but we note that the results in this paper rely heavily on the nontriviality of the meridian of $K$ in $\Gamma/\Gamma_2$.
Thus, the full power of the $\Z[\pi_1(M)]$-lower central quotients may be required, necessitating the use of the $\Z[\pi_1(M)]$-lower central homology invariants.


\subsection*{Schneiderman's invariant}

Schneiderman developed a \emph{relative self-linking invariant} in \cite{Schneiderman} which can be used to prove the \hyperref[conj:ac]{Almost-Concordance Conjecture} for the trivial homotopy class in any 3-manifold.
We expect Schneiderman's invariant is capable of detecting infinitely many almost-concordance classes in many cases; however, the indeterminacy subgroup $\Phi$ involved in the definition of the invariant is not well-understood.
We anticipate that the lower central homotopy invariant $h_1$ relative to a fixed knot $K$ is closely related to Schneiderman's self-linking invariant relative to $K$; this is the subject of future work.


\bibliography{references.bib}


\end{document}